\documentclass[letterpaper,12pt,titlepage,oneside,final]{article}

\newcommand{\href}[1]{#1} 

\usepackage{ifthen}
\newboolean{PrintVersion}
\setboolean{PrintVersion}{false} 

\usepackage{amsmath,amssymb,amstext} 
\usepackage[pdftex]{graphicx} 

\usepackage[pdftex,pagebackref=false]{hyperref} 
\hypersetup{
    plainpages=false,       
    unicode=false,          
    pdftoolbar=true,        
    pdfmenubar=true,        
    pdffitwindow=false,     
    pdfstartview={FitH},    
    pdftitle={uWaterloo\ LaTeX\ Thesis\ Template},    
    pdfauthor={Jonathan Herman},    
    pdfnewwindow=true,      
    colorlinks=true,        
    linkcolor=blue,         
    citecolor=green,        
    filecolor=magenta,      
    urlcolor=cyan           
}
\ifthenelse{\boolean{PrintVersion}}{   
\hypersetup{	
    citecolor=black,%
    filecolor=black,%
    linkcolor=black,%
    urlcolor=black}
}

\let\origdoublepage\cleardoublepage
\newcommand{\clearemptydoublepage}{%
  \clearpage{\pagestyle{empty}\origdoublepage}}
\let\cleardoublepage\clearemptydoublepage

\usepackage{setspace}
\usepackage[utf8]{inputenc}
\usepackage[margin=1in]{geometry}
\usepackage{array}
\usepackage{ stmaryrd,amsxtra,amstext,amscd,amsfonts,fancyhdr}
\usepackage{ amssymb }
\usepackage{amsmath}
\usepackage{amsthm}
\usepackage{enumerate}

\usepackage{tikz}
\usetikzlibrary{calc}

\newcommand\del[1]{}

\allowdisplaybreaks

\makeatletter
\@namedef{subjclassname@2010}{%
  \textup{2010} Mathematics Subject Classification}
\makeatother

\newtheorem{theorem}{Theorem}[section]
\newtheorem{corollary}[theorem]{Corollary}
\newtheorem{proposition}[theorem]{Proposition}

\theoremstyle{definition}
\newtheorem{definition}[theorem]{Definition}
\newtheorem{lemma}[theorem]{Lemma}
\newtheorem{question}[theorem]{Question}
\newtheorem{example}[theorem]{Example}
\newtheorem{remark}[theorem]{Remark}

\numberwithin{equation}{section}

\newcommand\N{{\mathbb{N}}}
\newcommand\R{{\mathbb{R}}}
\newcommand\C{{\mathbb{C}}}

\newcommand\pd{{\partial}}
\newcommand\Z{{\mathbb{Z}}}

\renewcommand\L{{\mathcal{L}}}
\newcommand{\hk}{\mathbin{\! \hbox{\vrule height0.3pt width5pt depth 0.2pt \vrule height5pt width0.4pt depth 0.2pt}}}

\renewcommand\o{\overline }

\newcommand\Rho{\mathcal{P}}
\newcommand\g{\mathfrak{g}}
\newcommand\X{\mathfrak{X}}

\begin{document}

\title{N$\ddot{\text{o}}$ether's Theorem}
\author{Jonathan Herman}
\begin{titlepage}
\begin{center}
\setlength{\parindent}{0pt}
\setlength{\parskip}{17pt}

\vspace*{10pt}

{\Large \bf{Weak Moment Maps in Multisymplectic Geometry}\rm\\ 
\vspace{0.6cm}

\Large{by}

\Large{Jonathan Herman}}
\del{\\

{\it Department of Pure Mathematics, University of Waterloo}
\\

\tt{j3herman@uwaterloo.ca}}
\vspace*{10pt}

{\small A thesis\\ 
\vspace{0.1cm}
presented to the University of Waterloo\\
\vspace{0.1cm}in fulfilment of the\\
\vspace{0.1cm}thesis requirement for the degree of\\
\vspace{0.1cm}Doctor of Philosophy\\
\vspace{0.1cm}in\\
\vspace{0.1cm}Pure Mathematics

\vspace*{\fill}

Waterloo, Ontario, Canada, 2018
 
\copyright \  Jonathan Herman 2018

}
\end{center}

\end{titlepage}

\cleardoublepage

\section*{Examining Committee Membership}

 \pagenumbering{roman}
 \setcounter{page}{2}
 The following served on the Examining Committee for this thesis. The decision of the Examining committee is by majority vote.
 \\
 
\noindent \textbf{External Examiner:} Dr. Andrew Swann.
 \\
 
\noindent \textbf{Supervisors:}  Dr. Spiro Karigiannis and Dr. Shengda Hu.
\\

\noindent \textbf{Internal Members:} Dr. Ruxandra Moraru and Dr. David McKinnon.
\\

\noindent \textbf{Internal-External Member:} Dr. Raymond McLenaghan.

 \newpage
  \pagenumbering{roman}
 \setcounter{page}{3}
\vspace*{0.5in}
\noindent
{   \setlength{\parindent}{0pt}
   \setlength{\parskip}{24pt}
   \setlength{\textwidth}{7in}

   {\sffamily\bfseries \index{copyright!author's declaration}
   AUTHOR'S DECLARATION}

   I hereby declare that I am the sole author of this thesis.  This is a true
   copy of the thesis, including any required final revisions, as accepted by
   my examiners.

   I understand that my thesis may be made electronically available to the
   public. 
   
   \vspace*{1.0in}

}


\begin{abstract}
\thispagestyle{plain} 

 \pagenumbering{roman}
 \setcounter{page}{4}
 
\indent We introduce the notion of a weak (homotopy) moment map associated to a Lie group action on a multisymplectic manifold. We show that the existence/uniqueness theory governing these maps is a direct generalization from symplectic geometry.

We use weak moment maps to extend Noether's theorem from Hamiltonian mechanics by exhibiting a correspondence between multisymplectic conserved quantities and continuous symmetries on a multi-Hamiltonian system. We find that a weak moment map interacts with this correspondence in a way analogous to the moment map in symplectic geometry. 

We define a multisymplectic analog of the classical momentum and position functions on the phase space of a physical system by introducing momentum and position forms. We show that these differential forms satisfy generalized Poisson bracket relations extending the classical bracket relations from Hamiltonian mechanics. We also apply our theory to derive some identities on manifolds with a torsion-free $G_2$ structure.

\del{
Our symmetries and conserved quantities will constitute an $L_\infty$-algebra and we demonstrate how a moment map restricts to an $L_\infty$-morphism between the two. Under certain quotients, we will show that there is actually an isomorphism of graded Lie algebras between the symmetries and conserved quantities.
}

\end{abstract}


\onehalfspacing
\cleardoublepage
\section*{Acknowledgements} \pagenumbering{roman}
 \setcounter{page}{5}The first people I need to acknowledge are my two supervisors, Spiro and Shengda.

Spiro has been my supervisor ever since I was a masters student, coaching me all the way through my PhD. I learned so much from him, both academically and professionally. I am definitely fortunate to have had such an involved and down to earth supervisor. 

Shengda became my co-supervisor when my research transitioned into symplectic geometry. He was also an invaluable support in helping me understand what research is and in formulating the content that comprised my 2 papers.

I would also like to thank my defence committee for providing valuable input and making my defence a day I will always remember.

Next, I need to acknowledge my family. My family experienced all of the ups and downs of the PhD with me. They always had my back, and I could not have gotten through the program without their love and support. 

Lastly, I need to acknowledge my partner Sarah. Ever since we reconnected a few years ago, she has only had a positive and healthy impact on my life, and has made me so happy. I will always appreciate her.

\del{First and foremost, a very sincere thank you goes to my supervisor Dr. Spiro Karigiannis. I am extremely grateful for his patience and the tremendous amount of time he spent teaching and helping me this summer. I would also like to thank Dr. Shengda Hu, my second reader, for the very useful comments and corrections that he provided. I need to acknowledge two of my good friends; Janis Lazovksis and Cameron Williams. Janis is a LaTeX machine, and I am very appreciative of the time he spent helping me this summer. Cam was a great support this year, always there to help me work through any problem. I also would like to thank my family for their constant love and interest, it means very much to me. Last, but not least (I'd say least goes to Cameron Williams), a thank you goes to my cousin Matt Rappoport, for without our discussions some of the contents in this paper would not exist.}


\del{

\cleardoublepage
\vspace*{70pt}
\begin{center}
\itshape OPTIONAL DEDICATION CAN GO HERE.
\end{center}

}

\newpage
\tableofcontents

\newpage

\pagenumbering{arabic}
\setcounter{page}{1}

\section{Introduction}
\del{
Noether's theorem says that for every symmetry on a classical Hamiltonian system, there is a corresponding conserved quantity. That is, the total energy of a conservative physical system in $\R^3$ is conserved because Hamilton's equations are invariant under time translation. Similarily, conservation of linear and angular momentum are consequences of the invariance of Hamilton's equations under space translations and rotations. In this paper, we replace the symplectic form with an arbitrary non-degenerate form (i.e. work on a multisymplectic manifold) and show how Noether's theorem generalizes. 
}

 Multisymplectic geometry is the natural generalization of symplectic geometry in which the symplectic $2$-form is replaced by an `$n$-plectic' form, where $n$ is any positive integer. Manifolds that are K\"ahler, hyper-K\"ahler, $G_2$, and Spin$(7)$ are all naturally multisymplectic. In physics, $n$-plectic manifolds are used to describe $n$-dimensional covariant field theories (see e.g. \cite{covariant}) and the case $n=2$ is relevant to string theory (see e.g. \cite{string}). 
 
The main idea this thesis is concerned with is generalizing moment maps from symplectic geometry to multisymplectic geometry, and the corresponding applications. In particular, we introduce a weak (homotopy) moment map, defined in equation (\ref{equation 2}) below.

Recall that for a symplectic manifold $(M,\omega)$, a Lie algebra $\g$ is said to act symplectically if $\L_{V_\xi}\omega=0$, for all $\xi\in\g$, where $V_\xi$ is its infinitesimal generator. A symplectic group action is called Hamiltonian if one can find a moment map, that is, a map $f:\g\to C^\infty(M)$ satisfying \[df(\xi)=V_\xi\hk\omega,\] for all $\xi\in\g$.
\\

In multisymplectic geometry, $\omega$ is replaced by a closed, non-degenerate $(n+1)$-form, where $n\geq 1$. The pair $(M,\omega)$ is called an $n$-plectic manifold. A Lie algebra action is called multisymplectic if $\L_{V_\xi}\omega=0$ for each $\xi\in\g$. A generalization of moment maps from symplectic to multisymplectic geometry is given by a (homotopy) moment map. These maps are discussed in detail in \cite{questions}. A homotopy moment map is a collection of maps, $f_k:\Lambda^k\g\to \Omega^{n-k}(M)$, with $1\leq k \leq n+1$, satisfying \begin{equation}\label{equation1}df_k(p)=-f_{k-1}(\partial_k(p))+(-1)^{\frac{k(k+1)}{2}}V_p\hk\omega,\end{equation}for all $p\in\Lambda^k\g$, where $V_p$ is an infinitesimal generator (see Definition \ref{fhmm}) and $\partial_k$ is the $k$-th Lie algebra homology differential $\partial_k:\Lambda^k\g\to\Lambda^{k-1}\g$,  defined by
\[\partial_k:\Lambda^k\g\to\Lambda^{k-1}\g \ \ \ \ \ \ \ \ \xi_1\wedge\cdots\wedge\xi_k\mapsto\sum_{1\leq i<j\leq k}(-1)^{i+j}[\xi_i,\xi_k]\wedge\xi_1\wedge\cdots\wedge\widehat\xi_i\wedge\cdots\wedge\widehat\xi_j\wedge\cdots\wedge\xi_k,\] for $k\geq 1$ and $\xi_1,\cdots,\xi_k\in \g$.
A weak (homotopy) moment map is a collection of maps $f_k:\Rho_{\g,k}\to\Omega^{n-k}(M)$ satisfying  \begin{equation}\label{equation 2}df_k(p)=(-1)^{\frac{k(k+1)}{2}}V_p\hk\omega,\end{equation} for $p\in\Rho_{\g,k}$. Here $\Rho_{\g,k}$ is the Lie kernel, which is defined to be the kernel of $\partial_k$. We call the $k$-th component $f_k$ of a weak moment map a weak $k$-moment map.

We see that any collection of functions satisfying equation (\ref{equation1}) must also satisfy (\ref{equation 2}). That is, any homotopy moment map induces a weak homotopy moment map.

Also note that the $n$-th component of a homotopy moment map coincides with the multi-moment maps of Madsen and Swann introduced in \cite{ms} and \cite{MS}.\del{ in \cite{ms} and \cite{MS}, and were used to give a  multisymplectic version of Noether's theorem in \cite{me}. Homotopy moment maps also give rise to a large family of conserved quantities in multisymplectic systems as shown in \cite{me} and \cite{cq}. }

Of central importance in this thesis is the applications of these weak moment maps to multi-Hamiltonian systems. We define a multi-Hamiltonian system to be a triple $(M,\omega,H)$ where $(M,\omega)$ is $n$-plectic and $H$ is a `Hamiltonian' $(n-1)$-form. This means that there exists a vector field $X_H\in\Gamma(TM)$ such that $X_H\hk\omega=-dH$. In analogy to Hamiltonian mechanics, the integral curves of the vector field are to be interpreted as the motions in the relevant physical system. 

One of our main results is a generalization of Noether's theorem to multisymplectic geometry. In order to state our version of Noether's theorem, we first develop a multisymplectic `Poisson' bracket $\{\cdot,\cdot\}$, which was introduced in \cite{poisson}, and a notion of multisymplectic conserved quantity and symmetry.

Multisymplectic conserved quantities were introduced in \cite{cq}. They defined three types: A differential form $\alpha$ is called a local, global, or strict conserved quantity if it is Hamiltonian and $\L_{X_H}\alpha$ is closed, exact, or zero respectively.  Here Hamiltonian means that there exists a multivector field $X_\alpha$ such that $-d\alpha=X_\alpha\hk\omega$. By adding in this requirement, we are then able to study how the extended `Poisson' bracket  interacts with the conserved quantities. We find that, analogous to the case of Hamiltonian mechanics, the Poisson bracket of two conserved quantities is always strictly conserved.
That is, \begin{proposition} Let $\alpha$ and $\beta$ be two (local, global or strict) conserved quantities on a multi-Hamiltonian system $(M,\omega, H)$. Then $\{\alpha,\beta\}$ is strictly conserved, meaning $\L_{X_H}\{\alpha,\beta\}=0$.
\end{proposition} From this proposition we show that the conserved quantities, modulo closed forms, constitute a graded Lie algebra. We also show that when restricted to a certain subspace, namely the Lie $n$-algebra of observables (see Definition \ref{Lie n observables} or \cite{rogers}), the conserved quantities form an $L_\infty$-algebra.
\vspace{0.3cm}

Similarly, we find that our continuous symmetries also generate a graded Lie algebra. As an extension from Hamiltonian mechanics, we define a symmetry to be a Hamiltonian multivector field with respect to which the Lie derivative of the Hamiltonian has a specific form. Just as for the conserved quantities, we have three types of continuous symmetry. Namely, a multivector field $X$ is a local, global, or strict symmetry on $(M,\omega,H)$ if $\L_X\omega=0$ and $\L_XH$ is closed, exact, or zero respectively. A generalization from Hamiltonian mechanics is \begin{proposition}Given any two (local, global, strict) continuous symmetries $X$ and $Y$, their Schouten bracket $[X,Y]$ is a continuous symmetry of the same type.\end{proposition} From this proposition we show that the continuous symmetries, modulo elements in the kernel of $\omega$, form a graded Lie algebra. 
\vspace{0.3cm}

Our first generalization of Noether's theorem says that there is a correspondence between these notions of symmetry and conserved quantity on a multisymplectic manifold.

\begin{theorem}
If $\alpha$ is a (local or global) conserved quantity, then every corresponding Hamiltonian multivector field $X_\alpha$ is a (local or global) continuous symmetry. Conversely, if $X$ is a (local or global) continuous symmetry, then every corresponding Hamiltonian form is a (local or global) conserved quantity.
\end{theorem}

As in symplectic geometry, this correspondence is not one-to-one. Indeed, for a Hamiltonian form, any two of its corresponding Hamiltonian multivector fields differ by an element in the kernel of $\omega$. Conversely, any two Hamiltonian forms corresponding to a Hamiltonian multivector field differ by a closed form:

Let  \[\Omega_{\text{Ham}}(M)=\{\alpha\in\Omega^\bullet(M); d\alpha= X\hk\omega \text{ for some } X\in\Gamma(\Lambda^\bullet(TM))\}\] denote the graded vector space of Hamiltonian forms, and let $\widetilde\Omega_{\text{Ham}}(M)$ denote the quotient of $\Omega_{\text{Ham}}(M)$ by closed forms. Similarily, we let \[\X_{\text{Ham}}(M)=\{X\in\Gamma(\Lambda^\bullet(TM)); X\hk \omega \text{ is exact}\}\]  denote the graded vector space of Hamiltonian multivector fields and $\widetilde\X_{\text{Ham}}(M)$ denote the quotient of $\X_{\text{Ham}}(M)$ by elements in the kernel of $\omega$. We slightly improve on the results of \cite{dropbox} and show that $\{\cdot,\cdot\}$ descends to a well defined graded Poisson bracket on $\widetilde\Omega_{\text{Ham}}(M)$. Then we show that

\begin{theorem}
There is a natural isomorphism of graded Lie algebras between $(\widetilde\X_{\text{Ham}}(M),[\cdot,\cdot])$ and $(\widetilde\Omega_{\text{Ham}}(M),\{\cdot,\cdot\})$.
\end{theorem}

As a consequence of this theorem, we then show that our symmetries and conserved quantities, after appropriate quotients, are in one-to-one correspondence. In particular, we let $\mathcal{C}_\text{loc}(X_H)$, $\mathcal{C}(X_H)$, $\mathcal{C}_\text{str}(X_H)$ denote the spaces of local, global, and strict conserved quantities respectively, and  $\widetilde{\mathcal{C}}_\text{loc}(X_H),\widetilde{\mathcal{C}}(X_H)$, and $\widetilde{\mathcal{C}}_{\text{str}}(X_H)$ their quotients by closed forms. Similarily, we let $\mathcal{S}_\text{loc}(H)$, $\mathcal{S}(H)$, and $\mathcal{S}_\text{str}(H)$ denote the space of local, global, and strict continuous symmetries respectively, and  $\widetilde{\mathcal{S}}_\text{loc},(H),\widetilde{\mathcal{S}}(H)$, and $\widetilde{\mathcal{S}}_\text{str}(H)$ their quotient by elements in the kernel of $\omega$. We obtain:
\begin{theorem}
There exists an isomorphism of graded Lie algebras from $(\widetilde{\mathcal{S}}(H),[\cdot,\cdot])$ and $(\widetilde{\mathcal{C}}(X_H),\{\cdot,\cdot\})$ and from $(\widetilde{\mathcal{S}}_{\text{loc}}(H),[\cdot,\cdot])$ to $(\widetilde{\mathcal{C}}_{\text{loc}}(X_H),\{\cdot,\cdot\})$. Moreover, there exists an injective graded Lie algebra homomorphism from $(\widetilde{\mathcal{S}}_{\text{str}}(H),[\cdot,\cdot])$ to $(\widetilde{\mathcal{C}}(X_H),\{\cdot,\cdot\})$ and from   $(\widetilde{\mathcal{C}}_{\text{str}}(X_H),\{\cdot,\cdot\})$ to $(\widetilde{\mathcal{S}}(H),[\cdot,\cdot])$.
\end{theorem}

Furthermore, we show that under certain assumptions for a group action on $M$, a weak moment map $(f)$ gives rise to a whole family of conserved quantities and continuous symmetries. Specifically, a group action on a multi-Hamiltonian system $(M,\omega,H)$ is called locally, globally, or strictly $H$ preserving if the Lie derivative of $H$ under each infinitesimal generator from $\g$ is closed, exact, or zero respectively. Under a locally or globally $H$ preserving action, it was shown in \cite{cq} that for any $p\in\Rho_{\g,k}$, the $k$-th Lie kernel (see Definition \ref{Lie kernel}), $f_k(p)$ is locally conserved, and if the group strictly preserves $H$ then $f_k(p)$ is globally conserved. 

We add to this result by showing that under the above assumptions $V_p$ is a local or global continuous symmetry. In particular, let $S_k=\{V_p; p\in\Rho_{\g,k}\}$ denote the infinitesimal generators coming from the Lie kernel. Then $S=\oplus S_k$ is a differential graded Lie algebra. Let $C_k=\{f_k(p) ; p\in\Rho_{\g,k}\}$ denote the image of the moment map. We set $C=\oplus C_k$ and show that $C\cap L_\infty(M,\omega)$ is an $L_\infty$-subalgebra of $L_\infty(M,\omega)$, the Lie $n$-algebra of observables. We then obtain

\begin{theorem}
For any $H$ preserving action, a homotopy moment map induces an $L_\infty$-morphism from $S$ to $C\cap L_\infty(M,\omega)$.
\end{theorem}

We present two applications of our results:

First, we briefly recall the classical momentum and position functions discussed in Chapter 5.4 of \cite{Marsden}. Let $N$ be a manifold and $M=T^\ast N$. Then $M$ has a canonical symplectic form $\omega=-d\theta$, and the pair $(M,\omega)$ is called the phase space (see Example \ref{multisymplectic phase space}). Given a function $f\in C^\infty(N)$, by pulling it back to $M$ we obtain a function $\widetilde f\in C^\infty(M)$, called the classical position function corresponding to $f$. Given a vector field $X\in\Gamma(TN)$, we define $P(X)=X^\sharp\hk\theta$, where $X^\sharp$ is the complete lift of $X$ (see Definition \ref{complete lift}). We call the function $P(X)\in C^\infty(M)$ the classical momentum function corresponding to $X$. These position and momentum functions satisfy the following bracket relations: For $X,Y\in\Gamma(TN)$ and $f,g\in C^\infty(N)$,
\begin{equation}\label{classical 11}\{P(X),P(Y)\}=P([X,Y]),\end{equation}
\begin{equation}\label{classical 22}\{\widetilde f,\widetilde g\}=0,\end{equation}and
\begin{equation}\label{classical 33}\{\widetilde f, P(X)\}=\widetilde{Xf}.\end{equation} 
These bracket relations form the bridging gap from classical to quantum mechanics. 

In Section $7$ we generalize these bracket relations in the following way: Given a manifold $N$, we let $M=\Lambda^k(T^\ast N)$, for $k\geq 1$. Then $M$ has a canonical $k$-plectic structure $\omega=-d\theta$ and the pair $(M,\omega)$ is called the multisymplectic phase space (see Example \ref{multisymplectic phase space}). For $\alpha\in\Omega^{k-s}(N)$, we denote its pullback to $M$ by $\widetilde\alpha$, and call it the classical position form. Note that $\widetilde\alpha$ is in $\Omega^{k-s}(M)$. Moreover, given $X\in\Gamma(\Lambda^s(TN))$ we define its classical momentum form to be \[P(X)=-(-1)^{\frac{(s+1)(s+2)}{2}}X^\sharp\hk\theta,\] a $(k-s)$-form on $M$, where again $X^\sharp$ is the complete lift of $X$ (see Definition \ref{complete lift} for details). Letting $\mathfrak{g}=\Gamma(TN)$, we find that for $X\in\Rho_{\g,s}$, $Y\in\Rho_{\g,t}$ and $\alpha\in\Omega^{k-s}(N)$, $\beta\in\Omega^{k-t}(N)$,
\begin{equation}\{P(X),P(Y)\}=-(-1)^{ts+s+t}P([X,Y]) -(-1)^{\frac{(s+1)(s+2)+(t+1)(t+2)}{2}}d(X^\sharp\hk Y^\sharp\hk\theta),\end{equation}
\begin{equation}\{\widetilde\alpha,\widetilde \beta\}=0,\end{equation}and
\begin{equation}\{\widetilde\alpha,P(Y)\}=-(-1)^{\frac{t(t+1)}{2}}\widetilde{(Y\hk d\alpha)}.\end{equation} Notice that these equations are generalization of equations (\ref{classical 11}), (\ref{classical 22}), and (\ref{classical 33}) respectively.

In Section $7$, we also apply our work to manifolds with closed $G_2$-structure. In particular, we derive some identities and extend Example 6.7 of \cite{ms} by obtaining a homotopy moment map for a $T^2$ action on a closed $G_2$-manifold.
\del{

This issue of only obtaining locally defined conserved quantities from a continuous symmetries is also present in Hamiltonian mechanics. It is solved if one only considers the symmetries and conserved quantities coming from a moment map. We find the same thing happens in multisymplectic when one considers the symmetries and conserved quantities coming from a multi moment map. Moreover, we will have to restrict the domain of the moment map to the Lie kernel. In \cite{cq} it was shown that if a group acts on a multi Hamiltonian system $(M,\omega,H)$  and locally (or globally) preserves $H$ then for a homotopy moment map $f_k:\Lambda^k\to\Omega^{n-k}(M)$ we have that $f_k(p)$ is a local conserved quantity for each $p\in\Rho_{g,k}$.

Considering homotopy moment maps restricted to the Lie kernel we obtain

\begin{proposition}
A homotopy moment map restricts to an $L_\infty$-morphism from $(\Rho_\g,[\cdot,\cdot],\del)$ to $\widetilde L_\infty(M,\omega)$.
\end{proposition}

Next we consider how the algebraic structures of conserved quantities and continuous symmeries are related. 
We let $C$ denote the symmetries coming from the moment map and $S$ the conserved quantities.
\begin{proposition}
A homotopy moment map gives an isomorphism of graded Lie algebras from $C(X_H)\cap \widetilde C$ to $S(H)\cap \widetilde S$.
\end{proposition}

This proposition is actually a consequence of something more general. Namely,

\begin{theorem}
There is a graded Lie algebra isomorphism from $C(X_H)\cap\Omega_{\text{Ham}}(M)$ to $S(H)\cap\X_{\text{Ham}}(M)$.
\end{theorem}

In symplectic geometry, Noether's theorem says that the map $V_\xi\mapsto\mu(\xi)$ gives a one-to-one correspondence between symmetries and conserved quantities. Here $\mu(\xi)$ is the conserved quantity coming from the symmetry $V_\xi$. Hence we mimic this correspondence to say that a a symmetry is a vector field $X$ satisyftynig $X\hk\omega$ is exact. 

Since a 
By adding in this requirement, we can use the 'generalized Poisson bracket' defined in \ref{dropbox} and say more about the conserved quantity.

In symplectic geometry the conserved quantities form a Lie subalgebra of $(C^\infty(M),\{\cdot,\cdot\})$. The continuous symmetries form a Lie subalgebra of $(\Gamma(TM),[\cdot,\cdot])$. Noether's theorem says that  and we study how this generalizes to multisymplectic geometry.

says that for a differential form $\alpha$ is conserved was a different

In this thesis we use multisymplectic geometry  to generalize notions of symmetry and conserved quantity  Hamiltonian mechanics as extends concepts from symplectic geometry 
To formulate Noether's theorem, we use the notion of  conserved quantity on a multisymplectic manifold defined in \cite{cq}. Intuitively, 

Moreover, in \cite{dropbox} a generalization of the Poisson bracket to multisymplectic systems was introduced, and we show how this bracket interacts with the aforementioned conserved quantities. 

We define our continuous symmetries to be multivector field. In order to setup a corrsepondence between  Moreover, in \cite{cq} a notion of conserved quantity on a multisymplectic manifold was defined. 

The work in this thesis connects several results given in \cite{cq}, \cite{ms}, \cite{dropbox}.

In \cite{cq} a notion of conserved quantity on a multyismplectic manifold was introduced, generalizing conservation laws from Hamiltonian mechanics. In \cite{ms}, Madsen and Swann explained how one could define a moment map on a multi symplectic manifold so that its image was contained in a finite dimensional vector space, analogous to the scenario in symplectic geometry. In \cite{dropbox} they give a notion of Hamiltonian structures on multisymplectic manifolds.

In this thesis, we aim to give a notion of symmetry on a multiysmplectic manifold and set up a correspondence generalizing Noethers theorem in Hamiltonian mechanics.  We work with a fixed multi-Hamiltonian system $(M,\omega,H)$ where $\omega$ is $n$-plectic and $H\in\Omega^{n-1}_{\text{Ham}}(M)$ is Hamiltonian. We modify the definition of a conserved quantity given in \cite{cq}, by adding the requirement that a conserved quantity must be Hamiltonian. This allows us to make use of some of the structures developed in \cite{dropbox}. In this thesis we build on these structures to give a notion of symmetry on a multisymplectic manifold. In particular, we will build on these structures to give a notion of symmetry on a multisymplectic manifold. Moreover, we show that these generalized Hamiltonian structures interact nicely with a homotopy co-moment  map. In particular,

\begin{proposition}
$\{\alpha,\beta\}$ is strictly conserved.
\end{proposition}

We then adopt  ideas from \cite{ms} by replacing the Chevalley-Eilenberg complex with a sub complex containing only Lie kernels. By doing this, we make a connection with symplectic geometry, that the image of each element in the Lie algebra under a moment map is a conserved quantity. The fact that each element of the Lie kernel gets mapped to a conserved quantity is one of the main ideas in \cite{cq}. In this thesis we show that the infinitesimal generator of each element in the Lie kernel is a symmetry and its corresponding conserved quantity is that obtained in \cite{cq}.

That is, we will investigate how a momentum map converts symmetries to conserved quantities and vice versa.  Similar to \cite{cq}, we will have three types of symmetries, and will setup the correspondence in each of the three cases.

\begin{proposition}
Fix a multi-Hamiltonian system $(M,\omega,H)$ where $\omega\in\Omega^{n+1}(M)$ and $H\in\Omega^{n-1}_{\text{Ham}}(M)$. For an element $p$ in the Lie kernel $\Rho_{\g,k}$ we have that its infinitesimal generator is a global symmetry. Its corresponding conserved quantity is $f_k(p)$, where $(f)$ is a homotopy moment map.
\end{proposition}

We get the following generalization of Noether's theorem.

\begin{theorem}
If $\alpha\in\Omega^{n-k}_{\text{Ham}}(M)$ is a (local, global) conserved quantity then $X_\alpha$ is a (local, global) continuous symmetry, for any Hamiltonian vector field $X_\alpha$ of $\alpha$. Conversely, if $A\in\Gamma(\Lambda^k(TM))$ is a continuous symmetry, then locally $A=X_\alpha$ for some $(n-k)$-form $\alpha$. If $A$ is a (local, global) continuous symmetry, then this form $\alpha$ is a (local, global) conserved quantity, locally.
\end{theorem}

We also show that a moment map is an $L_\infty$-morphism between various different spaces.

\begin{theorem}
A momentum map restricts to an $L_\infty$-morphism between $(\Rho_\g,\partial,[\cdot,\cdot])$ and $(\widetilde L, \{l_k\})$.

\end{theorem}

\begin{theorem}
A momentum map is an $L_\infty$-algebra morphism between $C(X_H)\cap \widetilde L$ and $S(H)\cap \mathfrak{X}$ (also local and strict).
\end{theorem}

We also exhibit how a moment map reduces to an Lie algebra isomorphism when restricted in a certain cohomological sense.

(This addresses two open problems proposed in \cite{questions}, namely the first and third bullet points of section 13.  Actually, the results concerning the 3rd bullet are only referenced here, the details are in other notes.)
}

\del{
Recall that for a symplectic manifold $(M,\omega)$, a Lie algebra $\g$ is said to act symplectically if $\L_{V_\xi}\omega=0$, for all $\xi\in\g$, where $V_\xi$ is its infinitesimal generator. A symplectic group action is called Hamiltonian if one can find a moment map, that is, a map $f:\g\to C^\infty(M)$ satisfying \[df(\xi)=V_\xi\hk\omega,\] for all $\xi\in\g$.
\\

In multisymplectic geometry, $\omega$ is replaced by a closed, non-degenerate $(n+1)$-form, where $n\geq 1$. A Lie algebra action is called multisymplectic if $\L_{V_\xi}\omega=0$ for each $\xi\in\g$. A generalization of moment maps from symplectic to multisymplectic geometry is given by a (homotopy) moment map. These maps are discussed in detail in \cite{questions}. A homotopy moment map is a collection of maps, $f_k:\Lambda^k\g\to \Omega^{n+1-k}(M)$, with $1\leq k \leq n+1$, satisfying \begin{equation}\label{equation1}df_k(p)=-f_{k-1}(\partial_k(p))+(-1)^{\frac{k(k+1)}{2}}V_p\hk\omega,\end{equation}for all $p\in\Lambda^k\g$, where $V_p$ is its infinitesimal generator (see Definition \ref{inf gener}). A weak (homotopy) moment map is a collection of maps $f_k:\Rho_{\g,k}\to\Omega^{n+1-k}(M)$ satisfying  \begin{equation}\label{equation 2}df_k(p)=(-1)^{\frac{k(k+1)}{2}}V_p\hk\omega,\end{equation} for $p\in\Rho_{\g,k}$. Here $\Rho_{\g,k}$ is the Lie kernel $\Rho_{\g,k}$, which is the kernel of the $k$-th Lie algebra cohomology differential $\partial_k:\Lambda^k\g\to\Lambda^{k-1}\g$, (see Definition \ref{cohomology differential}).

We see that any collection of functions satisfying equation (\ref{equation1}) must also satisfy (\ref{equation 2}). That is, any homotopy moment map induces a weak homotopy moment map.

 Weak moment maps generalize the moment maps of Madsen and Swann in \cite{ms} and \cite{MS}, and were also used to give a  multisymplectic version of Noether's theorem in \cite{me}. In this thesis, we study the existence and uniqueness of weak homotopy moment maps and show that the theory is a generalization from symplectic geometry. 
 
Recall that in symplectic geometry we have the following well-known results for a symplectic action of a connected Lie group $G$ on a symplectic manifold $(M,\omega)$:

 \begin{proposition}If the first Lie algebra cohomology vanishes, i.e. $H^1(\g)=0$, then a not necessarily equivariant moment map exists.
 \end{proposition}
 
\begin{proposition} If the second Lie algebra cohomology vanishes, i.e. $H^2(\g)=0$, then any non-equivariant moment map can be made equivariant.
 \end{proposition}
 \begin{proposition}
 If the first Lie algebra cohomology vanishes, i.e. $H^1(\g)=0$, then equivariant moment maps are unique,
 \end{proposition}

 and combining these results,
 
\begin{proposition}
  If both the first and second Lie algebra cohomology vanish, i.e. $H^1(\g)=0$ and $H^2(\g)=0$, then there exists a unique equivariant moment map.
\end{proposition}

 We generalize these results with the following theorems. Letting $\Omega^{n-k}_\mathrm{cl}$ denote the set of closed $(n-k)$-forms on $M$, we get the above propositions, in their respective order, by taking $n=k=1$.

\begin{theorem}
If $H^k(\g)=0$ or $H^0(\g,\Rho_{\g,k}^\ast)=0$ then there exists a not necessarily equivariant $k$-th component of a weak homotopy moment map. The same result holds if $H^0(\g,\Rho_{\g,k}^\ast\otimes\Omega^{n-k}_{\mathrm{cl}})=0$ and $H^0(\g,\Omega^{n-k}_{\mathrm{cl}})\not=0$.
\end{theorem}
 
\begin{theorem}
If $H^1(\g,\Rho_{\g,k}^\ast\otimes\Omega^{n-k}_{\mathrm{cl}})=0$, then any non-equivariant $k$-th component of a weak homotopy moment map can be made equivariant.
\end{theorem}

\begin{theorem}
If $H^0(\g,\Rho_{\g,k}^\ast\otimes\Omega^{n-k}_{\mathrm{cl}})=0$ then an equivariant $k$-th component of a weak homotopy moment map is unique,
\end{theorem}
and combining these results,

\begin{theorem}
If $H^0(\g,\Rho_{\g,k}^\ast\otimes\Omega^{n-k}_{\mathrm{cl}})=0$, $H^0(\g,\Omega^{n-k}_{\mathrm{cl}})\not=0$, and $H^1(\g,\Rho_{\g,k}^\ast\otimes\Omega^{n-k}_{\mathrm{cl}})=0$, then there exists a unique equivariant $k$-th component of a weak homotopy moment map. Moreover, if $H^0(\g,\Rho_{\g,k}^\ast\otimes\Omega^{n-k}_{\mathrm{cl}})=0$, $H^0(\g,\Omega^{n-k}_{\mathrm{cl}})\not=0$,  and $H^1(\g,\Rho_{\g,k}^\ast\otimes\Omega^{n-k}_{\mathrm{cl}})=0$ for all $1\leq k\leq n$, then a full equivariant weak homotopy moment map exists and is unique.
\end{theorem}
 
 \del{
 is a  If the first Lie algebra cohomology of a Lie group vanishes then we there exists a not necessarily equivariant moment map.  If the second Lie algebra cohomology vanishes, then we can make any non-equivariant moment map  equivariant by adding a cocycle. Moreover, any two equivariant moment maps differ by something in $[\g,\g]^0$ and thus if $H^1(\g)=0$, then any equivariant moment map is unique.

We obtain strong generalizations of these results by showing analogous results are obtained in the setting of multisymplecitc geometry. In order to do this, however, we must consider a refined version of a homotopy moment map. This refining makes two previously non-existing connections with symplectic geometry. In particular, with our new definition, every element in the image of the moment map is a conserved quantity in the sense defined in \cite{cq}. Moreover, we also get that $\L_{X_H}H=0$ for a Hamiltonian vector field, which doesn't happen in the general setting.  

Specifically, in our setup we obtain that if $H^k(\g)=0$ then the $k$th component of a pre homotopy moment map exists. Moreover, if $H^{k+1}(\g)=0$ then any $k$th component of a pre homotopy moment map can be made equivariant. We also show that the difference of two $kth$ components of homotopy moment maps differ by something in $[\Rho_{\g,k},\g]^0$, analogous to the setup of symplectic geometry.
We find that if $H^k(\g)=0$ then the $k$th component of an homotopy moment map is necessarily unique. Here are some theorems that generalize results from symplectic geometry.

\begin{theorem}
If $H^k(\g)=0$ then the $k$-th component of a not necessarily equivariant refined homotopy moment map exists. If there exists an equivariant $k$th component of a refined homotopy moment, then it is unique.
\end{theorem}

\begin{theorem}
If $H^{k+1}(\g)=0$ then any non equivariant restricted homotopy moment map can be made equivariant.
\end{theorem}

\begin{theorem}
For any refined homotopy moment map $(f)$ there exists a well defined cohomology class $[\sigma]$.
\end{theorem}

\begin{theorem}
If $H^1(\g)=\cdots=H^{n+1}(\g)=0$ then a unique equivariant refined homotopy moment map exists.
\end{theorem}

Recall that in symplectic geometry, a moment map gives a Lie algebra morphism. Suppose that $f:g\to\ C^\infty(M)$ is a non-equivariant moment map. Then $f$ will only setup a Lie algebra morphism from $(\g,[\cdot,\cdot])$ to $(C^\infty(M),\{\cdot,\cdot\})$ if $f$ is equivariant. 

In general, an equivariant moment map will give a Lie algebra morphism. However, it is not necessarily true that a moment map which is a Lie algebra morphism is necessarily equivariant.  It is true if the group is compact and connected. A moment map that is not equivariant will not necessarily be a Lie algebra morphism. This is because it is only true that $f([\xi,\eta])-\{f(\xi),f(\eta)\}$ is constant. If the moment map is equivariant then this constant vanishes. However, it is not true that if the constant vanishes then the map is equivariant. Now given 

All the existence and uniqueness theorems hold if the group is compact and connected. If it is not compact or connected, then the definition of equivariance has to be changed to be that is a Lie algebra morphism.

In symplectic geometry, a moment map is always a Lie algebra homomorphism from $\g$ to $\widetilde\Omega$. If the moment map is equivariant, then it is a Lie algebra morphism from $\g$ to $\Omega/exact$. This is the first thing we generalize.

An equivariant moment map always gives a Lie algebra morph from $\g$ to $\Omega/exact$. Since this is the property we are interested, we will say a moment is equivariant if its a Lie alge morphism between these spaces, just like Da Silva does.

If $H^k(\g)=0$ then for a symplectic action a not-necessarily-equivariant moment map exists. Moreover, if $H^k(\g)=0$ then equivariant moment maps are unique.  If $H^{k+1}(\g)=0$ then one can always obtain an equivariant moment from a non-equivariant moment. We generalize each of these theorems.

This work gernealizes the existence and uniqueness results of moment maps in symplectic geometry. It also generalizes the existence and uniquess results of the moment maps introduced by Madsen and Swann in \cite{ms} and \cite{MS}.
}

We also show that the morphism properties of moment maps from symplectic geometry are preserved in multisymplectic geometry. More specifically, recall that in symplectic geometry the equivariance of a moment map $f:\g\to C^\infty(M)$ is characterized by whether or not $f$ is a Lie algebra morphism. That is, $f$ is equivariant if and only if \[f([\xi,\eta])=\{f(\xi),f(\eta)\},\] for all $\xi,\eta\in\g$. However, as shown in Theorem 4.2.8 of \cite{Marsden} it is always true that $f$ induces a Lie algebra morphism between $\g$ and $C^\infty(M)/\text{closed}$, because $df([\xi,\eta])=d\{f(\xi),f(\eta)\}$. 

We generalize these results to multisymplectic geometry by showing that:

\begin{theorem}
For any $1\leq k \leq n$, a weak homotopy moment map is always a $\g$-module morphism from $\Rho_{\g,k}\to\Omega^{n-k}_{\mathrm{Ham}}(M)/\text{closed}$. An equivariant weak homotopy moment map is a $\g$-module morphism from $\Rho_{\g,k}\to\Omega^{n-k}_{\mathrm{Ham}}(M)$.
\end{theorem}

Here $\Omega^{n-k}_{\mathrm{Ham}}(M)$ is denoting the space of multi-Hamiltonian forms, which are differential forms $\alpha\in\Omega^{n-k}(M)$ satisfying $d\alpha=X_\alpha\hk\omega$ for some $X_\alpha\in\Gamma(\Lambda^k(TM))$ (see Definition \ref{ham k form}). These forms were introduced in \cite{me}, and give a notion of a multi-symmetry which occurs when there is a given Hamiltonian $(n-1)$-form $H\in\Omega^{n-1}_{\mathrm{Ham}}(M)$ (see Definition \ref{ham 1 form}).
}

 Lastly, in Section 4, we study the existence and uniqueness of weak homotopy moment maps and show that the theory can be generalized directly from symplectic geometry. We also show that the equivariance of a weak moment map can be characterized in terms of $\g$-module morphisms, analogous to the case of symplectic geometry. To state our results, recall that in symplectic geometry we have the following well-known results on the existence and uniqueness of moment maps, whose proofs can be found in \cite{dasilva}.

 \begin{proposition} Consider the symplectic action of a connected Lie group $G$ acting on a symplectic manifold $(M,\omega)$.
 
 \begin{itemize}
 \item  If the first Lie algebra cohomology vanishes, i.e. $H^1(\g)=0$, then a not necessarily equivariant moment map exists.
 
 \item If the second Lie algebra cohomology vanishes, i.e. $H^2(\g)=0$, then any non-equivariant moment map can be made equivariant.

\item  If the first Lie algebra cohomology vanishes, i.e. $H^1(\g)=0$, then equivariant moment maps are unique,
\end{itemize}

 and combining these results,
 
\begin{itemize}

\item  If both the first and second Lie algebra cohomology vanish, i.e. $H^1(\g)=0$ and $H^2(\g)=0$, then there exists a unique equivariant moment map.

\end{itemize}

\end{proposition}

 We generalize these results with the following theorems. Letting $\Omega^{n-k}_\mathrm{cl}$ denote the set of closed $(n-k)$-forms on $M$, we get the above propositions, in their respective order, by taking $n=k=1$.

\begin{theorem}
If $H^0(\g,\Rho_{\g,k}^\ast)=0$ then there exists a not necessarily equivariant weak homotopy $k$-moment map.
\end{theorem}
 
\begin{theorem}
If $H^1(\g,\Rho_{\g,k}^\ast\otimes\Omega^{n-k}_{\text{cl}})=0$, then any non-equivariant weak homotopy $k$-moment map can be made equivariant.
\end{theorem}

\begin{theorem}
If $H^0(\g,\Rho_{\g,k}^\ast\otimes\Omega^{n-k}_{\text{cl}})=0$ then an equivariant weak homotopy $k$-moment map is unique.
\end{theorem}
Combining these results, we obtain from the Kunneth formula (see e.g. \cite{kunneth}):

\begin{theorem}
If $H^0(\g,\Rho_{\g,k}^\ast)=0$ and $H^1(\g,\Rho_{\g,k}^\ast)=0$, then there exists a unique equivariant weak $k$-moment map $f_k:\Rho_{\g,k}\to\Omega^{n-k}$. Moreover, if $H^0(\g,\Rho_{\g,k}^\ast)=0$ and $H^1(\g,\Rho_{\g,k}^\ast)=0$ for all $1\leq k\leq n$, then a full equivariant weak moment map exists and is unique.
\end{theorem}
 
 \del{
 is a  If the first Lie algebra cohomology of a Lie group vanishes then we there exists a not necessarily equivariant moment map.  If the second Lie algebra cohomology vanishes, then we can make any non-equivariant moment map  equivariant by adding a cocycle. Moreover, any two equivariant moment maps differ by something in $[\g,\g]^0$ and thus if $H^1(\g)=0$, then any equivariant moment map is unique.

We obtain strong generalizations of these results by showing analogous results are obtained in the setting of multisymplecitc geometry. In order to do this, however, we must consider a refined version of a homotopy moment map. This refining makes two previously non-existing connections with symplectic geometry. In particular, with our new definition, every element in the image of the moment map is a conserved quantity in the sense defined in \cite{cq}. Moreover, we also get that $\L_{X_H}H=0$ for a Hamiltonian vector field, which doesn't happen in the general setting.  

Specifically, in our setup we obtain that if $H^k(\g)=0$ then the $k$th component of a pre homotopy moment map exists. Moreover, if $H^{k+1}(\g)=0$ then any $k$th component of a pre homotopy moment map can be made equivariant. We also show that the difference of two $kth$ components of homotopy moment maps differ by something in $[\Rho_{\g,k},\g]^0$, analogous to the setup of symplectic geometry.
We find that if $H^k(\g)=0$ then the $k$th component of an homotopy moment map is necessarily unique. Here are some theorems that generalize results from symplectic geometry.

\begin{theorem}
If $H^k(\g)=0$ then the $k$-th component of a not necessarily equivariant refined homotopy moment map exists. If there exists an equivariant $k$th component of a refined homotopy moment, then it is unique.
\end{theorem}

\begin{theorem}
If $H^{k+1}(\g)=0$ then any non equivariant restricted homotopy moment map can be made equivariant.
\end{theorem}

\begin{theorem}
For any refined homotopy moment map $(f)$ there exists a well defined cohomology class $[\sigma]$.
\end{theorem}

\begin{theorem}
If $H^1(\g)=\cdots=H^{n+1}(\g)=0$ then a unique equivariant refined homotopy moment map exists.
\end{theorem}

Recall that in symplectic geometry, a moment map gives a Lie algebra morphism. Suppose that $f:g\to\ C^\infty(M)$ is a non-equivariant moment map. Then $f$ will only setup a Lie algebra morphism from $(\g,[\cdot,\cdot])$ to $(C^\infty(M),\{\cdot,\cdot\})$ if $f$ is equivariant. 

In general, an equivariant moment map will give a Lie algebra morphism. However, it is not necessarily true that a moment map which is a Lie algebra morphism is necessarily equivariant.  It is true if the group is compact and connected. A moment map that is not equivariant will not necessarily be a Lie algebra morphism. This is because it is only true that $f([\xi,\eta])-\{f(\xi),f(\eta)\}$ is constant. If the moment map is equivariant then this constant vanishes. However, it is not true that if the constant vanishes then the map is equivariant. Now given 

All the existence and uniqueness theorems hold if the group is compact and connected. If it is not compact or connected, then the definition of equivariance has to be changed to be that is a Lie algebra morphism.

In symplectic geometry, a moment map is always a Lie algebra homomorphism from $\g$ to $\widetilde\Omega$. If the moment map is equivariant, then it is a Lie algebra morphism from $\g$ to $\Omega/exact$. This is the first thing we generalize.

An equivariant moment map always gives a Lie algebra morph from $\g$ to $\Omega/exact$. Since this is the property we are interested, we will say a moment is equivariant if its a Lie alge morphism between these spaces, just like Da Silva does.

If $H^k(\g)=0$ then for a symplectic action a not-necessarily-equivariant moment map exists. Moreover, if $H^k(\g)=0$ then equivariant moment maps are unique.  If $H^{k+1}(\g)=0$ then one can always obtain an equivariant moment from a non-equivariant moment. We generalize each of these theorems.

This work gernealizes the existence and uniqueness results of moment maps in symplectic geometry. It also generalizes the existence and uniquess results of the moment maps introduced by Madsen and Swann in \cite{ms} and \cite{MS}.
}

We also show that the morphism properties of moment maps and their relationship to equivariance in multisymplectic geometry is analogous to the case of symplectic geometry. More specifically, recall that in symplectic geometry the equivariance of a moment map $f:\g\to C^\infty(M)$ is characterized by whether or not $f$ is a Lie algebra morphism. That is, $f$ is equivariant if and only if \[f([\xi,\eta])=\{f(\xi),f(\eta)\},\] for all $\xi,\eta\in\g$. However, as shown in Theorem 4.2.8 of \cite{Marsden} it is always true that $f$ induces a Lie algebra morphism between $\g$ and $C^\infty(M)/\text{constant}$, because $df([\xi,\eta])=d\{f(\xi),f(\eta)\}$. 

We generalize these results to multisymplectic geometry by showing that:

\begin{theorem}
For any $1\leq k \leq n$, a weak  $k$-moment map is always a $\g$-module morphism from $\Rho_{\g,k}\to\Omega^{n-k}_{\mathrm{Ham}}(M)/\text{closed}$. A weak $k$-moment map is equivariant if and only if it is a $\g$-module morphism from $\Rho_{\g,k}\to\Omega^{n-k}_{\mathrm{Ham}}(M)$.
\end{theorem}

This thesis is a synthesis of the results obtained in \cite{me2} and \cite{future}. If not stated otherwise, we will always assume our manifold to be connected.
\newpage
\section{Background}

We start by recalling some basic concepts from symplectic geometry.

\subsection{Symplectic Geometry}

Let $M$ be a manifold and let $\omega\in\Omega^2(M)$ be a $2$-form. By definition, for each $p\in M$ we have that $\omega(p):=\omega_p$ is a skew-symmetric bilinear map $\omega_p:T_pM\times T_pM\to\R$.
\begin{definition} A $2$-form $\omega\in\Omega^2(M)$ is said to be symplectic if it is closed and if $\omega_p$ is non-degenerate for each $p\in M$. Non-degeneracy of $\omega$ means that the map \begin{equation}T_pM\to T^\ast_p M \ \ \ \ \ \ \ \ \ \ V_p\mapsto V_p\hk\omega_p\end{equation} is an isomorphism for each $p\in M$. In such a case, the pair $(M,\omega)$ is called a symplectic manifold.
\end{definition}

\begin{definition}
A triple $(M,\omega, H)$ where $(M,\omega)$ is a symplectic manifold and $H\in C^\infty(M)$ is called a Hamiltonian system. The function $H$ is called the associated Hamiltonian function.
\end{definition}


\begin{definition}
Given $f\in C^\infty(M)$, we let $X_f$ represent the unique vector field in $\Gamma(TM)$ satisfying \begin{equation}\label{whatsup}X_f\hk\omega = df.\end{equation} The vector field $X_f$ is called the Hamiltonian vector field associated to $f$. 
\end{definition}

Hamiltonian vector fields allow us to define the Poisson bracket on $C^\infty(M)$. \begin{definition} Given $f,g\in C^\infty(M)$ their Poisson bracket is defined to be \begin{equation}\label{whatsup2}\{f,g\}:=X_g\hk X_f\hk\omega=\omega(X_f,X_g).\end{equation} \end{definition}The Poisson bracket turns $C^\infty(M)$ into a Lie algebra. To show this $\{\cdot,\cdot\}$, we use the fact from (\ref{whatsup}) and (\ref{whatsup2}) that $\{f,g\}$ is also equal to $X_fg$. Moreover, one can verify the Leibniz rule: \[\{f,gh\}=g\{f,h\}+\{f,g\}h.\]A straightforward computation shows that  the map $C^\infty(M)\to \Gamma(TM)$ given by $f\mapsto X_f\in\Gamma(TM)$ is a Lie algebra anti-homomorphism. That is \[X_{\{f,g\}}=-[X_f,X_g].\]

\del{
We next recall the equivariance of moment maps and the relevant cohomology. Let $(M,\omega)$ be a symplectic manifold, and $\Phi:G\times M\to M$ a symplectic Lie group action by a connected Lie group $G$ and consider the induced  Lie algebra action $\g\times\Gamma(TM)\to\Gamma(TM)$. Suppose that a moment map  $f:\g\to C^\infty(M)$ exists. That is, $df(\xi)=V_\xi\hk\omega$ for all $\xi\in\g$. By definition, $f$ is equivariant if \[f(\mathrm{Ad}_{g^{-1}}\xi)=\Phi_g^\ast f(\xi).\]In order to study equivariance in more detail we follow Chapter 4.2 of \cite{Marsden} and for $g\in G$ and $\xi\in\g$ define $\psi_{g,\xi}\in C^\infty(M)$ by \begin{equation}\label{psi}\psi_{g,\xi}(x):= f(\xi)(\Phi_g(x))-f(\mathrm{Ad}_{g^{-1}}\xi)(x).\end{equation}

\begin{proposition}\label{psi constant}
For each $g\in G$ and $\xi\in\g$, the function $\psi_{g,\xi}\in C^\infty(M)$ is constant.
\end{proposition}

\begin{proof}
This is Proposition 4.2.3 of \cite{Marsden}. We give a more general proof in Proposition \ref{general closed}.
\end{proof}
Since $\psi_{g,\xi}$ is constant, we may define the map $\sigma:G\to\g^\ast$ by \[\sigma(g)(\xi):=\psi_{g,\xi},\] where the right hand side is the constant value of $\psi_{g,\xi}$.
\del{This motivates consideration of the chain complex \[\g^\ast\otimes\R\to(G\to \g^\ast\otimes\R)\to(G\times G\to\g^\ast\otimes\R)\to\cdots,\]with the natural action of $G$ on $\g^\ast\otimes\R$ given by \[g\cdot\alpha(\xi):=Ad_{g^{-1}}^\ast\alpha(\xi)\] for $g\in G,\alpha\in\g^\ast\otimes\R$ and $\xi\in\g$.}
\begin{proposition}\label{cocycle}
The map $\sigma:G\to\g^\ast$ is a cocycle in the chain complex \[\g^\ast\to C^1(G,\g^\ast)\to C^2(G,\g^\ast)\to\cdots.\]That is, $\sigma(gh)=\sigma(g)+\mathrm{Ad}_{g^{-1}}^\ast\sigma(h)$ for all $g,h\in G$. 
\end{proposition}
\begin{proof}
This is proved in Proposition 4.2.3 of \cite{Marsden}. We give a more general version of the proof in Proposition \ref{multi cocycle}.
\end{proof}
\begin{definition}The map $\sigma$ is called the cocycle corresponding to $f$. A moment map is equivariant if $\sigma=0$. \end{definition} 

The following proposition shows that for any symplectic group action, the cocycle gives a well defined cohomology class. 

\begin{proposition}\label{cohomology class}
For any symplectic action of  $G$ on $M$ admitting a moment map, there is a well defined cohomology class. More specifically, if $f_1$ and $f_2$ are two moment maps, then their corresponding cocycles $\sigma_1$ and $\sigma_2$ are in the same cohomology class, i.e. $[\sigma_1]=[\sigma_2]$.
\end{proposition}

\begin{proof}
This is Proposition 4.2.5 of \cite{Marsden}. We give a more general proof in Theorem \ref{general class}.
\end{proof}

By definition, we see that $\sigma$ is measuring the equivariance of $f$. That is, $\sigma=0$ if and only if $f$ is equivariant. Moreover, if the cocycle corresponding to a moment map vanishes in cohomology, the next proposition shows that we can modify the original moment map to make it equivariant.

\begin{proposition}\label{sigma class zero}
Suppose that $f$ is a moment map with corresponding cocycle $\sigma$. If $[\sigma]=0$ then $\sigma=\partial\theta$ for some $\theta\in\g^\ast$, and $f+\theta$ is an equivariant moment map.
\end{proposition}

\begin{proof}
We have that $(f)+\theta$ is a moment map since $\theta(\xi)$ is constant for all $\xi\in\g$.  Let $\widetilde\sigma$ denote the cocycle of $f+\theta$. Note that by equation (\ref{group differential 2}) we have $\partial\theta(g)(\xi)=\theta(\mathrm{Ad}_{g^{-1}}\xi)$. It follows that for $g\in G$ and $\xi\in\g$ we have that \begin{align*}
\widetilde\sigma(g)(\xi)&=\Phi_g^\ast f(\xi)+\Phi_g^\ast\theta(\xi)-f(\mathrm{Ad}_{g^{-1}}(\xi))-\theta(\mathrm{Ad}_{g^{-1}}\xi)\\
&=\Phi_g^\ast f(\xi)-f(\mathrm{Ad}_{g^{-1}}(\xi))-\theta(\mathrm{Ad}_{g^{-1}}\xi)&\text{ since $\theta(\xi)$ is constant}\\
&=\sigma(g)(\xi)-\partial\theta(g)(\xi)\\
&=\sigma(g)(\xi)-\sigma(g)(\xi)&\text{since $\partial\theta =\sigma$}\\
&=0.
\end{align*}
\end{proof}

Next we turn to the infinitesimal version of equivariance. That is, we differentiate equation (\ref{psi}) to obtain the map $\Sigma:\g\times\g\to C^\infty(M)$ defined  by $\Sigma(\xi,\eta):=\left.\frac{d}{dt}\right|_{t=0}\psi_{\exp(t\eta),\xi}$. A straightforward computation, which we generalize in Proposition \ref{comp of Sigma}, shows that \[\Sigma(\xi,\eta)=f([\xi,\eta])-\{f(\xi),f(\eta)\}.\] Another quick computation shows that $df([\xi,\eta])=d\{f(\xi),f(\eta)\}$, so that $\Sigma(\xi,\eta)$ is a constant function for every $\xi,\eta\in\g$. That is, $\Sigma$ is a function from $\g\times\g$ to $\R$. 

\begin{proposition}\label{inf cocycle}The map $\Sigma:\g\times\g\to\R$ is a Lie algebra $2$-cocycle in the chain complex \[\R\to C^1(\g, \R)\to C^2(\g,\R)\to\cdots.\] 
\end{proposition}

\begin{proof}

This is Theorem 4.2.6 of \cite{Marsden}. We give a more general proof in Proposition \ref{general inf}.
\end{proof}

\begin{definition}\label{inf equiv moment}
A moment map $f:\g\to C^\infty(M)$ is infinitesimally equivariant if $\Sigma=0$, i.e. if \begin{equation}\label{equivariant equation 2}f([\xi,\eta])=\{f(\xi),f(\eta)\}\end{equation} for all $\xi,\eta\in\g$.
\end{definition}

\begin{proposition}
For a connected Lie group, infinitesimal equivariance and equivariance are equivalent.
\end{proposition}

\begin{proof}
This is clear since $\Sigma$ is just the derivative of $\sigma$.
\end{proof}

Since we will always be working with connected Lie groups, we will abuse terminology and call a moment map equivariant if it satisfies equation (\ref{equivariant equation}) or (\ref{equivariant equation 2}).
\del{
\begin{proposition}\label{lie alg morph 1}
A pre-moment map is a Lie algebra morphism from $(\g,[\cdot,\cdot])$ to $(C^\infty(M)/\mathrm{closed},\{\cdot,\cdot\})$.
\end{proposition} 
\begin{proof}
Since $df([\xi,\eta])=d\{f(\xi),f(\eta)\}$ it follows that $f([\xi,\eta])-\{f(\xi),f(\eta)\}$ is a constant. Hence, for a pre-moment map it is always true that $f([\xi,\eta])=\{f(\xi),f(\eta)\}$ in the quotient space.
\end{proof}

\begin{proposition}\label{lie alg morph 2}
If $f$ is an equivariant moment map then $f$ is a Lie algebra morphism from $(\g,[\cdot,\cdot])$ to $(C^\infty(M),\{\cdot,\cdot\})$.
\end{proposition}
\begin{proof}
This follows from the above, since if $f$ is equivariant, then $\Sigma=0$.
\end{proof}

\begin{remark}
The converse to Proposition \ref{lie alg morph 2} is not necessarily true. That is, if a moment map is a Lie algebra morphism, then it is not necessarily equivariant. This is true, however, if the acting group is compact and connected. Nonetheless, we will abuse terminology and call a moment map equivariant if it is a Lie group homomorphism.
\end{remark}
}

We now recall the classical results on the existence and uniqueness of moment maps.

Let $G$ be a connected Lie group acting symplectically on a symplectic manifold $(M,\omega)$. A straightforward computation shows that

\begin{proposition}\label{bracket gives}
For any $\xi,\eta\in\g$ we have \[[V_\xi,V_\eta]\hk\omega=d(V_\xi\hk V_\eta\hk\omega).\]
\end{proposition}

\del{
\begin{proposition}\label{H1}
We have that $H^1(\g)=[\g,\g]^0$, where $[\g,\g]^0$ is the annihilator of $[\g,\g]$.
\end{proposition}

\begin{proof}
This follows since for $c\in\g^\ast$ we have $\partial c(\xi,\eta)=c([\xi,\eta])$ by definition.
\end{proof}
}

Another standard result is the following.
\begin{proposition}\label{H2}
We have that $H^1(\g)=0$ if and only if $\g=[\g,\g]$.
\end{proposition}

\begin{proof}
For $c\in\g^\ast=C^1(\g)$ we have $\partial c(\xi,\eta)=c([\xi,\eta])$ by definition. It follows that $H^1(\g)=[\g,\g]^0$, where $[\g,\g]^0$ is the annihilator of $[\g,\g]$. The claim now follows.
\end{proof}

Combining these two propositions, we thus see that 

\begin{proposition}\label{H4}
If $H^1(\g)=0$, then any symplectic action admits a moment map, which is not necessarily equivariant.
\end{proposition}

\begin{proof}
By Proposition \ref{H2} we see that every element in $\g$ is a sum of elements of the form $[\xi,\eta]$. The claim now follows from Proposition \ref{bracket gives} as we may define $f([\xi,\eta])$ to be $V_\xi\hk V_\eta\hk\omega$. 
\end{proof}

Next we study when a non-equivariant moment map can be made equivariant:

Proposition \ref{inf cocycle} shows that the map $\Sigma$ corresponding to a moment map $f$ is a Lie algebra $2$-cocycle. The next proposition shows that if the cocycle is exact then $f$ can be made equivariant.

\begin{proposition}\label{if exact}
Let $f$ be a moment map and $\Sigma$ its corresponding cocycle. If $\Sigma=\partial(l)$ for some $l$, then $f-l$ is equivariant.
\end{proposition}

\begin{proof}

A proof may be found in Theorem 26.5 of \cite{Da Silva}. We give a more general proof in Proposition \ref{exact Sigma} below.

\end{proof}

Consequently, 
\begin{proposition}\label{obtain theorem}
If $H^2(\g)=0$ then one can obtain an equivariant moment map from a non-equivariant moment map.
\end{proposition}
\begin{proof}
Proposition \ref{inf cocycle} showed that $\Sigma$ is a cocycle. The claim now follows from Proposition \ref{if exact}.
\end{proof}

\begin{proposition}\label{H5}
If $f$ and $g$ are two equivariant moment maps, then $f-g$ is in $H^1(\g)$.
\end{proposition}
\begin{proof}
For $\xi,\eta\in\g$ we have that $(f-g)([\xi,\eta])=\{(f-g)(\xi),(f-g)(\eta)\}$ since $f$ and $g$ are equivariant. However, $(f-g)(\xi)$ is a constant function since both $f$ and $g$ are moment maps. The claim now follows since the Poisson bracket with a constant function vanishes.
\end{proof}
From Proposition \ref{H5} it immediately follows that
\begin{proposition}\label{H6}
If $H^1(\g)=0$ then equivariant moment moments are unique.
\end{proposition}

Consider a symplectic action of a connected Lie group $G$ acting on a symplectic manifold $(M,\omega)$. Let $f:\g\to C^\infty(M)$ be a moment map. By Definition \ref{inf equiv moment}, $f$ is equivariant if and only if $f$ is a Lie algebra morphism from $(\g,[\cdot,\cdot])$ to $(C^\infty(M),\{\cdot,\cdot\})$. That is, if and only if \[f([\xi,\eta])=\{f(\xi),f(\eta)\}.\] For a non-equivariant moment map, it is always true that 

\begin{proposition}\label{morph 1}
A moment map $f$ always induces a morphism onto the quotient of $C^\infty(M)$ by closed forms. That is, a moment map induces a Lie algebra morphism from $(\g,[\cdot,\cdot])$ to $(C^\infty(M)/\text{closed},\{\cdot,\cdot\})$, regardless of equivariance.

\end{proposition}

\begin{proof}
This follows since $df([\xi,\eta])=d\{f(\xi),f(\eta)\}$ for all $\xi,\eta\in\g$,
\end{proof}

If $f$ is equivariant, there is no need to take quotients. That is,
\begin{proposition}\label{morph 2}
If the moment map $f$ is equivariant, then $f$ is a morphism from $(\g,[\cdot,\cdot])$ to $(C^\infty(M),\{\cdot,\cdot\})$.
\end{proposition}

\begin{proof} This follows from Definition \ref{inf equiv moment}.\end{proof}

We now restate these two propositions in a way that is directly generalized to multisymplectic geometry in the next section:

Notice that $\g$ is a $\g$-module under the Lie bracket action and $C^\infty(M)$ is $\g$-module under the action $\xi\cdot\g=L_{V_{\xi}}g$, where $\xi\in\g$ and $g\in C^\infty(M)$.

Proposition \ref{morph 1} is equivalent to:

\begin{proposition}\label{Morph 1}
A moment map  $f$ always induces a $\g$-module morphism from $\g$ to $C^\infty(M)/\text{closed}$, regardless of equivariance.
\end{proposition}
Proposition \ref{morph 2} is equivalent to:
\begin{proposition}\label{Morph 2}
If the moment map $f$ is equivariant, then it is a $\g$-module morphism from $\g$ to $C^\infty(M)$.
\end{proposition}
}
\subsection{Differential Graded Lie Algebras}

Recall the definition of a differential graded Lie algebra: 

\begin{definition}
A differential graded Lie algebra is a $\Z$-graded vector space $L=\oplus_{i\in\Z} L_i$ together with a bracket $[\cdot,\cdot]:L_i\otimes L_j\to L_{i+j}$ and a differential $d:L_i\to L_{i-1}$. The bilinear map $[\cdot,\cdot]$ is graded skew symmetric:  \[[x,y]=-(-1)^{|x||y|}[y,x],\] and satisfies the graded Jacobi identity: \[(-1)^{|x||z|}[x,[y,z]]+(-1)^{|y||x|}[y,[z,x]]+(-1)^{|z||y|}[z,[x,y]]=0.\] Lastly, the differential and bilinear map satisfiy the graded Leibniz rule: \[d[x,y]=[dx,y]+(-1)^{|x|}[x,dy].\] Here we have let $x,y,$ and $z$ be arbitrary homogeneous elements in $L$ of degrees $|x|,|y|$ and $|z|$ respectively.
\end{definition}

\begin{definition}
A Gerstenhaber algebra is a $\Z$-graded algebra $A=\oplus_{i\in\Z}A_i$ that is graded commutative and has a bilinear map $[\cdot,\cdot]:A\otimes A \to A$ along with the following properties: 

\begin{itemize}

\item $|ab|=|a| + |b|$,
\item $|[a,b]|=|a|+|b|-1$ (the bilinear map has degree $-1$),

\item $[a,bc]=[a,b]c+(-1)^{(|a|-1)|b|}b[a,c]$ (the bilinear map satisfies the Poisson identity),

\item $[a,b]=-(-1)^{(|a|-1)(|b|-1)}[b,a]$,
\end{itemize} and lastly, the bilinear map satisfies the Jacobi identity: 

\begin{itemize}
\item $(-1)^{(|a|-1)(|c|-1)}[a,[b,c]]+(-1)^{(|b|-1)(|a|-1)}[b,[c,a]]+(-1)^{(|c|-1)(|b|-1)}[c,[a,b]]=0$.
 \end{itemize}
 Here we have let $|a|$ denote the degree of $a\in A$, and $ab$ the product of $a$ and $b$ in $A$.
\end{definition}
Let $(V,[\cdot,\cdot])$ be a Lie algebra. The Schouten bracket turns $\Lambda^\bullet V$, the exterior algebra of $V$, into a Gerstenhaber algebra. We quickly recall some properties of the Schouten bracket. A more detailed discussion can be found in \cite{marle}.

On decomposable multivectors $X=X_1\wedge\cdots \wedge X_k \in\Lambda^k V$ and $Y=Y_1\wedge\cdots \wedge Y_l\in\Lambda^lV$ the Schouten bracket is given by \begin{equation}\label{schoutenlocalformula} [X,Y]:=\sum_{i=1}^k\sum_{j=1}^l(-1)^{i+j}[X_i,Y_j]X_1\wedge\cdots\wedge \widehat X_i\wedge\cdots\wedge X_k\wedge Y_1\wedge\cdots\wedge\widehat Y_j\wedge\cdots\wedge Y_l.\end{equation}

\begin{proposition}
The Schouten bracket is the unique bilinear map $[\cdot,\cdot]:\Lambda^\bullet V\times\Lambda^\bullet V \to\Lambda^\bullet V$ satisfying the following properties: 
\begin{itemize}
\item If $\deg X = k$ and $\deg Y= l$ then $\deg([X,Y])=k+l-1$.
\item $[X,Y]=-(-1)^{(k+1)(l+1)}[Y,X]$.
\item It coincides with the Lie bracket on $V$.
\item It satisfies the graded Leibniz rule: For $X,Y,$ and $Z$ of degree $k,l$, and $m$ respectively, \[[X,Y\wedge Z]=[X,Y]\wedge Z+(-1)^{(k-1)l}Y\wedge[X,Z].\]
\item It satisfies the graded Jacobi identity: For $X,Y$, and $Z$ of degree $k,l$ ,and $m$ respectively, \[\sum_{\mathrm{cyclic}}(-1)^{(k-1)(m-1)}[X,[Y,Z]] = 0.\]
\end{itemize}
\end{proposition}

\begin{proof}
We leave the details of the proof to the reader, but note that existence can be proved by showing the expression given in (\ref{schoutenlocalformula}) satisfies the desired properties. In particular, that the Schouten bracket satisfies the graded Jacobi identity follows from the Jacobi identity for the Lie bracket and the graded Leibniz rule. The uniqueness of the bracket follows from the requirement that the bracket is $\R$-bilinear. More details can be found in \cite{marle}, for example.
\end{proof}

We now let $M$ be a manifold and consider the Gerstenhaber algebra $(\Gamma(\Lambda^\bullet(TM)),\wedge,[\cdot,\cdot])$. 

\begin{definition}
For a decomposable multivector field $X=X_1\wedge\cdots\wedge X_k$ in $\Gamma(\Lambda^k(TM))$ and a differential form $\tau$, we define the contraction of $\tau$ by $X$ to be \[X\hk\tau:=X_k\hk\cdots\hk X_1\hk\tau,\] and extend by linearity to all multivector fields. We define the Lie derivative of $\tau$ in the direction of $X$ to be \begin{equation}\label{Lie}\L_X\tau:=d(X\hk\tau)-(-1)^{k}X\hk d\tau.\end{equation} Note that this is the usual Lie derivative when $k=1$.
\end{definition}Throughout the thesis we will make extensive use of the following propositions.

\begin{proposition}
\label{properties}
Let $X\in\Gamma(\Lambda^k(TM))$ and $Y\in\Gamma(\Lambda^l(TM))$ be arbitrary. For a differential form $\tau$, the following identities hold:

\begin{equation}\label{dL}d\L_X\tau=(-1)^{k+1}\L_Xd\tau\end{equation}

\begin{equation}\label{bracket hook}[X,Y]\hk\tau=(-1)^{(k+1)l}\L_X(Y\hk\tau)-Y\hk(\L_X\tau)\end{equation}

\begin{equation}\label{L bracket}\L_{[X,Y]}\tau=(-1)^{(k+1)(l+1)}\L_X\L_Y\tau-\L_Y\L_X\tau\end{equation}

\begin{equation}\label{L wedge}\L_{X\wedge Y}\tau=(-1)^lY\hk(\L_X\tau)+\L_Y(X\hk\tau)\end{equation}

\end{proposition}

\begin{proof}
A full proof is given in Proposition A.3 of \cite{poisson}. Note that equation (\ref{dL}) follows by taking the differential of both sides of (\ref{Lie}) and using that $d^2=0$. Equation (\ref{bracket hook}) is proved by induction on the tensor degrees of the multivector fields (see \cite{poisson}). Using equation (\ref{bracket hook}) we provide a proof of equation (\ref{L bracket}). We have that
\begin{align*}
\L_{[X,Y]}\tau&=d([X,Y]\hk\tau)+(-1)^{k+l}[X,Y]\hk d\tau\\
&=(-1)^{(k+1)l}d\L_X(Y\hk\tau)-d(Y\hk \L_X\tau)+(-1)^{k(l+1)}\L_X(Y\hk d\tau)-(-1)^{k+l}Y\hk\L_Xd\tau.
\end{align*}
Now using equation (\ref{dL}), this is equal to
\[(-1)^{(k+1)(l+1)}\L_X(d(Y\hk\tau))-d (Y\hk\L_X\tau)+(-1)^{k(l+1)}\L_X(Y\hk d\tau) +(-1)^lY\hk d\L_X\tau\]
which by (\ref{Lie}) is equal to
\[(-1)^{(k+1)(l+1)}\L_X\L_Y\tau -\L_Y\L_X\tau\] as desired. Lastly, equation (\ref{L wedge}) can be proved directly. We have
\begin{align*}
\L_{X\wedge Y}\tau&=d(Y\hk X\hk\tau)-(-1)^{k+l}Y\hk X\hk d\tau\\
&=d(Y\hk X\hk \alpha)-(-1)^lY\hk d(X\hk\tau)+(-1)^lY\hk d(X\hk\tau)-(-1)^{k+l}Y\hk X\hk d\tau\\
&=\L_Y(X\hk\tau)+(-1)^lY\hk \L_X\tau.
\end{align*}
\end{proof}

Another formula for the interior product by the Schouten bracket is given by the next proposition.
\begin{proposition}
\label{interior}For $X\in\Gamma(\Lambda^k(TM))$ and $Y\in\Gamma(\Lambda^l(TM))$ we have that interior product with their Schouten bracket satisfies
 \[i[X,Y]=[-[i(Y),d],i(X)],\] where the bracket on the right hand side is the graded commutator. Written out fully, this says that for an arbitrary form $\tau$,
 \begin{equation}\label{interior equation}[X,Y]\hk\tau=-Y\hk d(X\hk\tau)+(-1)^ld(Y\hk X\hk\tau)+(-1)^{kl+k}X\hk Y\hk d\tau-(-1)^{kl+k+l}X\hk d(Y\hk\tau)\end{equation}
\end{proposition}

\begin{proof}
This is Proposition 4.1 of \cite{marle}. It can also be derived directly from equations (\ref{Lie}) and (\ref{bracket hook}). We state it as a separate proposition because equation (\ref{interior equation}) will be used frequently in the rest of the thesis.
\end{proof}

Next we recall the Chevalley-Eilenberg complex. We start with a Lie algebra $\g$ and its exterior algebra $\Lambda^\bullet\g$. The Gerstenhaber algebra $(\Lambda^\bullet\g,\wedge,[\cdot,\cdot])$ is turned into a differential algebra by the following differential. 

\begin{definition} For a Lie algebra $\g$, consider the differential 

\[\partial_k:\Lambda^k\g\to\Lambda^{k-1}\g, \ \ \ \ \ \xi_1\wedge\cdots\wedge\xi_k\mapsto\sum_{1\leq i<j\leq k}(-1)^{i+j}[\xi_i,\xi_j]\wedge\xi_1\wedge\cdots\wedge\widehat\xi_i\wedge\cdots\wedge\widehat\xi_j\wedge\cdots\wedge\xi_k\] for $k\geq 1$, and extend by linearity to non-decomposables. Define $\Lambda^{-1}\g=\{0\}$ and $\partial_0$ to be the zero map. It follows from the graded Jacobi identity that $\partial^2=0$.
The differential Gerstenhaber algebra $(\Lambda^k\g,\wedge,\partial,[\cdot, \cdot])$ is called the Chevalley-Eilenberg complex.\end{definition}

\begin{definition}\label{Lie kernel}
We follow the terminology and notation of \cite{ms} and call $\Rho_{\g,k}=\ker \partial_k$ the $k$-th Lie kernel, which is a vector subspace of $\Lambda^k\g$. Let $\Rho_\g$ denote the direct sum of all the Lie kernels: \[\Rho_\g=\oplus_{k=0}^{\dim(\g)}\Rho_{\g,k}.\] Note that if the Lie algebra is abelian then $\Rho_{\g,k}=\Lambda^k\g$.
\end{definition}

A straightforward computation gives the following lemma.
\begin{lemma}\label{lemma formula}
For arbitrary $p\in\Lambda^k\g$ and $q\in\Lambda^l\g$ we have that 
\[\partial(p\wedge q)=\partial(p)\wedge q+(-1)^kp\wedge\partial(q)+(-1)^k[p,q].\]
\end{lemma}

From this lemma we get the following.
\begin{proposition} 
\label{Schouten is closed}
We have $(\Rho_{\g},\partial,[\cdot,\cdot])$ is a differential graded subalgebra of the Chevalley-Eilenberg complex, with $\partial=0$.
\end{proposition}

\begin{proof}
The only nontrivial thing we need to show is that the Schouten bracket preserves the Lie kernel. While this follows immediately from the fact that $\partial$ is a graded derivation of the Schouten bracket, we can also show it using Lemma \ref{lemma formula}. Indeed, for $p\in\Rho_{\g,l}$ and $q\in\Rho_{\g,l}$ we have that \begin{align*}\partial(p\wedge q)&=\partial(p)\wedge q +(-1)^k p\wedge\partial(q)+[p,q]\\
&=[p,q].\end{align*}Hence, $[p,q]$ is exact and therefore closed.
\end{proof}

We now define a differential graded Lie algebra consisting of multivector fields.

Let $G$ be a connected Lie group acting on a manifold $M$. For $\xi\in\g$ let $V_\xi\in\Gamma(TM)$ denote the infinitesimal generator of the induced action on $M$ by the one-parameter subgroup of $G$ generated by $\xi$.  For decomposable $p=\xi_1\wedge\cdots\wedge\xi_k$ in $\Lambda^k\g$ we introduce the notation $V_p:=V_{\xi_1}\wedge\cdots\wedge V_{\xi_k}$ for the associated multivector field, and extend by linearity. Let \begin{equation}\label{S_k}S_k=\{V_p \ :\  p\in\Rho_{\g,k}\}\end{equation} and set \begin{equation}S=\oplus_{k=0}^{\dim(\g)} S_k.\end{equation} 

\begin{proposition}
\label{infinitesimal generator of Schouten}
We have that $(S,[\cdot,\cdot])$ is a graded Lie algebra. Moreover, for $p\in\Lambda^k\g$, we have that $\partial V_p=-V_{\partial p}$. Note that we have abused notation and let $\partial$ denote the Chevalley-Eilenberg differentials for both the Lie algebras $(\Gamma(TM),[\cdot,\cdot])$ and $(\g,[\cdot,\cdot])$.\end{proposition}

\begin{proof}
We first show that $V_{[p,q]}=-[V_p,V_q]$. Let $p=\xi_1\wedge\cdots\wedge\xi_k$ and $q=\eta_1\wedge\cdots\wedge\eta_l$. Then we have that \begin{align*}
V_{[p,q]}&=\sum_{i,j}(-1)^{i+j}V_{[\xi_i,\eta_j]}\wedge V_{\xi_1}\wedge\cdots \widehat V_{\xi_i}\cdots\widehat V_{\eta_j}\wedge\cdots\wedge V_{\eta_l}\\
&=\sum_{i,j}-(-1)^{i+j}[V_{\xi_i},V_{\eta_j}]\wedge V_{\xi_1}\wedge\cdots \widehat V_{\xi_i}\cdots\widehat V_{\eta_j}\wedge\cdots \wedge V_{\eta_l}\\
&=-[V_p,V_q].
\end{align*}Now extend by linearity to all $p,q\in\Lambda^k\g.$
Here we used the standard result of group actions that $[V_\xi,V_\eta]=-V_{[\xi,\eta]}$. The first claim now follows since $(\Rho_\g,[\cdot,\cdot])$ is a graded Lie algebra by Proposition \ref{Schouten is closed}. Moreover, we have that 

\begin{align*}
\partial V_p&=\partial (V_{\xi_1}\wedge\cdots\wedge V_{\xi_k})\\
&=\sum_{1\leq i<j\leq k}(-1)^{i+j}[V_{\xi_i},V_{\xi_j}]\wedge V_{\xi_1}\wedge\cdots\wedge\widehat V_{\xi_i}\wedge\cdots\wedge\widehat V_{\xi_j}\wedge\cdots\wedge V_{\xi_k}\\
&=-\sum_{1\leq i<j\leq k}(-1)^{i+j}V_{[\xi_i,\xi_j]}\wedge V_{\xi_1}\wedge\cdots\wedge\widehat V_{\xi_i}\wedge\cdots\wedge\widehat V_{\xi_j}\wedge\cdots\wedge V_{\xi_k}\\
&=-V_{\partial p}
\end{align*}
Now extend by linearity to all $p\in\Lambda^k\g$. In particular then, if $p$ is in the Lie kernel, we have that $\partial V_p=-V_{\partial p}=0$.

\end{proof}

The following lemma will be used repeatedly in the rest of the thesis. We remark that it holds for arbitrary multivector fields; however, for our purposes it will suffice to consider the restriction to elements of $S$.

\begin{lemma}(\bf{Extended Cartan Lemma})\rm
\label{extended Cartan}
For decomposable $p=\xi_1\wedge\cdots\wedge\xi_k$ in $\Lambda^k\g$ and differential form $\tau$ we have that 

\begin{align*}
(-1)^kd(V_p\hk\tau)&=V_{\partial{p}}\hk\tau +\sum_{i=1}^k(-1)^i(V_{\xi_1}\wedge\cdots\wedge\widehat V_{\xi_i}\wedge\cdots\wedge V_{\xi_k})\hk\L_{V_{\xi_i}}\tau +V_p\hk d\tau.
\end{align*} 
\end{lemma}

\begin{proof}

This is Lemma 3.4 of \cite{ms} or Lemma 2.18 of \cite{cq}.
\end{proof}

Let $\Phi:G\times M\to M$ be a Lie group action on $M$. 
\begin{definition}
For $A\in\Gamma(TM)$ we let $\Phi_{g}^\ast A$ denote the vector field given by the push-forward of $A$ by $\Phi_{g}^{-1}$. That is, \[(\Phi_{g}^\ast(A))_x:=(\Phi_{g^{-1}})_{\ast,\Phi_{g}(x)}(A_{\Phi_{g(x)}}),\] where $x\in M$. For a decomposable multivector field $Y=Y_1\wedge\cdots \wedge Y_k$ in $\Gamma(\Lambda^k(TM))$ we will let $\mathrm{Ad}_gY$ denote the extended adjoint action \[\mathrm{Ad}_gY=\mathrm{Ad}_gY_1\wedge\cdots\wedge \mathrm{Ad}_gY_k\] and we will let $\Phi_g^\ast Y$ denote the multivector field \[\Phi_g^\ast Y=\Phi_g^\ast Y_1\wedge\cdots\wedge \Phi_g^\ast Y_k.\]We also extend $\mathrm{ad}$ to a map $\mathrm{ad}:\g\times \Lambda^k\g\to\Lambda^k\g$ by \begin{equation}\label{ad equation}\mathrm{ad}_\xi(Y_1\wedge\cdots\wedge Y_k)=\sum_{i=1}^kY_1\wedge\cdots\wedge \mathrm{ad}_{\xi}(Y_i)\wedge\cdots\wedge Y_k.\end{equation} For an arbitrary multivector field $Y\in\Gamma(\Lambda^k(TM))$, the above definitions are all extended by linearity. Notice that for arbitrary $p\in\Lambda^k\g$, we have $\mathrm{ad}_\xi(p)=[\xi,p]$.
\end{definition}

\del{
\begin{corollary}
\label{ad}
For $\xi\in\g$ we have that $\mathrm{ad}_\xi$ preserves the Lie kernel. That is, if $p$ is in $\Rho_{\g,k}$ then $\mathrm{ad}_{\xi}(p)$ is in $\Rho_{\g,k}$.
\end{corollary}
}
The next proposition shows that the infinitesimal generator of the extended adjoint action agrees with the pull back action.
\begin{proposition}
\label{adjoint over wedge}
Let $\Phi:G\times M\to M$ be a group action. For every $g\in G$ and $p\in\Lambda^k\g$ we have that \[V_{\mathrm{Ad}_gp}=\Phi_{g^{-1}}^\ast V_p.\] Equivalently, the map $\Lambda^k\g\to\Gamma(\Lambda^k(TM))$ given by $\xi_1\wedge\cdots\wedge\xi_k\mapsto V_{\xi_1}\wedge\cdots\wedge V_{\xi_k}$ is equivariant with respect to the extended adjoint and pull back action.
\end{proposition}

\begin{proof}
Fix $q\in M$, $g\in G$. First suppose that $\xi\in\g$. Then by Proposition 4.1.26 of \cite{Marsden} we have that \[V_{\mathrm{Ad}_g\xi}=\Phi_{g^{-1}}^\ast V_\xi.\] The claim now follows since for $p=\xi_1\wedge\cdots\wedge \xi_k$ in $\Lambda^k\g$,
\begin{align*}
V_{\mathrm{Ad}_gp}&:=V_{\mathrm{Ad}_g{\xi_1}}\wedge\cdots\wedge V_{\mathrm{Ad}_g{\xi_k}}\\
&=\Phi_{g^{-1}}^\ast V_{\xi_1}\wedge\cdots\wedge\Phi_{g^{-1}}^\ast V_{\xi_k}\\
&=\Phi_{g^{-1}}^\ast V_p&\text{by definition.}
\end{align*}

\end{proof}

While in this thesis we will mostly be concerned with differential graded Lie algebras, we will also have the need to consider the more general structure of an $L_\infty$-algebra.

\subsection{$L_\infty$-Algebras} We only state the definition of an $L_\infty$-algebra and do not go into detail. More detail can be found in \cite{rogers}, for example.

\begin{definition} An $L_\infty$-algebra is a graded vector space $L=\oplus_{i=-\infty}^\infty L_i$ together with a collection of  graded skew-symmetric linear maps  $\{l_k:L^{\otimes k}\to L \ ; \ k\geq 1\}$, with $\text{deg}(l_k)=k-2$, satisfying the following identity for all $m\geq 1$:

\[\sum_{\substack{ i+j=m+1\\ \sigma\in \text{Sh}(i,m-i)}}(-1)^\sigma\epsilon(\sigma)(-1)^{i(j-1)}l_j(l_i(x_{\sigma(1)},\ldots,x_{\sigma(i)}),x_{\sigma(i+1)},\ldots,x_{\sigma(m)})=0 .\] Here $\sigma$ is a permutation of $m$ letters, $(-1)^\sigma$ is the sign of $\sigma$, and $\epsilon(\sigma)$ is the Koszul sign. We will not have need for the Koszul sign in this thesis and direct the reader to \cite{rogers} for its definition. The subset Sh$(p,q)$ of permutations on $p+q$ letters is the set of $(p,q)$-unshuffles. A permutation $\sigma$ of $p+q$ letters is called a $(p,q)$-unshuffle if $\sigma(i)<\sigma(i+1)$ for $i\not=p$. 

An $L_\infty$-algebra $(L,\{l_k\})$ is called a Lie $n$-algebra if $L_i=0$ for $i\geq n$.

\end{definition}

$L_\infty$ algebras are not central to this thesis; however, they are needed to define the $L_\infty$ algebra of observables from \cite{rogers} and to state Theorem $1.6$. 

Since any differential graded Lie algebra is a Lie $n$-algebra (indeed, just take $l_1=\partial$, $l_2=[\cdot,\cdot]$ and $l_k=0$ for $k\geq 3$), Propositions \ref{Schouten is closed} and \ref{infinitesimal generator of Schouten} show that the spaces $S$ and $\Rho_\g$ have $L_\infty$-algebra structures.

Next we recall the basic notions from group and Lie algebra cohomology that will be needed in this thesis.

\subsection{Group Cohomology}
Let $G$ be a group and $S$ a $G$-module. For $g\in G$ and $s\in S$, let $g\cdot s$ denote the action of $G$ on $S$. Let $C^k(G,S)$ denote the space of smooth alternating functions from $G^k$ to $S$ and consider the differential $\delta_k:C^k(G,S)\to C^{k+1}(G,S)$ defined as follows. For $\sigma\in C^k(G,S)$ and $g_1,\cdots, g_{k+1}\in G$ define \begin{equation}\label{group differential}\begin{aligned}\delta&_k\sigma(g_1,\cdots,g_{k+1}):=\\
&g_1\cdot\sigma(g_2,\cdots,g_{k+1})+\sum_{i=1}^k(-1)^i\sigma(g_1,\cdots, g_{i-1},g_ig_{i+1},g_{i+2},\cdots,g_{k+1})-(-1)^k\sigma(g_1,\cdots,g_k).\end{aligned}\end{equation}A computation shows that $\delta^2=0$ so that $C^0(G,S)\to C^1(G,S)\to \cdots$ is a cochain complex.  This cohomology is known as the differentiable cohomology of $G$ with coefficients in $S$. We let $H^k(G,S)$ denote the $k$-th cohomology group and will call an equivalence class representative a $k$-cocycle.

\subsection{Lie Algebra Cohomology}
Let $\g$ be a Lie algebra and $R$ a $\g$-module. Given $\xi\in\g$ and $r\in R$, let $\xi\cdot r$ denote the action of $\g$ on $R$. We let $C^k(\g,R)$ denote the space of multilinear alternating functions from $\g^k$ to $R$ and consider the differential $\delta_k:C^k(\g,R)\to C^{k+1}(\g,R)$ defined as follows. For $f\in C^k(\g,R)$ and $\xi_1,\cdots\xi_{k+1}\in \g$ define
 \begin{equation}\label{group differential 2}\begin{aligned}\delta_k& f(\xi_1,\cdots,\xi_{k+1}):=\\
 &\sum_i(-1)^{i+1}\xi_i\cdot f(\xi_1,\cdots,\widehat\xi_i,\cdots,\xi_{k+1})+\sum_{i<j}(-1)^{i+j}f([\xi_i,\xi_j],\xi_1,\cdots,\widehat\xi_i,\cdots,\widehat\xi_j,\cdots,\xi_{k+1}). \end{aligned}\end{equation}

A computation shows that $\delta^2=0$. We let $H^k(\g,R)$ denote the $k$-th cohomology group and call an equivalence class representative a (Lie algebra) $k$-cocycle. Note that for $k=0$ the map $\delta_0:R\to  C^1(\g,R)$ is given by $(\delta_0r)(\xi)=\xi\cdot r$, where $r\in R$ and $\xi\in \g$. Thus, by definition, \[H^0(\g,R)=\{r\in R; \ \xi\cdot r=0 \text{ for all $\xi\in\g$}\}.\] For $k=1$ the map $\delta_1:C^1(\g,R)\to C^2(\g,R)$ is given by \[\delta_1(f)(\xi_1,\xi_2)=\xi_1\cdot f(\xi_2)-\xi_2\cdot f(\xi_1)-f([\xi_1,\xi_2]),\] where $f\in C^1(\g,R)$ and $\xi_1$ and $\xi_2$ are in $\g$.

The standard example of Lie algebra cohomology is given when $R=\R$:

\begin{example}\bf{(Exterior algebra of $\g^\ast$) }\rm  Consider the trivial $\g$-action on $\R$. Then $C^k(\g,\R)=\Lambda^k\g^\ast$, and the Lie algebra cohomology differential $\delta_k:\Lambda^k\g^\ast\to\Lambda^{k+1}\g^\ast$ is given by
\begin{equation}\label{differential g dual}\delta_k\alpha(\xi_1\wedge\cdots\wedge\xi_k):= \alpha\left(\sum_{1\leq i<j\leq k}(-1)^{i+j}[\xi_i,\xi_k]\wedge\xi_1\wedge\cdots\wedge\widehat\xi_i\wedge\cdots\wedge\widehat\xi_j\wedge\cdots\wedge\xi_k\right)\end{equation}where $\alpha\in\Lambda^k\g^\ast$, and $\xi_1\wedge\cdots\wedge\xi_k$ is a decomposable element of $\Lambda^k\g$, and extended by linearity to non-decomposables. It is easy to check that $\delta^2=0$. We will also make frequent reference to the corresponding Lie algebra homology differential which is given by

\begin{equation}\label{differential g}\partial_k:\Lambda^k\g\to\Lambda^{k-1}\g \ \ \ \ \ \ \ \ \xi_1\wedge\cdots\wedge\xi_k\mapsto\sum_{1\leq i<j\leq k}(-1)^{i+j}[\xi_i,\xi_k]\wedge\xi_1\wedge\cdots\wedge\widehat\xi_i\wedge\cdots\wedge\widehat\xi_j\wedge\cdots\wedge\xi_k,\end{equation}for $k\geq 1$. We define $\Lambda^{-1}\g=\{0\}$ and $\partial_0$ to be the zero map. 
\end{example} 
\del{
\begin{definition}\label{cohomology differential}
We follow the terminology and notation of \cite{ms} and call $\Rho_{\g,k}=\ker \partial_k$ the $k$-th Lie kernel, which is a vector subspace of $\Lambda^k\g$.  Notice that if $\g$ is abelian then $\Rho_{\g,k}=\Lambda^k\g$. We will let $\Rho_g$ denote the direct sum of all the Lie kernels: \[\Rho_\g=\oplus_{k=0}^{\dim(\g)}\Rho_{\g,k}.\] We will denote $H^k(\g,\R)$  simply by $H^k(\g)$.
\end{definition}
}

\newpage

\section{Multisymplectic Geometry}

In this section we develop the necessary background in multisymplectic geometry.

\subsection{Multi-Hamiltonian Systems}
Multisymplectic manifolds are the natural generalization of symplectic manifolds.
\begin{definition}
A manifold $M$ equipped with a closed $(n+1)$-form $\omega$ is called a pre-multisymplectic (or pre-$n$-plectic) manifold. If in addition the map $T_pM\to\Lambda^n T_p^\ast M,\  V\mapsto V\hk\omega$ is injective, then $(M,\omega)$ is called a multisymplectic  or $n$-plectic manifold.
\end{definition} 
\del{
The following proposition demonstrates the abundance of multisymplectic manifolds.

\begin{proposition}
A vector space of dimension $m$ admits an $n$-plectic form, with $n\geq 2$ if and only if $m\geq n+1$ and $m\not=n+2$. 
\end{proposition} 
\begin{proof}
This is proposition 2.6 of \cite{MS}.
\end{proof}
}
The next example is a generalization of the phase space in Hamiltonian mechanics. It comes up frequently in this thesis. 

\begin{example}\label{Multisymplectic Phase Space}\bf{(Multisymplectic Phase Space)}\rm

Let $N$ be a manifold and let $M=\Lambda^{k}(T^\ast N)$. Then $\pi:M\to N$ is a vector bundle over $N$ with canonical $k$-form $\theta\in\Omega^k(M)$ defined by \[\theta_{\mu_x}(Z_1,\cdots, Z_k) :=\mu_x(\pi_\ast(Z_1),\cdots,\pi_\ast(Z_k)), \] for $x\in N$, $\mu_x\in\Lambda^k(T^\ast_xN)$, and $Z_1,\cdots, Z_k \in T_{\mu_x}M$. The $(k+1)$-form $\omega\in\Omega^{k+1}(M)$ defined by $\omega=-d\theta$ is the canonical $(k+1)$-form. The pair $(M,\omega)$ is a $k$-plectic manifold. Notice that for $k=1$ we recover the usual symplectic structure on the cotangent bundle.
\end{example}
\begin{definition}

If for $\alpha\in\Omega^{n-1}(M)$ there exists $X_\alpha\in\Gamma(TM)$ such that $d\alpha=-X_\alpha\hk \omega$ then we call $\alpha$ a Hamiltonian $(n-1)$-form and $X_\alpha$ a corresponding Hamiltonian vector field to $\alpha$. We let $\Omega^{n-1}_{\text{Ham}}(M)$ denote the space of Hamiltonian $(n-1)$-forms. \end{definition}

\begin{remark}
If $\omega$ is $n$-plectic then the Hamiltonian vector field $X_\alpha$ is unique. If $\omega$ is pre-$n$-plectic then Hamiltonian vector fields are unique up to an element in the kernel of $\omega$. Also, notice that in the $1$-plectic (i.e. symplectic) case, every function is Hamiltonian.
\end{remark}

\begin{definition} In analogy to Hamiltonian mechanics, for a fixed $n$-plectic form $\omega$ and Hamiltonian $(n-1)$-form $H$, we call $(M,\omega,H)$ a multi-Hamiltonian system. We denote the Hamiltonian vector field of $H$ by $X_H$.
 \end{definition}

There are many examples of multi-Hamiltonian systems and we refer the reader to Section 3.1 of \cite{cq} for some results on their existence. 
\del{

\begin{example}
A orientable manifold $M$ with a volume form $\mu$. A corresponding Hamiltonian could be...We could take the acting group to be the volume preserving diffeomorphisms. 
\end{example}

\begin{example}
A G2 manifold where the corresponding Hamiltonian is...
\end{example}
We will study in more detail the following example

\begin{example}

The cotangent bundle with Hamiltonian

\end{example}

}

 In \cite{rogers} it was shown that to any multisymplectic manifold one can associate the following $L_\infty$-algebra.
\begin{definition}\label{Lie n observables}
The Lie $n$-algebra of observables, $L_\infty(M,\omega)$ is the following $L_\infty$-algebra. Let $L=\oplus_{i=0}^nL_i$ where $L_0=\Omega^{n-1}_{\text{Ham}}(M)$ and $L_i=\Omega^{n-1-i}(M)$ for $1\leq i\leq n-1$. The maps $l_k:L^{\otimes k}\to L$ of degree $k-2$ are defined as follows: For $k=1$,

\[l_1(\alpha) =\left\{\begin{array}{l}d\alpha \ \ \ \ \ \text{if deg $\alpha>0$},\\
0 \ \ \ \ \ \ \ \text{if deg $\alpha=0$.}\end{array}\right.\]

For $k>1$,
\[l_k(\alpha_1,\cdots,\alpha_k) =\left\{\begin{array}{l}\zeta(k)X_{\alpha_k}\hk\cdots X_{\alpha_1}\hk\omega \ \ \ \ \ \text{if deg $\alpha_1\otimes\cdots\otimes\alpha_k=0$},\\
0 \ \ \ \ \ \ \ \ \ \ \ \ \ \ \ \ \ \ \ \ \ \ \ \ \ \ \ \ \ \ \text{if deg $\alpha_1\otimes\cdots\otimes\alpha_k>0$.}\end{array}\right.\]
Here $\zeta(k)$ is defined to equal $-(-1)^{\frac{k(k+1)}{2}}$. We introduce this notation as this sign comes up frequently.

\end{definition}

\begin{remark}\label{ugly signs}
It is easily verified that $\zeta(k)\zeta(k+1)=(-1)^{k+1}$. For future reference we also note that $\zeta(k)\zeta(l)\zeta(k+l-1)=-(-1)^{k+l+kl}$ and $\zeta(k)\zeta(l)=-(-1)^{lk}\zeta(k+l).$
\end{remark}

The following lemma from \cite{rogers} will be useful later on in the thesis.

\begin{lemma}
\label{rogers 3.7}
Let $\alpha_1,\ldots,\alpha_m\in\Omega^{n-1}_{\text{Ham}}(M)$ be arbitrary Hamiltonian $(n-1)$-forms on a multisymplectic manifold $(M,\omega)$. Let $X_1,\ldots, X_m$ denote the associated Hamiltonian vector fields. Then
\[d(X_m\hk\cdots\hk X_1\hk\omega)=(-1)^m\sum_{1\leq i<j\leq m}(-1)^{i+j}X_m\hk\cdots\hk\widehat X_j\hk\cdots\hk\widehat X_i\hk\cdots\hk X_1\hk[X_i,X_j]\hk\omega.\]
\end{lemma}

\begin{proof}
This is Lemma 3.7 of \cite{rogers}.
\end{proof}
Lastly we recall the terminology for group actions on a multisymplectic manifold. 

\begin{definition}
A Lie group action $\Phi:G\times M\to M$ is called multisymplectic if $\Phi_g^\ast\omega=\omega$. A Lie algebra action $\g\times \Gamma(TM)\to \Gamma(TM)$ is called multisymplectic if $\L_{V_\xi}\omega=0$ for all $\xi\in\g$. We remark that a multisymplectic Lie group action induces a multisymplectic Lie algebra action. Conversely, a multisymplectic Lie algebra action induces a multisymplectic group action if the Lie group is simply-connected.  Moreover, as in \cite{cq}, we will call a Lie group action on a multi-Hamiltonian system $(M,\omega,H)$ locally, globally, or strictly $H$-preserving if it is multisymplectic and if $\L_{V_\xi} H$ is closed, exact, or zero respectively, for all $\xi\in\g$. 
\end{definition}

\subsection{Weak Homotopy Moment Maps}
For a group acting on a symplectic manifold $M$, a moment map is a Lie algebra morphism between $(\g,[\cdot,\cdot])$ and  $(C^\infty(M),\{\cdot,\cdot\})$, where $\{\cdot,\cdot\}$ is the Poisson bracket. In multisymplectic geometry, the $n$-plectic form no longer provides a Lie algebra structure on the space of smooth functions. However, as we saw in the previous section, the $n$-plectic structure does define an $L_\infty$-algebra, namely the Lie $n$-algebra of observables. A homotopy moment map is an $L_\infty$-morphism from $\g$ to the Lie $n$-algebra of observables. We explain what this means in the following definition and refer the reader to \cite{questions} for further information on $L_\infty$-morphisms.

For the rest of this section, we assume a multisymplectic action of a Lie algebra $\g$ on $(M,\omega)$.
\begin{definition}\label{fhmm}
A (homotopy) moment map is an $L_\infty$-morphism $(f)$ between $\g$ and the Lie $n$-algebra of observables. This means that $(f)$ is a collection of maps $f_1:\Lambda^1\g\to\Omega^{n-1}_{\text{Ham}}(M)$ and $f_k:\Lambda^k\g\to\Omega^{n-k}(M)$ for $k\geq 2$ satisfying, \[df_1(\xi)=-V_\xi\hk\omega\] for $\xi\in\g$ and  \begin{equation}\label{fhcmm}-f_{k-1}(\partial p)=df_k(p)+\zeta(k)V_p\hk\omega,\end{equation} for $p\in\Lambda^k\g$ and $k\geq 1$. A moment map is called equivariant if each component $f_i:\Lambda^i\g\to\Omega^{n-i}(M)$ is equivariant with respect to the adjoint and pullback actions respectively.
\end{definition}

\begin{definition} A weak (homotopy) moment map, is a collection of maps $(f)$ with $f_k:\Rho_{\g,k}\to\Omega_{\mathrm{Ham}}^{n-k}(M)$ satisying \begin{equation}\label{wmm kernel}df_k(p)=-\zeta(k)V_p\hk\omega.\end{equation}
\end{definition}

\begin{remark}
Notice that any collection of functions satisfying $(\ref{fhcmm})$ also satisfies $(\ref{wmm kernel})$. That is, any homotopy moment map induces a weak homotopy moment map. Moreover, the multi-moment maps of Madsen and Swann (in \cite{ms} and \cite{MS}) are given precisely by the $n$-th component of our weak homotopy moment maps. Lastly, notice that when $n=1$,  both full and weak homotopy moment maps reduce to the standard definition in symplectic geometry.
\end{remark}

We conclude this section with some examples of weak moment maps. In Hamiltonian mechanics, the phase space of a manifold $M$ is the symplectic manifold $(T^\ast M,\omega=-d\theta)$. The next example generalizes this to the setting of multisymplectic geometry.
 
 \begin{example}\bf{(Multisymplectic Phase Space)}\rm
 \label{multisymplectic phase space}
 As in Example \ref{Multisymplectic Phase Space}, let $N$ be a manifold and let $M=\Lambda^{k}(T^\ast N)$, with $\pi:M\to N$ the projection map. Let $\theta$ and $\omega=-d\theta$ denote the canonical $k$ and $(k+1)$-forms respectively.  Let $G$ be a group acting on $N$ and lift this action to $M$ in the standard way. Such an action on $M$ necessarily preserves $\theta$. We define a weak homotopy moment map by \[f_l(p):=-\zeta(l+1)V_p\hk\theta,\]for $p\in\Rho_{\g,l}$.

We now show that $(f)$ is a weak homotopy moment map. For $l\geq 1$, first consider a decomposable element $p=A_1\wedge\cdots\wedge A_l$ in $\Lambda^l\g$. Then, using Lemma \ref{extended Cartan} and the fact that $G$ preserves $\omega$, we find

\begin{align*}
df_l(p)&=-\zeta(l+1)d(V_p\hk\theta)\\
&=-\zeta(l+1)(-1)^{l}\left(\partial V_p\hk\theta +\sum_i^l(-1)^iA_1\wedge\cdots\wedge\widehat A_i\wedge\cdots\wedge A_l\hk \L_{A_i}\theta-V_p\hk\omega\right)\\
&=\zeta(l)\left(\partial V_p\hk\theta -V_p\hk\omega\right).
\end{align*}

By linearity, we thus see that this equation holds for an arbitrary element in $\Lambda^l\g$. That is, for all $p\in\Lambda^l\g$, \[df_l(p)=\zeta(l)\left(\partial V_p\hk\theta -V_p\hk\omega\right).\] If we assume now that $p\in\Rho_{\g,l}$, it then follows from Proposition \ref{infinitesimal generator of Schouten} that  \[df_l(p)=-\zeta(l)V_p\hk\omega.\] Thus by equation $(\ref{wmm kernel})$ we see $(f)$ is a weak homotopy moment map.
\end{example}

\begin{remark}
In symplectic geometry, symmetries on the phase space $T^\ast M$ have an important relationship with the classical momentum and position functions (see Chapter 4.3 of \cite{Marsden}). These momentum and position functions satisfy specific commutation relations which play an important role in connecting classical and quantum mechanics. Once we extend the Poisson bracket to multisymplectic manifolds and discuss a generalized notion of symmetry, we will come back to this multisymplectic phase space and give a generalization of these classical momentum and position functions (see Section 6.1).
\end{remark}
 
The next two examples will be used when we look at manifolds with a torsion-free $G_2$ structure.

\begin{example}\bf{($\C^3$ with the standard holomorphic volume form)}\rm
\label{complex moment map 1}
Consider $\C^3$ with standard coordinates $z_1,z_2,z_3$. Let $\Omega=dz_1\wedge dz_2\wedge dz_3$ denote the standard holomorphic volume form. Let \[\alpha=\text{Re}(\Omega)=\frac{1}{2}(dz_1\wedge dz_2\wedge dz_3+d\o z_1\wedge d\o z_2\wedge d\o z_3).\] It follows that $\alpha$ is a $2$-plectic form on $\C^3$.  We consider the diagonal action by the maximal torus $T^2\subset SU(3)$ given by $(e^{i\theta},e^{i\eta})\cdot(z_1,z_2,z_3)=(e^{i\theta}z_1,e^{i\eta}z_2,e^{-i(\theta+\eta)}z_3)$. We have $\mathfrak{t}^2=\R^2$ and that the infinitesimal generators of $(1,0)$ and $(0,1)$ are \[A=\frac{i}{2}\left(z_1\frac{\pd}{\pd z_1}-z_3\frac{\pd}{\pd z_3}-\o z_1\frac{\pd}{\pd\o z_1}+\o z_3\frac{\pd}{\pd \o z_3}\right)\]and \[B=\frac{i}{2}\left(z_2\frac{\pd}{\pd z_2}-z_3\frac{\pd}{\pd z_3}-\o z_2\frac{\pd}{\pd\o z_2}+\o z_3\frac{\pd}{\pd \o z_3}\right)\] respectively.

A computation then shows that 
\[A\hk\alpha=\frac{1}{2}d(\text{Im}(z_1z_3dz_2))\]
and
\[B\hk\alpha=\frac{1}{2}d(\text{Im}(z_1z_3dz_1)).\]

Moreover,
\[B\hk A\hk\alpha=-\frac{1}{4}d(\text{Re}(z_1z_2z_3))\]
Since $G=T^2$ is abelian we have that $\Rho_{\g,2}=\Lambda^2\g$. Thus, by equation $(\ref{wmm kernel})$ we see a weak homotopy moment map is given by

\[f_1(A)=\frac{1}{2}(\text{Im}(z_1z_3dz_2)) \ \ \ \ \ \ \ \ \ \ f_1(B)=\frac{1}{2}(\text{Im}(z_1z_3dz_1))\]
and
\[f_2(A\wedge B)=\frac{1}{4}(\text{Re}(z_1z_2z_3)).\]

If instead of Re$(\Omega)$ we were to consider Im$(\Omega)$ then in the above expressions for $f_1$ and $f_2$ we would just swap the roles of Re and Im.
\end{example}

\begin{example}\bf{($\C^3$ with the standard Kahler form)}\rm
\label{complex moment map 2}
Working with the same set up as Example \ref{complex moment map 1}, now consider the standard Kahler form $\omega=\frac{i}{2}(dz_1\wedge d\o z_1+dz_2\wedge d\o z_2+dz_3\wedge d\o z_3)$. This is a $1$-plectic (i.e. symplectic) form on $\C^3$.  A computation shows that \[A\hk\omega=-\frac{1}{4}d(|z_1|^2-|z_3|^2)\] and \[B\hk\omega=-\frac{1}{4}d(|z_2|^2-|z_3|^2).\] Thus, by equation $(\ref{wmm kernel})$ a weak homotopy moment map is given by

\[f_1(A)=-\frac{1}{4}(|z_1|^2-|z_3|^2) \ \ \ \ \ \ \ \ \ \ f_1(B)=-\frac{1}{4}(|z_2|^2-|z_3|^2).\]

\end{example}

For our last example of this section, we consider a multi-Hamiltonian system which models the motion of a particle, with unit mass, under no external net force. 
\begin{example}\label{translation}\bf{(Motion in a conservative system under translation)}\rm

Consider $\R^3$ with the standard metric $g$ and standard coordinates $q^1$, $q^2$, $q^3$. Let $q^1$, $q^2$, $q^3$, $p_1$, $p_2$, $p_3$ denote the induced coordinates on $T^\ast\R^3=\R^6$.
The motion of a particle in $\R^3$, subject to no external force, is given by a geodesic. That is, the path $\gamma$ of the particle is an integral curve for the geodesic spray $S$, a vector field on $T\R^3=\R^6$.  Using the metric to identify $T\R^3$ and $T^\ast\R^3$, the geodesic spray is given by \[S=g^{kj}p_j\frac{\pd}{\pd q^k}-\frac{1}{2}\frac{\pd g^{ij}}{\pd q^k}p_ip_j\frac{\pd}{\pd p_k},\] as shown in Example 5.21 of \cite{me}. This vector field $S$ is the standard Hamiltonian vector field on the phase space.  Since we our working with the standard metric, the geodesic spray is just \[S=\sum_{i=1}^3p_i\frac{\pd}{\pd q^i}.\] 

Let $M=T^\ast\R^3=\R^6$ and consider the multi-Hamiltonian system $(M,\omega,H)$ where \[\omega=\text{vol}=dq^1dq^2dq^3dp_1dp_2dp_3\] is the canonical volume form, and \[H=\frac{1}{2}\left((p_1q^2dq^3-p_1q^3dq^2)+(p_2q^1dq^3-p_2q^3dq^1)+(p_3q^1dq^2-p_3q^2dq^1)\right)dp_1dp_2dp_3.\] Then \[S\hk\omega=dH\] so that the $X_H=S$. That is, the Hamiltonian vector field in this multi-Hamiltonian system is the geodesic spray.  Consider the translation action of $G=\R^3$ on $\R^3$ and pull this back to an action on $M$. The infinitesimal generators of $e_1,e_2,e_3$ on $M$ are $\frac{\pd}{\pd q^1},\frac{\pd}{\pd q^2}$ and $\frac{\pd}{\pd q^3}$ respectively.  We compute the moment map for this action:

Since $\frac{\pd}{\pd q^1}\hk\omega=dq^2dq^3dp_1dp_2dp_3$ it follows that  $f_1(e_1)=\frac{1}{2}(q^2dq^3-q^3dq^2)dp_1dp_2dp_3$ satisfies $df(e_1)=V_{e_1}\hk\omega$. Similar computations show that the following is a homotopy moment map for the translation action on $(M,\omega,H)$:

\[f_1(e_1)=\frac{1}{2}(q^2dq^3-q^3dq^2)dp_1dp_2dp_3, \] \[ f_1(e_2)=\frac{1}{2}(q^1dq^3-q^3dq^1)dp_1dp_2dp_3, \] \[f_1(e_3)=\frac{1}{2}(q^1dq^2-q^2dq^1)dp_1dp_2dp_3,\]
\[f_2(e_1\wedge e_2)=q^3dp_1dp_2dp_3, \ \  \ \ \ \ f_2(e_1\wedge e_3)=q^2dp_1dp_2dp_3, \ \ \ \ \ \ f_2(e_2\wedge e_3)=q^1dp_1dp_2dp_3,\]and
\[f_3(e_1\wedge e_2\wedge e_3)=\frac{1}{3}\left(p_1dp_2dp_3+p_2dp_3dp_1+p_3dp_1dp_2\right).\]

\end{example}
\begin{remark}
In Section 5.3 we will  come back to Example \ref{translation} and consider the multisymplectic symmetries and conserved quantities coming from this homotopy moment map.
\end{remark}

\newpage
\section{Existence and Uniqueness of Weak Moment Maps}

In this section we show that the classical results on the existence and uniqueness of moment maps in symplectic geometry generalize directly to weak homotopy moment maps in multisymplectic geometry. In particular, we show that their existence and uniqueness is governed by a Lie algebra cohomology complex which reduces to the Chevalley-Eilenberg complex in the symplectic setup. \del{In Section ?? we show will study the morphism properties of weak moment maps (see Propositions ?? and ??).}

\subsection{Equivariance in Multisymplectic Geometry } 

\del{We first extend the cohomological characterization of equivariance of moment maps in symplectic to multisymplectic geometry.}

\del{ 

Recall the definition of a homotopy moment map:

\begin{definition}\label{hmm}
A (homotopy) moment map is an $L_\infty$-morphism $(f)$ between $\g$ and the Lie $n$-algebra of observables. This means that $(f)$ is a collection of maps $f_1:\Lambda^1\g\to\Omega^{n-1}_{\text{Ham}}(M)$ and $f_k:\Lambda^k\g\to\Omega^{n-k}(M)$ for $k\geq 2$ satisfying, for $p\in\Lambda^k\g$ \begin{equation}\label{hcmm}-f_{k-1}(\partial p)=df_k(p)+\zeta(k)V_p\hk\omega.\end{equation}
\end{definition}

It follows immediately from equation (\ref{hcmm}) that if $p$ is in $\Rho_{\g,k}$ then $f_k(p)$ is a Hamiltonian form. That is, if the domain of a homotopy moment map $(f)$ is restricted to the Lie kernel, then the image of $(f)$ is completely contained in the space of Hamiltonian forms. This motivates the definition of a weak homotopy moment map:

Many examples of homotopy moment maps are given in \cite{??}. For example they are explicity afadksfjds; invariant foms, cotangent lifts...
\begin{definition} A weak (homotopy) moment map, is a collection of maps $(f)$ with $f_k:\Rho_{\g,k}\to\Omega_{\mathrm{Ham}}^{n-k}(M)$ satisying \begin{equation}\label{hcmm kernel}df_k(p)=-\zeta(k)V_p\hk\omega.\end{equation}
\end{definition}
\del{
Notice that equation (\ref{wmm kernel}) is directly analogous to the defining equation of a moment map in symplectic geometry. Moreover, it was shown in \cite{cq} that elements in the Lie kernel get mapped to conserved quantities, just like in symplectic geometry. Moreover, in \cite{me} it was shown that if one consider weak moment maps, then one can obtain a generalization of Noether's theorem to multisymplectic geometry.
}

\begin{remark}
Notice that any moment map gives a weak moment map. Indeed, if $(f)$ satisfies equation (\ref{hcmm}) then it satisfies equation (\ref{wmm kernel}).
\end{remark}
\begin{remark}
Notice that a weak homotopy moment map coincides with the moment map from symplectic geometry in the case $n=1$. Indeed, setting $n=1$ in equations (\ref{hcmm}) and (\ref{wmm kernel}) yields $f:\g\to C^\infty(M)$ such that $df(\xi)=-V_\xi\hk\omega$. Also notice that the $n$-th component of a weak moment map is precisely the moment map introduced by Madsen and Swann in \cite{ms} and \cite{MS}.
\end{remark}

The next proposition says that a weak moment map is still an $L_\infty$-morphism.
\begin{proposition}
A weak moment map is an $L_\infty$-morphism from $\g$ to $\widehat L_\infty(M,\omega)$.
\end{proposition}

\begin{proof}
This is Proposition 5.9 of \cite{me}.
\end{proof}

}
\begin{definition}
A homotopy moment map is called equivariant if each component $f_k:\Lambda^k\g\to\Omega^{n-k}(M)$ is equivariant with respect to the adjoint and pullback actions respectively. That is, $(f)$ is equivariant if for all $g\in G$, $p\in\Lambda^k\g$, and $1\leq k\leq n$ \begin{equation}\label{equivariant equation}f_k(\mathrm{Ad}_{g^{-1}}^\ast p)=\Phi_g^\ast f_k(p).\end{equation} Similarly, a weak homotopy moment map is equivariant if equation (\ref{equivariant equation}) holds for all $p\in\Rho_{\g,k}$.
\end{definition}

\del{

\begin{definition}
A  pre homotopy comoment maps is an $L_\infty$ morphism $(f)$ between the Chevalley Eilenberg complex and the Lie n-algebra of observables. That is, $(f)$ is a collection of maps $f_1:\Lambda^1\g\to\Omega^{n-1}_{\text{Ham}}(M)$ and $f_k:\Lambda^k\g\to\Omega^{n-k}(M)$ for $k\geq 2$ satisfying, for $p\in\Lambda^k\g$ \[-f_{k-1}(\partial p)=df_k(p)+(-1)^{k(k+1)}V_p\hk\omega.\]
\end{definition}

\begin{definition} We say a pre homotopy moment map is equivariant if for each $k$ we have \[f_k(p)=\Phi_{g^{-1}}^\ast(f_k(Ad_{g^{-1}}(p)))\] where $g$ is an arbitrary element of the group. We will call a pre homotopy moment map that is equivariant a homotopy moment map.
\end{definition}

Let us first point out aspects of this definition that differ from those in Hamiltonian mechanics. Fix a multi Hamiltonian system $(M,\omega, H)$. In symplectic geometry,  it is always true that $\L_{X_H}H=0$. However, this is no longer valid in the multisymplectic setup as the following examples shows.

\begin{example}
On $\R^3$, let $\omega=dx\wedge dy\wedge dz$ and let $H=xdx+ydz$. Then $dH=dy\wedge dz$ and so $X_H=\frac{\partial}{\partial x}$.  It follows that $X_H\hk H=x$ and $X_H\hk dH=0$, and so $\L_{X_H}H=dx$
\end{example}

Moreover, in the symplectic setup, given a moment map $\mu:\g\to C^\infty(M)$ it is always true that $\mu(\xi)$ is a conserved quantity. That is, $\L_{X_H}\mu(\xi)=0$. However, this result is not true in the multiysymplectic setup. Recall from \cite{cq} that a $k$-form $\alpha$ can be conserved with respect to $X_H$ in three different way. Namely, $\alpha$ is locally, globally, strictly conserved if either $\L_{X_H}\alpha$ is exact, closed or zero respectively. The next example shows that an element in the image of a moment map may not be conserved in any of the three ways just mentioned.

\begin{example}
To do
\end{example}
We now refine the definition of a homotopy moment map, and by doing so will resolve the two issues mentioned above.

Firstly, we restrict the domain of a moment map to the Lie kernel. It now follows, by the results in \cite{cq}, that every element in the image of a moment map is a conserved quantity. Summarizing the results from \cite{cq} we have that if the action globally or locally preserves $H$ then $f_k(p)$ is locally conserved. If the action strictly preserves $H$, then $f_k(p)$ is globally conserved.

The final restriction is that we require the image of a moment map to be contained in the algebra of special forms. By definition, the image is contained in the Hamiltonian forms since $X_{f_k(p)}=V_p$; however, each $f_k(p)$ is not necessarily special.

\begin{example}
Moment map where $f_k(p)$ is not special.
\end{example}

By adding in the constraint of special forms, we see that now $\L_{X_H}H$ does equal zero by proposition \ref{preserve omega}. Summarizing, our new definition is

\begin{definition}
A refined homotopy moment map is a collection of maps $(f)=\{f_k:\Rho_{\g,k}\to\Omega^{n-k}_S(M)$ satisfying \[d(f_k(p))=-(-1)^{k(k+1)}V_p\hk\omega\] for $k=1,\ldots,n+1$ and $p\in\Rho_{\g,k}$. Equivalently, this means that \[(V_p-X_{f_k(p)})\hk\omega=0.\]As before, the collection $(f)$ is called equivariant if $f_k(p)=\Phi_{g^{-1}}^\ast(f_k(Ad_{g^{-1}}(p)))$ for each $k$.
\end{definition}

\begin{remark}
Here we are talking about equivariance with respect to the natural action on $\Lambda^k\g\otimes\Omega^{n-k}_{\text{Ham}}$ induced from the action of $G$ on $M$. If $g\cdot$ were an arbitrary action on $\Lambda^k\g\otimes\Omega^{n-k}_{\text{Ham}}$ then we would say the moment map is equivariant if \[f_k(p)=g\cdot f_k(p).\]See proposition \ref{new action} below.
\end{remark}

}

\del{
\subsection{Hamiltonian forms}

The following definition extends the notion of a Hamiltonian function from symplectic geometry.
\begin{definition}

If for $\alpha\in\Omega^{n-1}(M)$ there exists $X_\alpha\in\Gamma(TM)$ such that $d\alpha=X_\alpha\hk \omega$ then we call $\alpha$ a Hamiltonian $(n-1)$-form and $X_\alpha$ its corresponding Hamiltonian vector field. Notice that if $\omega$ is multisymplectic then the Hamiltonian vector field is unique. We let $\Omega^{n-1}_{\text{Ham}}(M)$ denote the space of Hamiltonian $(n-1)$ forms.
\end{definition}

We now extend the notion of Hamiltonian $(n-1)$-forms.

\begin{definition}We call
 \[\Omega^{n-k}_{\text{Ham}}:=\{\alpha\in\Omega^{n-k}(M) ; \text{ there exists $X_\alpha\in\Gamma(\Lambda^k(TM))$ with $d\alpha=-X_\alpha\hk\omega$ } \} \]the set of Hamiltonian $(n-k)$ forms. We call $X_\alpha$ the (multi) Hamiltonian vector field associated to $\alpha$.\end{definition} Of course, multi Hamiltonian vector fields are not necessarily unique. However, it's clear that they differ by something in the kernel of $\omega$.
\begin{proposition}
If $\alpha\in\Omega^{n-k}_{\text{Ham}}$ then any two of its Hamiltonian vector fields differ by something in the kernel of $\omega$.
\end{proposition}
}

\del{

\subsection{A Generalized Poisson Bracket}

We now put in a structure that represents a generalized Poisson bracket.
\begin{definition}
Given $\alpha\in\Omega^{n-k}_{\text{Ham}}$ and $\beta\in\Omega^{n-l}_{\text{Ham}}$ we define their Poisson bracket to be \[\{\alpha,\beta\} := (-1)^{l+1}X_\alpha\hk X_\beta\hk \omega\]
\end{definition}

\begin{remark}
The sign choice in the above proposition is so that the Hamiltonian forms modulo closed forms constitutes a graded Lie algebra 
\end{remark}

We first show that the space of Hamiltonian vector fields is closed under the Poisson bracket.

\begin{lemma}
\label{Poisson is Schouten}
For $\alpha\in\Omega^{n-k}_{\text{Ham}}$ and $\beta\in\Omega^{n-l}_{\text{Ham}}$ we have that $\{\alpha,\beta\}$ is in $\Omega^{n+1-k-l}_{\text{Ham}}(M)$. More precisely, we have that (check signs) \[X_{\{\alpha,\beta\}}=(-1)^{l}[X_\alpha,X_\beta]\]
\end{lemma}

\begin{proof}
By lemma \ref{interior} we have that 
\begin{align*}
[X_\alpha,X_\beta]\hk\omega&=-X_\beta\hk(d(X_\alpha\hk\omega))+(-1)^ld(X_\alpha\hk X_\beta\hk\omega)-(-1)^{k(l+1)}X_\alpha\hk(X_\beta\hk d\omega)-(-1)^{(k+1)(l+1)}X_\alpha\hk(d(X_\beta\hk\omega))\\
&=(-1)^ld(X_\alpha\hk X_\beta\hk\omega)\\
&=(-1)^ld(\{\alpha,\beta\})
\end{align*}

\end{proof}

\del{
Notice also that each Hamiltonian vector field preserves $\omega$.

\begin{proposition}
\label{preserve omega}
If $X\in\Gamma(\Lambda^k(TM))$ is a Hamiltonian vector field, then $\L_X\omega=0$
\end{proposition}
\begin{proof}
By definition, there exists some $n-k$-form $\alpha$ such that $X\hk\omega=d\alpha$.Thus \[\L_X\omega=d(X\hk\omega)=d^2\alpha=0\]
\end{proof}
}
The bracket commutes in a graded fashion

\begin{proposition}
For $\alpha\in\Omega^{n-k}_{\text{Ham}}(M)$ and $\beta\in\Omega^{n-l}_{\text{Ham}}(M)$ we have that \[\{\alpha,\beta\}=(-1)^{(n-k)(n-l)}\{\beta,\alpha\}\]
\end{proposition}

\begin{proof}
This follows since $\omega$ is skew symnetric. 
\end{proof}
It's not true in general that our bracket satisfies the Jacobi identity; however, as the next proposition shows, it is satisfied up to an exact term. 

\begin{proposition}
\label{Jacobi} Fix $\alpha\in\Omega^{n-k}_{\text{Ham}}(M)$, $\beta\in\Omega^{n-l}_{\text{Ham}}(M)$ and $\gamma\in\Omega^{n-p}_{\text{Ham}}(M)$. Let $X_\alpha, X_\beta$ and $X_\gamma$ denote Hamiltonian vector fields for $\alpha,\beta$ and $\gamma$ respectively. Then we have that
\[(-1)^{(n-k)(n-p)}\{\alpha,\{\beta,\gamma\}\}+(-1)^{(n-l)(n-k)}\{\beta,\{\gamma,\alpha\}\}+(-1)^{(n-p)(n-l)}\{\gamma,\{\alpha,\beta\}\} = d(\varphi(X_\alpha,X_\beta,X_\gamma)).\] 
\end{proposition}
\begin{proof}
(May do this proof in the other paper, and cite results. Still missing the negative signs)
First note that 
\begin{align*}
\{\alpha,\beta\}&=(X_\alpha\hk\varphi)(X_\beta,\cdot)\\
&=d\alpha(X_\beta,\cdot)\\
&=\L_{X_\beta}\alpha-d(X_\beta\hk\alpha)\\
&=-\L_{X_\alpha}\beta+d(X_\alpha\hk\beta)&\text{since $\{\alpha,\beta\}=-\{\beta,\alpha\}$}
\end{align*}
Thus 
\begin{align*}
\{\alpha,\{\beta,\gamma\}\}&=-\L_{X_\alpha}\{\beta,\gamma\}+d(X_\alpha\hk\{\beta,\gamma\})\\
&=-\L_{X_\alpha}(-\L_{X_\beta}\gamma+d(X_\beta\hk\gamma))+d(X_\alpha\hk\{\beta,\gamma\})\\
&=\L_{X_\alpha}\L_{X_\beta}\gamma-\L_{X_\alpha}d(X_\beta\hk\gamma)+d(\varphi(X_\alpha,X_\beta,X_\gamma))
\end{align*}
and
\begin{align*}
\{\beta,\{\gamma,\alpha\}\}&=-\L_{X_\beta}(\{\gamma,\alpha\})+d(X_\beta\hk\{\gamma,\alpha\})\\
&=-\L_{X_\beta}(\L_{X_\alpha}\gamma-d(X_\alpha\hk\gamma))+d(X_\beta\hk\{\gamma,\alpha\})\\
&=-\L_{X_\beta}\L_{X_\alpha}\gamma+\L_{X_\beta}d(X_\alpha\hk\gamma)+d(\varphi(X_\alpha,X_\beta,X_\gamma))
\end{align*}
and
\begin{align*}
\{\gamma,\{\alpha,\beta\}\}&=-\{\{\alpha,\beta\},\gamma\}\\
&=\L_{X_{\{\alpha,\beta\}}}\gamma-d(X_{\{\alpha,\beta\}}\hk\gamma)
\end{align*}
Adding these terms and using proposition \ref{Poisson is Schouten} and proposition \ref{Lie of bracket} we get that

\begin{align*}
\{\alpha,\{\beta,\gamma\}\}+\{\beta,\{\gamma,\alpha\}\}+\{\gamma,\{\alpha,\beta\}\}&=-\L_{X_\alpha}(d(X_\beta\hk\gamma))+\L_{X_\beta}(d(X_\alpha\hk\gamma))+\\
&\ \ \ \ \ -d(X_{\{\alpha,\beta\}}\hk\gamma)+2d(\varphi(X_\alpha,X_\beta,X_\gamma))
\end{align*}
However, 
\begin{align*}
-d(X_{\{\alpha,\beta\}}\hk\gamma)&=d([X_\alpha,X_\beta]\hk\gamma)\\
&=d\left(\L_{X_\alpha}(X_\beta\hk\gamma)-X_\beta\hk \L_{X_\alpha}\gamma\right)\\
&=d(\L_{X_\alpha}(X_\beta\hk\gamma))-d(X_\beta\hk \L_{X_\alpha}\gamma)\\
&=d(\L_{X_\alpha}(X_\beta\hk\gamma))-d\bigl[X_\beta\hk(X_\alpha\hk d\gamma)-X_\beta\hk(d(X_\alpha\hk\gamma))\bigr]\\
&=\L_{X_\alpha}(d(X_\beta\hk\gamma))-d(X_\beta\hk X_\alpha\hk X_\gamma\hk\varphi)-d(X_\beta\hk(d(X_\alpha\hk\gamma)))\\
&=\L_{X_\alpha}(d(X_\beta\hk\gamma))-d(\varphi(X_\alpha,X_\beta,X_\gamma))-d(\L_{X_\beta}(X_\alpha\hk\gamma))\\
&=\L_{X_\alpha}(d(X_\beta\hk\gamma))-d(\varphi(X_\alpha,X_\beta,X_\gamma))-\L_{X_\beta}(d(X_\alpha\hk\gamma))
\end{align*}

Substituting this in, we see that  \[\{\alpha,\{\beta,\gamma\}\}+\{\beta,\{\gamma,\alpha\}\}+\{\gamma,\{\alpha,\beta\}\} = d(\varphi(X_\alpha,X_\beta,X_\gamma)).\] 
\end{proof}

}
\del{
\subsection{Special Forms}
We make a new definition
\begin{definition}
An element $\alpha\in\Omega^{n-k}_{\text{Ham}}(M)$ is called a special form if its Lie derivative with respect to any multivector field $X$ in that preserves $\omega$ vanishes. That is, if \[\L_X\omega=0 \implies \L_X\alpha=0.\] 
We let $\Omega^k_S(M)$ denote the set of special $k$-forms.
\end{definition}
}
 Next we recall the cohomology theory governing equivariance from symplectic geometry, without proof, and then generalize to the multisymplectic setting. The results from symplectic geometry can all be found in Chapter 4.2 of \cite{Marsden} for example.  We will provide more general proofs later on in this section. Let $(M,\omega)$ be a symplectic manifold, and $\Phi:G\times M\to M$ a symplectic Lie group action by a connected Lie group $G$ . We consider the induced symplectic Lie algebra action $\g\times\Gamma(TM)\to\Gamma(TM)$. Suppose that a moment map  $f:\g\to C^\infty(M)$ exists. That is, $df(\xi)=V_\xi\hk\omega$ for all $\xi\in\g$. By definition, $f$ is equivariant if \[f(\mathrm{Ad}_{g^{-1}}\xi)=\Phi_g^\ast f(\xi).\]Following Chapter 4.2 of \cite{Marsden}, for $g\in G$ and $\xi\in\g$ define $\psi_{g,\xi}\in C^\infty(M)$ by \begin{equation}\label{psi}\psi_{g,\xi}(x):= f(\xi)(\Phi_g(x))-f(\mathrm{Ad}_{g^{-1}}\xi)(x).\end{equation}

\begin{proposition}\label{psi constant}
For each $g\in G$ and $\xi\in\g$, the function $\psi_{g,\xi}\in C^\infty(M)$ is constant.
\end{proposition}
Since $\psi_{g,\xi}$ is constant, we may define the map $\sigma:G\to\g^\ast$ by \[\sigma(g)(\xi):=\psi_{g,\xi},\] where the right hand side is the constant value of $\psi_{g,\xi}$.
\del{This motivates consideration of the chain complex \[\g^\ast\otimes\R\to(G\to \g^\ast\otimes\R)\to(G\times G\to\g^\ast\otimes\R)\to\cdots,\]with the natural action of $G$ on $\g^\ast\otimes\R$ given by \[g\cdot\alpha(\xi):=Ad_{g^{-1}}^\ast\alpha(\xi)\] for $g\in G,\alpha\in\g^\ast\otimes\R$ and $\xi\in\g$.}
\begin{proposition}\label{cocycle}
The map $\sigma:G\to\g^\ast$ is a cocycle in the chain complex \[\g^\ast\to C^1(G,\g^\ast)\to C^2(G,\g^\ast)\to\cdots.\]That is, $\sigma(gh)=\sigma(g)+\mathrm{Ad}_{g^{-1}}^\ast\sigma(h)$ for all $g,h\in G$. 
\end{proposition}
The map $\sigma$ is called the cocycle corresponding to $f$. The following proposition shows that for any symplectic group action, the cocycle gives a well defined cohomology class. 

\begin{proposition}\label{cohomology class}
For any symplectic action of  $G$ on $M$ admitting a moment map, there is a well defined cohomology class. More specifically, if $f_1$ and $f_2$ are two moment maps, then their corresponding cocycles $\sigma_1$ and $\sigma_2$ are in the same cohomology class, i.e. $[\sigma_1]=[\sigma_2]$.
\end{proposition}

By definition, we see that $\sigma$ is measuring the equivariance of $f$. That is, $\sigma=0$ if and only if $f$ is equivariant. Moreover, if the cocycle corresponding to a moment map vanishes in cohomology, the next proposition shows that we can modify the original moment map to make it equivariant.

\begin{proposition}\label{sigma class zero}
Suppose that $f$ is a moment map with corresponding cocycle $\sigma$. If $[\sigma]=0$ then $\sigma=\partial\theta$ for some $\theta\in\g^\ast$ and $f+\theta$ is an equivariant moment map.
\end{proposition}

We now show how this theory generalizes to multisymplectic geometry. For the rest of this section we let $(M,\omega)$ denote an $n$-plectic manifold and $\Phi:G\times M\to M$ a multisymplectic connected group action. We consider the induced multisymplectic Lie algebra action $\g\times\Gamma(TM)\to\Gamma(TM)$. Assume that we have a weak homotopy moment map $(f)$, i.e. a collection of maps $f_k:\Rho_{\g,k}\to \Omega^{n-k}_{\mathrm{Ham}}(M)$ satisfying equation (\ref{wmm kernel}). 

To extend equation (\ref{psi}) to multisymplectic geometry, for $g\in G$ and $p\in\Rho_{\g,k}$, we define the following $(n-k)$-form: \begin{equation}\label{ms equiv}\psi^k_{g,p}:= f_k(p)-\Phi_{g^{-1}}^\ast f_k(\mathrm{Ad}_{g^{-1}}(p)).\end{equation}The following proposition generalizes Proposition \ref{psi constant}. 

\begin{proposition}\label{general closed}
The $(n-k)$-form $\psi^k_{g,p}$ is closed.
\end{proposition}
\begin{proof}

\del{Since our moment maps are taking values in the special forms, it follows that $\psi_{g,p}$ is a special form as the sum of two special forms. It remains to show that $\psi_{g,p}$ is closed. } Since $\Phi^\ast_g$ is injective and commutes with the differential, our claim is equivalent to showing that $\Phi^\ast_g(\psi^k_{g,p})$ is closed. Indeed we have that
\begin{align*}
d(\Phi_g^\ast(\psi^k_{g,p}))&=d(\Phi_g^\ast f_k(p)-f_k(\mathrm{Ad}_{g^{-1}}p))\\
&=\Phi_g^\ast(df_k(p))-d(f_k(\mathrm{Ad}_{g^{-1}}p))\\
&=-\zeta(k)\Phi_g^\ast(V_p\hk\omega)+\zeta(k)V_{\mathrm{Ad}_{g^{-1}}p}\hk\omega&\text{since $(f)$ is moment map}\\
&=-\zeta(k)\Phi_g^\ast(V_p\hk\omega)+\zeta(k)(\Phi_g^\ast V_p)\hk\omega&\text{by Proposition \ref{adjoint over wedge}}\\
&=-\zeta(k)\Phi_g^\ast(V_p\hk\omega)+\zeta(k)\Phi_g^\ast(V_p\hk\omega)&\text{since $G$ preserves $\omega$}\\
&=0.
\end{align*}
\end{proof}

\del{
\begin{definition}
To ease the notation, we will let $\Omega_{Sc}^k(M)$ denote the intersection of $\Omega^k_{S}(M)$ and $\Omega^k_{\text{cl}}(M)$. 
\end{definition}
}

In analogy to symplectic geometry, we now see each component of a weak moment map gives a cocycle.
\begin{definition} We call the map $\sigma_k:G\to\Rho_{\g,k}^\ast\otimes\Omega^{n-k}_{\text{cl}}$ defined by \[\sigma_k(g)(p):=\psi^k_{g,p}\]the cocycle corresponding to $f_k$. 
\end{definition}

As a generalization of Proposition \ref{cocycle} we obtain:

\begin{proposition}
\label{multi cocycle}
The map $\sigma_k$ is a $1$-cocycle in the chain complex \[\Rho_{\g,k}^\ast\otimes\Omega^{n-k}_{\mathrm{cl}}\to C^1(G,\Rho_{\g,k}^\ast\otimes\Omega^{n-k}_{\mathrm{cl}}) \to C^2(G,\Rho_{\g,k}^\ast\otimes\Omega^{n-k}_{\mathrm{cl}}) \to \cdots,\]where the action of $G$ on $\Rho_{\g,k}^\ast\otimes\Omega^{n-k}_{\mathrm{cl}}$ is given by the tensor product of the co-adjoint and pullback actions. The induced infinitesimal action of $\g$ on $\Rho_{\g,k}^\ast\otimes\Omega^{n-k}_{\mathrm{cl}}$ is defined as follows: for $f\in\Rho_{\g,k}^\ast\otimes\Omega^{n-k}_{\mathrm{Ham}}$, $p\in\Rho_{\g,k}$ and $\xi\in \g$, \begin{equation}(\xi\cdot f)(p):=f(\mathrm{ad}_\xi(p))+\L_{V_\xi}f(p).\end{equation}
\end{proposition}

\begin{proof}
By equation (\ref{group differential}) we know that $(\partial(\sigma)(g,h))(p):=\sigma(gh)(p)-\sigma(g)(p)-g\cdot\sigma(h)(p)$. For arbitrary $p\in\Rho_{\g,k}$ we have
\begin{align*}
\sigma_k(gh)(p)&=f_k(p)-\Phi^\ast_{(gh)^{-1}}(f_k\mathrm{Ad}_{(gh)^{-1}}p)\\
&=f_k(p)-\Phi^\ast_{g^{-1}}\Phi^\ast_{h^{-1}}(f_k(\mathrm{Ad}_{h^{-1}}\mathrm{Ad}_{g^{-1}}p))\\
&=f_k(p)-\Phi^\ast_{g^{-1}}(f_k(\mathrm{Ad}_{g^{-1}}p))+\Phi^\ast_{g^{-1}}(f_k(\mathrm{Ad}_{g^{-1}}p))-\Phi^\ast_{g^{-1}}(\Phi^\ast_{h^{-1}}(f_k(\mathrm{Ad}_{h^{-1}}\mathrm{Ad}_{g^{-1}}p)))\\
&=\sigma_k(g)(p)+\Phi_{g^{-1}}^\ast(\sigma_k(h)(\mathrm{Ad}_{g^{-1}}p))\\
&=\sigma_k(g)(p)+g\cdot \sigma_k(h)(p).
\end{align*}
\end{proof}

\begin{definition} Let \[\mathfrak{C}=\bigoplus_{k=1}^n \Rho_{\g,k}^\ast\otimes\Omega^{n-k}_{\text{cl}}\] Let $\sigma=\sigma_1+\sigma_2+\cdots $. We call the map $\sigma\in\mathfrak{C}$ the cocycle corresponding to $(f)$. 
\end{definition}
Since the components of a weak moment map do not interact, as a corollary to Proposition \ref{multi cocycle} we obtain
\begin{proposition}
The map $\sigma$ is a cocycle in the complex \[\mathfrak{C}\to C^1(G,\mathfrak{C})\to C^2(G,\mathfrak{C})\to\cdots\]

\end{proposition}
The next theorem shows that multisymplectic  Lie algebra actions admitting weak moment maps give a well defined cohomology class, generalizing Proposition \ref{cohomology class}.

\begin{theorem}\label{general class}
Let $G$ act multisymplectically on $(M,\omega)$. To any weak moment map, there is a well defined cohomology class $[\sigma]$ in $H^1(G,\mathfrak{C})$. More precisely if $(f)$ and $(g)$ are two weak moment maps with cocycles $\sigma$ and $\tau$, then $\sigma-\tau$ is a coboundary.
\end{theorem}

\begin{proof}
We need to show that $\sigma_k-\tau_k$ is a coboundary for each $k$. We have that \[\sigma_k(g)(p)-\tau_k(g)(p)=f_k(p)-g_k(p)-\Phi_{g^{-1}}^\ast(f_k(\mathrm{Ad}_{g^{-1}}p)-g_k(\mathrm{Ad}_{g^{-1}}(\xi))).\]
However, $(f)$ and $(g)$ are both moment maps and so $d(f_k(p)-g_k(p))=0$. Thus $f_k-g_k$ is in $\mathfrak{C}$. Moreover, by equation (\ref{group differential 2}), we see that $\sigma_k-\tau_k=\partial(f_k-g_k)$.

\end{proof}

If $(f)$ is not equivariant but its cocycle vanishes, then we can define a new equivariant moment map from $(f)$, in analogy to Proposition \ref{sigma class zero}. 

\begin{proposition}\label{general make}
Let $(f)$ be a weak moment map with cocycle satisfying $[\sigma]=0$. Then $\sigma=\partial\theta$ for some $\theta\in\mathfrak{C}$ and the map $(f)+\theta$ is an equivariant weak moment map.
\end{proposition}

\begin{proof}
We have that $(f)+\theta$ is a moment map since $\theta(p)$ is closed for all $p\in\Rho_{\g,k}$.  Let $\widetilde\sigma$ denote the corresponding cocycle. Note that by equation (\ref{group differential 2}) we have $(\partial(\theta)(g))(p)=\theta(\mathrm{Ad}_{g^{-1}}p)-\Phi_g^\ast\theta(p)$. By the injectivity of $\Phi_g^\ast$, to show that $\widetilde\sigma=0$, it is sufficient to show that $\Phi_g^\ast(\widetilde\sigma(g)(p))=0$ for all $g\in G$ and $p\in\Rho_{g,k}$. Indeed, \begin{align*}
\Phi_g^\ast(\widetilde\sigma_k(g)(p))&=\Phi_g^\ast f(p)+\Phi_g^\ast\theta(p)-f(\mathrm{Ad}_{g^{-1}}p)-\theta(\mathrm{Ad}_{g^{-1}}p)\\
&=\sigma(g)(\xi)-\partial\theta(g)(\xi)\\
&=\sigma(g)(\xi)-\sigma(g)(\xi)&\text{since $\partial\theta =\sigma$}\\
&=0.
\end{align*}
\end{proof}

\del{

If $(f)$ is not equivariant with respect to the $G$-action, then we can define a new action for which $(f)$ is equivariant.

\begin{proposition}
\label{new action}
For $g\in G$ define $\Upsilon_g:\Rho_{\g,k}^\ast\otimes\Omega^{n-k}_{\mathrm{Ham}}\to\Rho_{\g,k}^\ast\otimes\Omega^{n-k}_{\mathrm{Ham}}$ by 

\[\Upsilon_g(\theta)(p):= \Phi_{g^{-1}}^\ast \theta(\mathrm{Ad}_{g^{-1}}p)+\sigma(g)(p)\]where $\theta$ is in $\Rho_{\g,k}^\ast\otimes\Omega^{n-k}_{\mathrm{Ham}}$ and $p$ is in $\Rho_{\g,k}$. Then $\Upsilon_g$ is a group action and $(f)$ is $\Upsilon_g$-equivariant.

\end{proposition}

\begin{proof}
The proof is a direct extension from the proof of Proposition 4.2.7 in \cite{Marsden}. We first show that $\Upsilon_g$ is a group action. Indeed, $\sigma(e)=0$ and $\mathrm{Ad}_{e}$ is the identity showing that $\Upsilon_e(\theta)=\theta$. For the multiplicative property of the group action we have
\begin{align*}
\Upsilon_{gh}(\theta)(p)&=\Phi^\ast_{(gh)^{-1}}\theta(\mathrm{Ad}_{(gh)^{-1}}p)+\sigma(gh)(p)\\
&=\Phi^\ast_{g^{-1}}(\Phi^\ast_{h^{-1}}(\theta(\mathrm{Ad}_{h^{-1}}\mathrm{Ad}_{g^{-1}}p)))+\sigma(g)(p)+\Phi^\ast_{g^{-1}}(\sigma(h)(\mathrm{Ad}_{g^{-1}}p))&\text{by Proposition \ref{cocycle}}\\
&=\Phi^\ast_{g^{-1}}(\Upsilon_h(\theta)(\mathrm{Ad}_{g^{-1}}p))+\sigma(g)(p)\\
&=\Upsilon_g(\Upsilon_h(\theta))(p).
\end{align*} 
The moment map $f_k$ is equivariant with respect to this action because 
\begin{align*}
\Upsilon_g(f_k)(p)&=\Phi^\ast_{g^{-1}}(f_k(\mathrm{Ad}_{g^{-1}}p))+\sigma(g)(p)\\
&=\Phi^\ast_{g^{-1}}f_k(\mathrm{Ad}_{g^{-1}}p)+ f_k(p)-\Phi^\ast_{g^{-1}}f_k(\mathrm{Ad}_{g^{-1}}p)\\
&=f_k(p).
\end{align*}
\end{proof}

}

\subsection{Infinitesimal Equivariance in Multisymplectic Geometry}
We now consider infinitesimal equivariance. We start by recalling the notions from symplectic geometry. That is, we differentiate equation (\ref{psi}) to obtain the map $\Sigma:\g\times\g\to C^\infty(M)$ defined  by $\Sigma(\xi,\eta):=\left.\frac{d}{dt}\right|_{t=0}\psi_{\exp(t\eta),\xi}$. A straightforward computation, which we generalize in Proposition \ref{comp of Sigma}, gives that \[\Sigma(\xi,\eta)=f([\xi,\eta])-\{f(\xi),f(\eta)\}.\] Another quick computation shows that $df([\xi,\eta])=d\{f(\xi),f(\eta)\}$, showing $\Sigma(\xi,\eta)$ is a constant function for every $\xi,\eta\in\g$. That is, $\Sigma$ is a function from $\g\times\g$ to $\R$. 

\begin{proposition}\label{inf cocycle}The map $\Sigma:\g\times\g\to\R$ is a Lie algebra $2$-cocycle in the chain complex \[\R\to C^1(\g, \R)\to C^2(\g,\R)\to\cdots.\] 
\end{proposition}

\begin{definition}\label{inf equiv moment}
A moment map $f:\g\to C^\infty(M)$ is infinitesimally equivariant if $\Sigma=0$, i.e. if \begin{equation}\label{equivariant equation 2}f([\xi,\eta])=\{f(\xi),f(\eta)\}\end{equation} for all $\xi,\eta\in\g$.
\end{definition}

Notice that since $\Sigma$ is just the derivative of $\sigma$, it follows that for a connected Lie group, infinitesimal equivariance and equivariance are equivalent. Since we will always be working with connected Lie groups, we will abuse terminology and call a moment map equivariant if it satisfies equation (\ref{equivariant equation}) or (\ref{equivariant equation 2}).
\del{
\begin{proposition}\label{lie alg morph 1}
A pre-moment map is a Lie algebra morphism from $(\g,[\cdot,\cdot])$ to $(C^\infty(M)/\mathrm{closed},\{\cdot,\cdot\})$.
\end{proposition} 
\begin{proof}
Since $df([\xi,\eta])=d\{f(\xi),f(\eta)\}$ it follows that $f([\xi,\eta])-\{f(\xi),f(\eta)\}$ is a constant. Hence, for a pre-moment map it is always true that $f([\xi,\eta])=\{f(\xi),f(\eta)\}$ in the quotient space.
\end{proof}

\begin{proposition}\label{lie alg morph 2}
If $f$ is an equivariant moment map then $f$ is a Lie algebra morphism from $(\g,[\cdot,\cdot])$ to $(C^\infty(M),\{\cdot,\cdot\})$.
\end{proposition}
\begin{proof}
This follows from the above, since if $f$ is equivariant, then $\Sigma=0$.
\end{proof}

\begin{remark}
The converse to Proposition \ref{lie alg morph 2} is not necessarily true. That is, if a moment map is a Lie algebra morphism, then it is not necessarily equivariant. This is true, however, if the acting group is compact and connected. Nonetheless, we will abuse terminology and call a moment map equivariant if it is a Lie group homomorphism.
\end{remark}
}

Now we turn our attention towards the multisymplectic setting. As in symplectic geometry, the infinitesimal equivariance of a weak moment map comes from differentiating $\psi_{\exp(t\xi),p}$ for fixed $\xi\in\g$ and $p\in\Rho_{\g,k}$.

\begin{proposition}\label{comp of Sigma} Let $\Sigma_k$ denote $\left.\frac{d}{dt}\right|_{t=0}\psi^k_{\exp(t\xi),p}$. Then we have that $\Sigma_k$ is a map from $\g$ to $\Rho_{\g,k}^\ast\otimes \Omega^{n-k}_{\mathrm{cl}}$ and is given by \[\Sigma_k(\xi,p)=f_k([\xi,p])+\L_{V_\xi}f_k(p),\]for $\xi\in\g$ and $p\in\Rho_{\g,k}$.

\end{proposition}

\begin{proof}We have that 
\begin{align*}
\left.\frac{d}{dt}\right|_{t=0}\psi^k_{\exp(t\xi),p}&=\left.\frac{d}{dt}\right|_{t=0} f_k(p)-\left.\frac{d}{dt}\right|_{t=0}\Phi_{\exp(-t\xi)}^\ast(f_k(\mathrm{Ad}_{\exp(-t\xi)}(p)))\\
&=-\left.\frac{d}{dt}\right|_{t=0}\Phi^\ast_{\exp(-t\xi)}(f_k(\mathrm{Ad}_{\exp(-t\xi)}p))\\
&=-f_k(\left.\frac{d}{dt}\right|_{t=0}\mathrm{Ad}_{\exp(-t\xi)}p)-(\left.\frac{d}{dt}\right|_{t=0}\Phi^\ast_{\exp(-t\xi)})(f_k(p))\\
&=-f_k(-[\xi,p])+\L_{\xi_M}f_k(p)\\
\end{align*}

\end{proof}

\del{ 
The next corollary is a complete generalization of the situation in symplectic geometry. 

\begin{corollary}
If a refined homotopy moment map is equivariant, then it is a morphism between the differential graded Lie algebras $(\Lambda^\bullet\g, \partial,[\cdot,\cdot])$ and $(\Omega^\bullet_p(M),d,\{\cdot,\cdot\})$. (Change this is a little, this is true but it is not a corollary of the above)
\end{corollary}

\begin{proof}
If $(f)$ is equivariant, then $\sigma=0$ and so $\Sigma=0$. TO DO
\end{proof}

}

Let $R_k=\Rho_{\g,k}^\ast\otimes \Omega^{n-k}_\text{cl}$. Then $R_k$ is a $\g$-module under the induced action from the tensor product of the adjoint and Lie derivative actions. Concretely, for $\alpha\in R_k$, $\xi\in\g$ and $p\in\Rho_{\g,k}$, \[(\xi\cdot\alpha)(p)=\alpha([\xi,p])+\L_{V_\xi}\alpha.\]
Consider the chain complex \[R_k\to C^1(\g,R_k)\to C^2(\g,R_k)\to\cdots,\]where the differential is given in equation (\ref{group differential 2}).

The following is a generalization of Proposition \ref{inf cocycle}.

\begin{proposition}\label{general inf}
The map $\Sigma_k$ is in the kernel of $\partial_k$. That is, $\Sigma_k$ is a cocycle. \end{proposition}

\begin{proof}We need to show that $\partial\Sigma_k=0$. Indeed, for $\xi,\eta\in\g$ and $p\in\Rho_{\g,k}$, we have that 
\begin{align*}
\partial\Sigma_k(\xi,\eta)(p)&=\xi\cdot(\Sigma_k(\eta)(p))-\eta\cdot(\Sigma_k(\xi)(p))+\Sigma_k([\xi,\eta])(p)\\
&= \Sigma_k(\eta)(\mathrm{ad}_\xi(p))+\L_{V_\xi}(\Sigma_k(\eta)(p))-\Sigma_k(\xi)(\mathrm{ad}_\eta(p))\\
&\quad{}-\L_{V_\eta}(\Sigma_k(\xi)(p))+\Sigma_k([\xi,\eta])(p).
\end{align*}
By definition of the ad map, this is equal to
\[\Sigma_k(\eta)([\xi,p])+\L_{V_\xi}(\Sigma_k(\eta)(p))-\Sigma_k(\xi)([\eta,p])-\L_{V_\eta}(\Sigma_k(\xi)(p))+\Sigma_k([\xi,\eta])(p)\] and using the definition of $\Sigma$ this becomes

\begin{align*}
f_k([\eta,[\xi,p]])&+\L_{V_\eta}f_k([\xi,p])+\L_{V_\xi}f_k([\eta,p])+\L_{V_\xi}\L_{V_\eta}f_k(p)\\
&\quad{}-f_k([\xi,[\eta,p]])-\L_{V_\xi}f_k([\eta,p])-\L_{V_\eta}f_k([\xi,p])-\L_{V_\eta}\L_{V_\xi}f_k(p)\\
&\quad{}+f_k([[\xi,\eta],p])-\L_{V_{[\xi,\eta]}}f_k(p)\\
&=f_k([\eta,[\xi,p]])-f_k([\xi,[\eta,p]])+f_k([[\xi,\eta],p])\\
&\quad{}+\L_{V_\xi}\L_{V_\eta}f_k(p)-\L_{V_\eta}\L_{V_\xi}f_k(p)-\L_{V_{[\xi,\eta]}}f_k(p).\end{align*}

By the Jacobi identity this is equal to
\[\L_{V_\xi}\L_{V_\eta}f_k(p)-\L_{V_\eta}\L_{V_\xi}f_k(p)-\L_{V_{[\xi,\eta]}}f_k(p),\] 
and this vanishes by equation (\ref{L bracket}).
\del{For $\xi,\eta\in\g$ and $p\in\Rho_{\g,k}$ we have by definition that $\partial\Sigma_k(\xi,\eta)(p)=\Sigma_k(\eta, ad_\xi(p))-\Sigma_k(\xi,ad_\eta(p))+\Sigma_k(\xi,\eta)(p)$. That is, showing that $\partial\Sigma_k=0$ is equivalent to showing that $\Sigma$ satisfies the Jacobi identity \[\Sigma_k([\eta,[\xi,p]])+\Sigma_k([\xi,[p,\eta]])+\Sigma_k([p,[\xi,\eta]])=0.\]

Indeed we have that

\begin{align*}
\Sigma_k([\eta,[\xi,p]])+&\Sigma_k([\xi,[p,\eta]])+\Sigma_k([p,[\xi,\eta]])=f_k([\eta,[\xi,p]]+[\xi,[p,\eta]]+[p,[\xi,\eta]])\\
&+\{f_1(\eta),f_k[\xi,p]\}+\{f_1(\xi),f_k[\eta,p]\}+\{f_1([\xi,\eta]),f_k(p)\}\\
&+d(\eta\hk f_k([\xi,p]))+d(\xi\hk f_k([\eta,p]))+d([\xi,\eta]\hk f_k(p))\\
&=\{f_1(\eta),f_k[\xi,p]\}+\{f_1(\xi),f_k[\eta,p]\}+\{f_1([\xi,\eta]),f_k(p)\}&\text{by Jacobi}\\
&+d(\eta\hk f_k([\xi,p]))+d(\xi\hk f_k([\eta,p]))+d([\xi,\eta]\hk f_k(p))\\
&=\{f_1(\eta),\{f_k(p),f_1\xi\}]\}+\{f_1(\xi),\{f_k(p),f_1(\xi)\}\}+\{\{f_1(\xi),f_1(\eta)\},f_k(p)\}\\
&+d(\eta\hk f_k([\xi,p]))+d(\xi\hk f_k([\eta,p]))+d([\xi,\eta]\hk f_k(p))\\
&=\L_\eta(f_k([\xi,p]))+\L_\xi(f_k([\eta,p]))+\L_{[\xi,\eta]}f_k(p)&\text{by prop \ref{Jacobi}}\\
&=0&\text{since $im(f)\subset\Omega^\bullet_{Sc}(M)$}
\end{align*}}

\end{proof}

As in symplectic geometry, we have that for a connected Lie group, a weak homotopy moment map is equivariant if and only if it is infinitesimally equivariant. That is, the weak homotopy $k$-moment map is equivariant if and only if $\sigma_k=0$ or $\Sigma_k=0$. A weak homotopy moment map is equivariant if $\sigma_k=0$ or $\Sigma_k=0$ for all $1\leq k\leq n$.

Now that we have generalized the notions of equivariance from symplectic to multisymplectic geometry, we move on to study the existence and uniqueness of these weak homotopy moment maps.
\subsection{Existence of Not Necessarily Equivariant Weak Moment Maps}

For a connected Lie group $G$ acting symplectically on a symplectic manifold $(M,\omega)$, recall the following standard results from symplectic geometry. We refer the reader to \cite{dasilva} for proofs and note that we give more general proofs later on in this section.

\begin{proposition}\label{bracket gives}
For any $\xi,\eta\in\g$ we have \[[V_\xi,V_\eta]\hk\omega=d(V_\xi\hk V_\eta\hk\omega).\]
\end{proposition}

\del{
\begin{proposition}\label{H1}
We have that $H^1(\g)=[\g,\g]^0$, where $[\g,\g]^0$ is the annihilator of $[\g,\g]$.
\end{proposition}

\begin{proof}
This follows since for $c\in\g^\ast$ we have $\partial c(\xi,\eta)=c([\xi,\eta])$ by definition.
\end{proof}
}

\begin{proposition}\label{H2}
We have that $H^1(\g)=0$ if and only if $\g=[\g,\g]$.
\end{proposition}

Combining these two propositions, we obtain:

\begin{proposition}\label{H4}
If $H^1(\g)=0$, then any symplectic action admits a moment map, which is not necessarily equivariant.
\end{proposition}

We now show how these results generalize to multisymplectic geometry. Let a connected Lie group act multisymplectically on an $n$-plectic manifold $(M,\omega)$.
\begin{proposition}\label{existence 1} 
For arbitrary $q$ in $\Rho_{\g,k}$ and $\xi\in\g$ we have that \[[V_q,V_\xi]\hk\omega= -(-1)^kd(V_q\hk V_\xi\hk\omega).\] 
\end{proposition}

\begin{proof}
By linearity it suffices to consider decomposable $q=\eta_1\wedge\cdots\wedge\eta_k$. A quick computation shows that $[V_q,V_\xi]\hk\omega=-V_{[q,\xi]}\hk\omega$. Using Lemmas \ref{lemma formula} and \ref{extended Cartan} we obtain:

\begin{align*}
V_{[q,\xi]}\hk\omega&=V_{\partial({q\wedge\xi})}\hk\omega\\
&=(-1)^kd(V_{q\wedge\xi}\hk\omega)-\sum_{i=1}^k(-1)^i\eta_1\wedge\cdots\wedge\widehat \eta_i\wedge\cdots\wedge 
\eta_k\wedge\xi\hk\L_{\eta_i}\omega -V_{q\wedge\xi}\hk d\omega\\
&=(-1)^kd(V_q\hk V_\xi\hk\omega).\\
\end{align*}The claim now follows.
\end{proof}
The next proposition is a generalization of Proposition \ref{H2}.

\begin{proposition}\label{kernel equals bracket}
If $H^0(\g,\Rho_{\g,k}^\ast)=0$ then $\Rho_{\g,k}=[\Rho_{\g,k},\g]$.
\end{proposition}
\begin{proof}
By equation (\ref{group differential 2}), an element $c\in H^0(\g,\Rho_{\g,k}^\ast)$ satisfies $c([\xi,p])=0$ for all $\xi\in\g$.  That is, \[H^0(\g,\Rho_{\g,k}^\ast)=[\Rho_{\g,k},\g]^0,\] where $[\Rho_{\g,k},\g]^0$ is the annihilator of $[\Rho_{\g,k},\g]$ in $\Rho_{\g,k}^\ast$.
\end{proof}
We now arrive at our main theorem on the existence of not necessarily equivariant weak moment maps. The following is a generalization of Proposition \ref{H4}.

\begin{theorem}\label{theorem existence}
Let $G$ act multisymplectically on $(M,\omega)$. If $H^0(\g,\Rho_{\g,k}^\ast)=0$ then the $k$-th component of a not necessarily equivariant weak moment map exists. Moreover, if $H^0(\g,\Rho_{\mathfrak{g},k}^\ast)=0$ for all $1\leq k\leq n$, then a not necessarily equivariant weak moment map exists. The same result holds if $H^0(\g,\Rho_{\g,k}^\ast\otimes\Omega^{n-k}_{\mathrm{cl}})=0$, and $H^0(\g,\Omega^{n-k}_{\mathrm{cl}})\not=0$.
\end{theorem}

\begin{proof}

Note that when $H^0(\g,\Omega^{n-k}_{\mathrm{cl}})\not=0$, the space $H^0(\g,\Rho_{\g,k}^\ast\otimes\Omega^{n-k}_{\mathrm{cl}})=0$ if and only if $H^0(\g,\Rho_{\g,k}^\ast)=0$ by the Kunneth formula (see for example Theorem 3.6.3 of \cite{kunneth}). The claim now follows from Proposition \ref{kernel equals bracket} and Proposition \ref{existence 1}. Indeed, Proposition \ref{existence 1} says we may define a weak moment map on elements of the form $[p,\xi]$ by $(-1)^kV_p\hk V_\xi\hk\omega$, where $p\in\Rho_{\g,k}$ and $\xi\in\g$, and Proposition \ref{kernel equals bracket} says every element in $\Rho_{\g,k}$ is a sum of elements of this form.
\end{proof}

\begin{remark}
Notice that for the case $n=k$, it is always true that $H^0(\g,\Omega^{n-k}_{\mathrm{cl}})\not=0$ since any constant function is closed. Hence Theorem \ref{theorem existence} gives a generalization of Theorems 3.5 and 3.14 of \cite{ms} and \cite{MS} respectively. It also agrees with Lemma 4.4 and Corollary 4.2 of \cite{existence 1}. Moreover, by taking $n=k=1$, we see that we are obtaining a generalization from symplectic geometry.
\end{remark}

\begin{remark}
In \cite{existence 1}, Corollary 4.2 gives an existence result for a full homotopy moment map under a weakly Hamiltonian action (see Definition 2.1 of \cite{existence 1}) on an $n$-plectic manifold $(M,\omega)$. The result guarantees existence of a full homotopy moment map under the assumption of vanishing de Rham cohomology, i.e. $H^k_\mathrm{dR}(M)=0$ for all $1\leq k\leq n$ and the vanishing of a specific cocycle $[c]\in H^{n+1}(\g)$. It follows that this corollary provides an existence result for weak moment maps, since any full homotopy moment map reduces to a weak moment map when restricted to the Lie kernel. Corollary 4.2 of \cite{existence 1} appears to be independent from our result as it deals with de Rham cohomology, whereas we are concerned with Lie algebra cohomology.
\end{remark}

\begin{remark}
Further to the results of \cite{existence 1}, the results on the existence of full homotopy moment maps provided in \cite{existence 2} also apply to weak moment maps, as any full moment map reduces to a weak moment map. See for example Proposition 2.5 of \cite{existence 2}. It is not clear if the results in \cite{existence 1} and \cite{existence 2} are related to ours, since in those papers their existence theory is derived from the double complex $(\Lambda^\bullet\g^\ast\otimes\Omega^\bullet(M),d_{tot})$, with $d_{tot}=\delta\otimes 1 +1\otimes d,$ where $\delta$ is the Chevalley-Eilenberg differential and $d$ is the de Rham differential. In our work, we are considering a different cohomology complex. 
\end{remark}
\begin{example}\label{example 1}
Consider the multisympletic manifold $(\R^4,\omega)$ where $\omega=\mathrm{vol}$ is the standard volume form. That is, we are working in the case $n=3$. Let $x_1,\cdots,x_4$ denote the standard coordinates. Let $G=SU(2)$ act on $\R^4$ by rotations. The corresponding Lie algebra action is generated by the vector fields \[E_0=x_3\frac{\pd}{\pd x_1}+x_4\frac{\pd}{\pd x_2}-x_1\frac{\pd}{\pd x_3}-x_2\frac{\pd}{\pd x_4},\] \[E_1=-x_2\frac{\pd}{\pd x_1}+x_1\frac{\pd}{\pd x_2}-x_4\frac{\pd}{\pd x_3}+x_3\frac{\pd}{\pd x_4},\] and \[E_2=x_4\frac{\pd}{\pd x_1}+x_3\frac{\pd}{\pd x_2}-x_2\frac{\pd}{\pd x_3}-x_1\frac{\pd}{\pd x_4}.\] For the case $k=2$, consider the distance function $r=\sqrt{x_1^2+x_2^2+x_3^2+x_4^2}$. It is clear that the distance function is invariant under rotations and hence $\L_{E_i}dr=0$ for $i=0,1,2,3$. Since $dr$ is a closed $1$-form, it follows that $dr$ is a non-zero element of $H^0(\g,\Omega^{1}_{\mathrm{cl}}(M))$. That is, $H^0(\g,\Omega^{1}_{\mathrm{cl}}(M))\not=0$.

For the case $k=1$, consider $\alpha:=dx_1\wedge dx_2+dx_3\wedge dx_4$. A quick calculation shows that $E_i\hk\alpha=$ for $i=0,1,2,3$ so that $\alpha$ is invariant under the $\mathfrak{su}(2)$ action. Since $d\alpha=0$, it follows that $H^0(\g,\Omega^{2}_{\mathrm{cl}}(M))\not=0$ as well.

Hence, by Theorem \ref{theorem existence}, it follows that a weak moment map exists. 
\end{example}
The next example gives a scenario for which Theorem \ref{theorem existence} can only be applied to specific components of a weak moment map. 
\begin{example}Take the setup of Example \ref{example 1} but instead consider the action of $SO(4)$.  As in Example \ref{example 1}, $dr$ is a non-zero closed $1$-form which is invariant under the action. That is, $H^0(\g,\Omega^{1}_{\mathrm{cl}}(M))\not=0$. However, in this setup, $H^0(\g,\Omega^{2}_{\mathrm{cl}}(M))=0$. Indeed, the infinitesimal generators of $\mathfrak{so}(4)$ are of the form $x_i\frac{\pd}{\pd x_j}-x_j\frac{\pd}{\pd x_i}$ where $1\leq i,j,\leq 4$. An arbitrary $2$-form may be written as $\beta=\sum_{i,j}a_{ij}dx_i\wedge dx_j$. A computation shows that the condition $\L_{V_\xi}\beta=0$ for all $\xi\in\mathfrak{so}(4)$ implies that necessarily $\beta=0$. Hence  $H^0(\g,\Omega^{2}_{\mathrm{cl}}(M))=0$. 

It follows that, in this case, Theorem \ref{theorem existence} guarantees the existence of the second component of a weak moment map, but does not guarantee the existence of the first.

\end{example}
Another generalization of Proposition \ref{H4} to multisymplectic geometry is given by:

\begin{proposition}\label{dont know}
If $H^k(\g)=0$, then the $k$-th component of a not necessarily equivariant weak moment map exists.
\end{proposition}

\begin{proof}
If $H^k(\g)=0$ then $\Rho_{\g,k}=\mathrm{Image}(\partial_{k+1})$, since $\Rho_{\g,k}=\mathrm{ker}(\partial_k)$. But for $p\in\mathrm{Image}(\partial_{k+1})$ we have that $p=\partial q$ for some $q\in\Lambda^{k+1}\g$. Then by Lemma \ref{extended Cartan} we have 

\[V_p\hk\omega=(-1)^kd(V_q\hk\omega).\] Hence we may define $f_k(p)$ to be $(-1)^kV_q\hk\omega$.
\end{proof}

\begin{remark}
Proposition \ref{dont know} gives another generalization of the results of Madsen and Swann. Indeed, by taking $n=k$ we again arrive at Theorems 3.5 and 3.14 of \cite{ms} and \cite{MS} respectively. This also agrees with Lemma 4.4 of \cite{existence 2}.
\end{remark}
\del{
Summarizing Theorem \ref{theorem existence} and Proposition \ref{dont know} we obtain:

\begin{proposition}\label{theorem open question 1}
If $H^1(\g)=\cdots=H^n(\g)=0$ then a not necessarily equivariant weak moment map $(f)$ exists.
\end{proposition}

\begin{remark}
Proposition \ref{theorem open question 1} is vacuous for $n\geq 3$. Indeed, Theorem 4.1 of \cite{MS} shows that if $H^3(\mathfrak{g})=0$ then $\mathfrak{g}$ is solvable and not nilpotent, which implies $H^1(\g)\not=0$. Hence the relevant result is Proposition \ref{dont know}.
\end{remark}

\begin{theorem}
If $H^0(\g,\Rho_{\g,k}^\ast\otimes\Omega^{n-k}_{\mathrm{cl}})=0$ and $H^0(\g,\Omega^{n-k}_{\mathrm{cl}})\not=0$ for all $1\leq k \leq n$ then a not necessarily equivariant weak moment map $(f)$ exists.
\end{theorem}
}

\del{

\begin{proposition}\label{existence 2} If $H_k(\g)=0$ then every element $p$ in $\Rho_{\g,k}$ is of the form \[p=\sum[q_i,\xi_i],\]where $q_i\in\Rho_{\g,k}$ and $\xi_i\in\g$. 
\end{proposition}

\begin{proof}
This fact follows from Theorem 3.5 of Madsen and Swann, although have question about their argument.

 Don't know if this helps but we have that $[\Rho_{\g,k},\g]$ is a subset of $\Rho_{\g,k}$. Let $[p,\xi]$ be an arbitrary element of $[\Rho_{\g,k},\g]$.
\begin{align*}
\partial(p\wedge\xi)&=\partial(p)\wedge\xi+p\wedge\partial(\xi)+[p,\xi]\\
&=p\wedge\partial(\xi)+[p,\xi]&\text{since $\partial(p)=0$}\\
&=[p,\xi]&\text{since $\xi$ has degree $1$}
\end{align*}
Thus, $[p,\xi]$ is closed.
\end{proof} 
Combining these two propositions gives us the main theorem of this section.
\begin{proposition}
Let $\g$ act multisymplectically on a manifold $M$. If $H_k(\g)=0$, then the $k$-th component, $f_k$, of a moment map exists. The map $f_k$ is not necessarily equivariant.
\end{proposition}

\begin{proof}
This follows directly from Propositions \ref{existence 1} and \ref{existence 2}.
\end{proof}

\begin{proposition}
Let $\g$ act multisymplectically on $(M,\omega)$. If $H^1(\g)=\cdots=H^n(\g)=$ then a not necessarily equivariant weak homotopy moment map exists.
\end{proposition}
}

\subsection{Obtaining an Equivariant Weak Moment Map from a Non-Equivariant Weak Moment Map}

In this section we show that the theory involved in obtaining an equivariant moment map from a non-equivariant moment map extends from symplectic to multisymplectic geometry. We first recall the results from symplectic geometry. A proof of the results can be found in \cite{dasilva}. We give more general proofs later on in this section.

Proposition \ref{inf cocycle} shows that the map $\Sigma$ corresponding to a moment map $f$ is a Lie algebra $2$-cocycle. The next proposition says that if the cocycle is exact then $f$ can be made equivariant.

\begin{proposition}\label{if exact}
Let $f$ be a moment map and $\Sigma$ its corresponding cocycle. If $\Sigma=\partial(l)$ for some $l$, then $f+l$ is equivariant.
\end{proposition}

It follows from this that 
\begin{proposition}\label{obtain theorem}
If $H^2(\g)=0$ then one can obtain an equivariant moment map from a non-equivariant moment map.
\end{proposition}

Now let $G$ be a connected Lie group acting on an $n$-plectic manifold $(M,\omega)$. The following proposition generalizes Proposition \ref{if exact} to multisymplectic geometry.
\begin{proposition}\label{exact Sigma}
Let $f_k$ be the weak homotopy $k$-moment map, and let $\Sigma_k$ denote its corresponding cocycle. If $\Sigma_k=\partial(l_k)$ for some $l_k\in H^0(\g,\Rho_{\g,k}^\ast\otimes\Omega^{n-k}_{\mathrm{cl}})$, then $f_k+l_k$ is equivariant.
\end{proposition}

\begin{proof}
Fix $p\in\Rho_{\g,k}$ and $\xi\in\g$. Then

\begin{align*}
(f_k+l_k)([\xi,p])&=f_k([\xi,p])+l_k([\xi,p])\\
&=f_k([\xi,p])-((\partial l_k)(\xi))(p)+\L_{V_\xi}l_k(p)&\text{by equation (\ref{group differential 2})}\\
&=f_k([\xi,p])-\Sigma_k([\xi,p])+\L_{V_\xi}l_k(p)\\
&=\L_{V_\xi}f_k(p)+\L_{V_\xi}(l_k(p))&\text{by definition of $\Sigma_k$}\\
&=\L_{V_\xi}((f_k+l_k)(p)).
\end{align*}

\end{proof}
We now arrive at our generalization of Proposition \ref{obtain theorem}:

\begin{theorem}
If $H^1(\g,\Rho_{\g,k}^\ast\otimes\Omega^{n-k}_{\mathrm{cl}})=0$ then any weak $k$-moment map can be made equivariant. In particular, if $H^1(\g,\Rho_{\g,k}^\ast\otimes\Omega^{n-k}_{\mathrm{cl}})=0$ for all $1\leq k \leq n$, then any weak moment map $(f)$ can be made equivariant.
\end{theorem}

\begin{proof}
Let $f_k:\Rho_{\g,k}\to\Omega^{n-k}_\mathrm{Ham}$ be a weak $k$-moment map. If $H^1(\g,\Rho_{\g,k}^\ast\otimes\Omega^{n-k}_{\mathrm{cl}})=0$ then the corresponding cocycle $\Sigma_k$ is exact, i.e. $\Sigma_k=\partial(l_k)$ for some $l_k\in H^0(\g,\Rho_{\g,k})$. It follows from Proposition \ref{exact Sigma} that $f_k+l_k$ is equivariant.
\end{proof}

\del{

\begin{proposition}
For an element $m$ of degree $0$ in the complex \[(\R\to \Rho_{\g,k}^\ast\otimes\Omega^{n-k}_\text{cl})\to(\g\to\Rho_{\g,k}^\ast\otimes\Omega^{n-k}_\text{cl})\to (\g\times \g\to\Rho_{\g,k}^\ast\otimes\Omega^{n-k}_\text{cl})\to\cdots\] we  have that \[\delta(m)(\xi)(p)=m([\xi,p]).\]
\end{proposition}

\begin{proof}
Let $p=X_1\wedge\cdots\wedge X_k$ be an element of $\Rho_{\g,k}$. Recall that the Schouten bracket is defined by \[[p,\xi]:=\sum_{i=1}^n(-1)^i[X_i,\xi]\wedge X_1\cdots\widehat X_i\cdots\wedge X_n.\]By definition of the differential,

\begin{align*}
\delta m(\xi,p)&:=\sum_{i<j}(-1)^{i+j} m([X_i,X_j]X_1\cdots\widehat X_i\cdots\widehat X_j\cdots X_n)\\
&=\sum_{i=1}^n(-1)^im([X_i,\xi]\wedge X_1\cdots\widehat X_i\cdots\wedge X_n)&\text{since $p$ is in $\Rho_{\g,k}$}\\
&=m([p,\xi])
\end{align*}
\end{proof}

Now we come to one of the main results of this section.

\begin{theorem}
Suppose that we have a refined homotopy moment map $(f)$. Consider the $k$th component $f_k:\Rho_{\g,k}\to\Omega^{n-k}_S(M)$.  If $H^{k+1}(\g)=0$ then $f_k$ can be made equivariant.
\end{theorem}
\begin{proof}TO DO.

}
\del{
Wit the assumption that  $H_{k+1}(\g)=0$ then since $\delta\Sigma_k=0$, we have $\Sigma_k=\delta m_k$ for some $m_k\in \Rho_{\g,k}^\ast\otimes\Omega^{n-k}_\text{Sc}$. Since $m_k$ is closed, it follows that $d(f_k+m_k)=d(f_k)$ and so $f_k+m_k$ satisfies the moment properry. We show that $f_k + m_k$ is equivariant. Indeed

\begin{align*}
f_k([p,\xi])+m_k([p,\xi])&=f_k([p,\xi])+\Sigma(\xi,p)\\
&=\{f_k(p),f_1(\xi)\}+d(\xi\hk f_k(p))\\
&=\{f_k(p)+m_k(p),f_1(\xi)+m_k(p)\}+d(\xi\hk f_k(p))&\text{since $m_k(p)$ is closed}\\
&=\{f_k(p)+m_k(p),f_1(\xi)+m_k(p)\}+d(\xi\hk (f_k(p)+m_k(p))&\text{since $m_k(p)$ is a special form}
\end{align*}
Thus, the cocycle corresponding to $f_k+m_k$ vanishes showing that $f_k+m_k$ is equivariant.
\end{proof}

}

\del{

The way uniqueness is proved in the symplectic case is by noting that any two equivariant moment maps differ by something in $[\g,\g]^0=H^1(\g)$. Hence if $H^1(\g)=0$, any equivariant moment map is unique.  I have shown that in the Madsen and Swann setup, any two equivariant moment maps differ by something in $[\Rho_\g,\g]^0$ and if $H^2(\g)=0$ then this annhilator is zero and so moment maps are unique. (I believe this is a new proof for the existence and uniqueness of multi moment maps in the Madsenn and Swann setup)

We collect the results of this section into the following theorem. This generalizes the existence and uniqueness theorem of Madsen and Swann in \cite{ms} as our target spaces are now forms of arbitrary degree.

\begin{theorem}
If $H^k(\g)=0$ then the $k$th component of a refined homotopy moment map exists. This $k$th component is not necessarily equivariant; however, if it is, then it is unique. If $H^{k+1}(\g)=0$ then any non equivariant refined homotopy moment map can be made equivariant. Thus, if both $H^{k}(\g)=0$ and $H^{k+1}(\g)=0$ then an equivariant $kth$ component of a refined homotopy moment map exists and is unique.
\end{theorem}

We thus can say that

\begin{theorem} If $H^1(\g)=H^2(\g)=\cdots= H^{n+1}(\g)=0$ then a unique equivariant refined homotopy moment map exists.
\end{theorem}

\begin{question}
Given any $n\in \N$ does there exist a Lie group satisfying $H_1(\g)=\cdots= H_n(\g)=0$? In the language of Madsen and Swann, for any $n\in\N$, does there exist a $(k_1,\cdots,k_n)$-trivial Lie group?
\end{question}

}

\subsection{Uniqueness of Equivariant Weak Moment Maps}

We first recall the results from symplectic geometry without explicit proof. A proof can be found by setting $n=1$ (i.e. the symplectic case) in our more general Theorem \ref{dunno2}. A proof can also be found in \cite{dasilva}.
Let $\g$ be a Lie algebra acting on a symplectic manifold $(M,\omega)$.
\begin{proposition}\label{H5}
If $f$ and $g$ are two equivariant moment maps, then $f-g$ is in $H^1(\g)$.
\end{proposition}
\begin{proof}
For $\xi,\eta\in\g$ we have that $(f-g)([\xi,\eta])=\{(f-g)(\xi),(f-g)(\eta)\}$ since $f$ and $g$ are equivariant. However, $(f-g)(\xi)$ is a constant function since both $f$ and $g$ are moment maps. The claim now follows since the Poisson bracket with a constant function vanishes.
\end{proof}
From Proposition \ref{H5} it immediately follows that
\begin{proposition}\label{H6}
If $H^1(\g)=0$ then equivariant moment moments are unique.
\end{proposition}

The following is a generalization of Proposition \ref{H5}.
\begin{proposition}\label{uniqueness in H^0}
If $f_k$ and $g_k$ are $k$-th components of two equivariant weak moment maps, then $f_k-g_k$ is in $H^0(\g,\Rho_{\g,k}^\ast\otimes\Omega^{n-k}_{\mathrm{cl}})$.
\end{proposition}

\begin{proof}
If $f_k$ and $g_k$ are equivariant then $(f_k-g_k)([\xi,p])=\L_{V_\xi}((f_k-g_k)(p))$. Moreover, $(f_k-g_k)(p)$ is closed since both $f_k$ and $g_k$ are moment maps.
\end{proof}

We now arrive at our generalization of Proposition \ref{H6}. Let $\g$ be a Lie algebra acting on an $n$-plectic manifold $(M,\omega)$.
\begin{theorem}\label{dunno2}
If $H^0(\g,\Rho_{\g,k}^\ast\otimes\Omega^{n-k}_{\mathrm{cl}})=0$, then equivariant weak $k$-moment maps are unique. In particular, if $H^0(\g,\Rho_{\g,k}^\ast\otimes\Omega^{n-k}_{\mathrm{cl}})=0$ for all $1\leq k \leq n$ then equivariant weak moment maps are unique. 
\end{theorem}
\begin{proof}
If $f_k$ and $g_k$ are two equivariant weak $k$-moment maps, then Proposition \ref{uniqueness in H^0} shows that $f_k-g_k$ is in $H^0(\g,\Rho_{\g,k}^\ast\otimes\Omega^{n-k}_{\mathrm{cl}})$.

\end{proof}

\begin{remark}
This theorem gives a generalization of the results of Madsen and Swann. Indeed, by taking $n=k$ we again arrive at Theorems 3.5 and 3.14 of \cite{ms} and \cite{MS} respectively.
\end{remark}
\del{
We make two more observations.

\begin{proposition}\label{HH1}
If $f_k$ and $g_k$ are two components of an equivariant homotopy moment map $(f)$, then $f_k-g_k$ is in $[\Rho_{\g,k},\g]^0$. 
\end{proposition}
\begin{proof}
Let $[p,\xi]$ be an artbitrary element of $[\Rho_{\g,k},\g]$. Then we have that 
\begin{align*}
(f_k-g_k)([p,\xi])&=\{(f_k-g_k)(p),(f_1-g_1)(\xi)\}&\text{by equivariance}\\
&=0&\text{since $f_k(p)-g_k(p)$ is closed}
\end{align*}
\end{proof}
\begin{proposition}\label{HH2}
If $H^{k}(\g)=0$ then $[\Rho_{\g,k},\g]^0=0$
\end{proposition}

\begin{proof}
We have already shown that $\Rho_{\g,k}=[\Rho_{\g,k},\g]$. Since $\dim([\Rho_{\g,k},\g])+\dim([\Rho_{\g,k},\g])^0=\dim(\Rho_{\g,k})$ the claim now follows.
\end{proof}

Our generalization of Proposition \ref{H5} is:

\begin{proposition}
If $H^k(g)=0$ then the $k$-th components of homotopy moment maps are unique. 
\end{proposition}

\begin{proof}
This follows from Propositions \ref{HH1} and \ref{HH2}.
\end{proof}

Therefore, if an equivariant moment map exists and $H^k(\g)=0$, it is unique.
}

\newpage
\section{Multisymplectic Symmetries and Conserved Quantities}

In this section we give a definition of conserved quantities and continuous symmetries on multisymplectic manifolds. In symplectic geometry, the Poisson bracket plays a large role in the discussion of conserved quantities. To that end, we first generalize the Poisson bracket to multisymplectic geometry. \del{We will always work with a fixed multi-Hamiltonian system $(M,\omega,H)$. We let $X_H$ denote the unique Hamiltonian vector field corresponding to $H$. }

\subsection{A Generalized Poisson Bracket}
We first extend the notion of a Hamiltonian $(n-1)$-form to arbitrary forms of degree $\leq n-1$.
\begin{definition}We call
 \[\Omega^{n-k}_{\text{Ham}}(M):=\{\alpha\in\Omega^{n-k}(M) ; \text{ there exists $X_\alpha\in\Gamma(\Lambda^k(TM))$ with $d\alpha=-X_\alpha\hk\omega$ } \} \]the set of Hamiltonian $(n-k)$-forms. For a Hamiltonian $(n-k)$-form $\alpha$, we call $X_\alpha$ a corresponding Hamiltonian $k$-vector field (or multivector field if $k$ is not explicit).
 
We call 
 \[\X^k_{\text{Ham}}(M) := \{X\in\Gamma(\Lambda^k(TM)) ; X\hk\omega \text{ is exact} \} \] the set of Hamiltonian $k$-vector fields. We will refer to a primitive of $X\hk\omega$ as a corresponding Hamiltonian $(n-k)$-form. 
 
 \end{definition}

Of course, given a Hamiltonian $(n-k)$-form, it does not necessarily have a unique associated Hamiltonian multivector field. Moreover, a Hamiltonian $k$-vector field doesn't necessarily have a unique corresponding Hamiltonian $(n-k)$-form. However, the following is clear:
\begin{proposition}
\label{kernel}
For $\alpha\in\Omega^{n-k}_{\mathrm{Ham}}(M)$,  any two of its Hamiltonian $k$-vector fields differ by an element in the kernel of $\omega$. Conversely, for $X\in\X^k_{\mathrm{Ham}}(M)$, any two of its Hamiltonian forms differ by a closed form. 
\end{proposition}

Proposition \ref{kernel} motivates consideration of the following spaces. Let  $\widetilde{\X}^k_{\text{Ham}}(M)$ denote the quotient space of $\X^k_{\text{Ham}}(M)$ by elements in the kernel of $\omega$. Let $\widetilde{\Omega}^{n-k}_{\text{Ham}}(M)$ denote the quotient of $\Omega^{n-k}_{\text{Ham}}$ by closed forms. We let \[\Omega_{\text{Ham}}(M)=\oplus_{k=0}^{n-1}\Omega^k_{\text{Ham}}(M)\]and \[\widetilde\Omega_{\text{Ham}}(M)=\oplus_{k=0}^{n-1}\widetilde\Omega^k_{\text{Ham}}(M).\] Similarly, we let \[\X_{\text{Ham}}(M)=\oplus_{k=0}^{n-1}\X^k_{\text{Ham}}(M)\] and \[\widetilde\X_{\text{Ham}}(M)=\oplus_{k=0}^{n-1}\widetilde\X^k_{\text{Ham}}(M).\] It is clear that the map from $\widetilde{\Omega}^{n-k}_{\text{Ham}}(M)$ to $\widetilde{\X}^k_{\text{Ham}}(M)$ given by $[\alpha]\mapsto [X_\alpha]$ is a bijection.

\begin{proposition}
 The vector spaces $\widetilde{\Omega}^{n-k}_{\mathrm{Ham}}(M)$ and $\widetilde{\X}^k_{\mathrm{Ham}}(M)$ are isomorphic. 
\end{proposition}

Later on, we will see that there are graded Lie brackets on the vector spaces $\widetilde{\Omega}_{\text{Ham}}(M)$ and $\widetilde{\X}_{\text{Ham}}(M)$ making them isomorphic as graded Lie algebras.

The next proposition will be used to show that certain statements about a Hamiltonian form are independent of the choice of the corresponding Hamiltonian multivector field.
\begin{proposition}
\label{well defined}
If $\kappa$ is in the kernel of $\omega$, then for any $X\in\X^k_{\mathrm{Ham}}(M)$ we have $[X,\kappa]\ \hk \ \omega=0$.
\end{proposition}
\begin{proof}
Since $X\in\X^k_{\mathrm{Ham}}(M)$ by definition $\L_X\omega=0$. Using equation (\ref{bracket hook}) together with the fact that $\kappa\hk\omega=0$ we have \[[X,\kappa]\hk\omega=(-1)^{k(k+1)}\L_{X}(\kappa\hk\omega)-\kappa\hk\L_{X}\omega=0.\]
\end{proof}

The next proposition shows that, as in symplectic geometry, any Hamiltonian multivector field preserves $\omega$.

\begin{proposition}
\label{preserve}
For $\alpha\in\Omega^{n-k}_{\mathrm{Ham}}(M)$ we have that $\L_{X_\alpha}\omega=0$ for all Hamiltonian multivector fields $X_\alpha$ of $\alpha$.
\end{proposition}

\begin{proof}
Let $X_\alpha$ be a Hamiltonian multivector field. We have that \begin{align*}
\L_{X_\alpha}\omega&=\L_{X_\alpha}\omega\\
&=d(X_\alpha\hk\omega)-(-1)^kX_\alpha\hk d\omega&\text{by equation (\ref{Lie})}\\
&=d(X_\alpha\hk\omega)&\text{since $d\omega=0$}\\
&=-d(d\alpha)&\text{by definition}\\
&=0.
\end{align*}
\end{proof}

\del{
The next lemma is a generalization of Lemma 3.7 of \cite{rogers}. It will be used repeatedly throughout this paper.

\begin{lemma}
\label{Rogers}
Fix $m\geq 2$. Let $\alpha_1,\ldots,\alpha_m$ be Hamiltonian forms of degree $n-k_1,\ldots, n-k_m$ respectively. Let $X_1,\ldots, X_m$ be arbtirary Hamiltonian vector fields for $\alpha_1,\ldots,\alpha_m$. Note that $X_j$ has degree $k_j$. The following holds:
\begin{align*}&d(X_m\hk\cdots\hk X_1\hk\omega)=\\
&(-1)^m(-1)^{k_1\cdots k_m}\sum_{1\leq i<j\leq m}(-1)^{i+j}X_m\hk\cdots\hk\widehat X_i\hk\cdots\hk\widehat X_j\hk\cdots\hk X_1\hk[X_i,X_j]\hk\omega
\end{align*}
\end{lemma}
\begin{proof}
Brutal
\end{proof}
}

We now put in a structure analogous to the Poisson bracket in Hamiltonian mechanics, which has analogous graded properties. 

Given $\alpha\in\Omega^{n-k}_{\text{Ham}}(M)$ and $\beta\in\Omega^{n-l}_{\text{Ham}}(M)$, a first attempt would be to define their generalized bracket to be \[\{\alpha,\beta\} := X_\beta\hk X_\alpha\hk \omega,\] mimicking the Poisson bracket in symplectic geometry. However, we can see right away that this bracket is not graded anti-commutative since $\{\alpha,\beta\}=(-1)^{kl}\{\beta,\alpha\}$. Hence, we modify our grading of the Hamiltonian forms, following the work done in \cite{dropbox}.

\begin{definition}
Let $\mathcal{H}^p(M)=\Omega^{n-p+1}_{\text{Ham}}(M)$. That is, we are assigning the grading of $\alpha\in\Omega^{n-k}_{\text{Ham}}(M)$ to be $|\alpha|=k+1$. For $\alpha\in\Omega^{n-k}_{\text{Ham}}(M)$ and $\beta\in\Omega^{n-l}_{\text{Ham}}(M)$ (i.e. $\alpha\in\mathcal{H}^{k+1}(M)$ and $\beta\in\mathcal{H}^{l+1}(M)$) we define their (generalized) Poisson bracket to be \begin{align*}\{\alpha,\beta\}&:=(-1)^{|\beta|}X_\beta\hk X_\alpha\hk\omega\\
&=(-1)^{l+1}X_\beta\hk X_\alpha\hk\omega.
\end{align*}

Notice that this bracket is well defined follows directly from Proposition \ref{kernel}.
\end{definition}

With this new grading, the generalized Poisson bracket is graded commutative. 

\begin{proposition}
\label{skew graded}
 Let $\alpha$ be a form of grading $|\alpha|=k+1$ and $\beta$ a form of grading $|\beta|=l+1$. That is,  $\alpha\in\Omega^{n-k}_{\text{Ham}}(M)$ and $\beta\in\Omega^{n-l}_{\text{Ham}}(M)$. Then we have that \[\{\alpha,\beta\}=-(-1)^{|\alpha||\beta|}\{\beta,\alpha\}.\]

\end{proposition}

\begin{proof}
By definition,
\begin{align*}
\{\alpha,\beta\}&=(-1)^{l+1}X_\beta\hk X_\alpha\hk\omega\\
&=(-1)^{l+1}(-1)^{kl}X_\alpha\hk X_\beta\hk\omega\\
&=-(-1)^{l(k+1)}X_\alpha\hk X_\beta\hk\omega\\
&=-(-1)^{(l+1)(k+1)+k+1}X_\alpha\hk X_\beta\hk\omega\\
&=-(-1)^{|\alpha||\beta|}(-1)^{k+1}X_\alpha\hk X_\beta\hk\omega\\
&=-(-1)^{|\alpha||\beta|}\{\beta,\alpha\}.
\end{align*}
\end{proof}

The next lemma shows that the bracket of two Hamiltonian forms is Hamiltonian. In symplectic geometry, we have $X_{\{f,g\}}=[X_f,X_g]$ (or $X_{\{f,g\}}=-[X_f,X_g]$ if the defining equation for a Hamiltonian vector field is $X_\alpha\hk\omega=d\alpha$). In multisymplectic geometry we have

\begin{lemma}
\label{Poisson is Schouten}
For $\alpha\in\Omega^{n-k}_{\text{Ham}}(M)$ and $\beta\in\Omega^{n-l}_{\text{Ham}}(M)$ their bracket $\{\alpha,\beta\}$ is in $\Omega^{n+1-k-l}_{\text{Ham}}(M)$. That is, $\{\alpha,\beta\}$ is a Hamiltonian form with grading $|\{\alpha,\beta\}|=k+l-2$. More precisely, we have that  $[X_\alpha,X_\beta]$ is a Hamiltonian vector field for $\{\alpha,\beta\}$.
\end{lemma}

\begin{proof}
We have that
\begin{align*}
[X_\alpha,X_\beta]\hk\omega&=-X_\beta\hk d(X_\alpha\hk\omega)+(-1)^ld(X_\beta\hk X_\alpha\hk\omega)\\
&\quad{}+(-1)^{kl+k}X_\alpha\hk X_\beta\hk dw-(-1)^{kl+k+l}X_\alpha\hk d(X_\beta\hk\omega)&\text{by equation (\ref{interior equation})}\\
&=(-1)^ld(X_\beta\hk X_\alpha\hk\omega)\\
&=-d(\{\alpha,\beta\}).
\end{align*}

\end{proof}

We now investigate the Jacobi identity for this bracket. In \cite{dropbox} it was mentioned that the graded Jacobi identity holds up to a closed form. We now show that the graded Jacobi identity holds up to an exact term. 
\begin{proposition}\bf{(Graded Jacobi.)} \rm 
\label{Jacobi} \ Fix $\alpha\in\Omega^{n-k}_{\text{Ham}}(M)$, $\beta\in\Omega^{n-l}_{\text{Ham}}(M)$ and $\gamma\in\Omega^{n-m}_{\text{Ham}}(M)$. Let $X_\alpha, X_\beta$ and $X_\gamma$ denote arbitrary Hamiltonian multivector fields for $\alpha,\beta$ and $\gamma$ respectively. Then we have that
\[\sum_{\text{cyclic}}(-1)^{|\alpha||\gamma|}\{\alpha,\{\beta,\gamma\}\} = (-1)^{|\beta||\gamma|+|\beta||\alpha|+|\beta|}d(X_\alpha\hk X_\beta\hk X_\gamma\hk\omega).\] 

\end{proposition}
\begin{proof}
By definition, we have that \[\{\alpha,\beta\}=(-1)^{|\beta|}X_\beta\hk X_\alpha\hk\omega=(-1)^{|\beta|+1}X_\beta\hk d\alpha.\] Since $X_\beta$ is in $\Lambda^{|\beta|+1}(TM)$, by (\ref{Lie}) it follows that \begin{align}\{\alpha,\beta\}&\nonumber=(-1)^{|\beta|+1}(-1)^{|\beta|+1}(d(X_\beta\hk\alpha)-\L_{X_\beta}\alpha)\\&\label{bracket 1}=d(X_\beta\hk\alpha)-\L_{X_\beta}\alpha.\end{align}Thus, \begin{align*}
\{\alpha,\{\beta,\gamma\}\}&=(-1)^{|\beta||\gamma|+1}\{\alpha,\{\gamma,\beta\}\}\\
&=(-1)^{|\beta||\gamma|+1+(|\beta|+|\gamma|)|\alpha|+1}\{\{\gamma,\beta\},\alpha\}\\
&=(-1)^{|\beta||\gamma|+1+(|\beta|+|\gamma|)|\alpha|+1}(d(X_\alpha\hk\{\gamma,\beta\})-\L_{X_\alpha}\{\gamma,\beta\})&\text{by $(\ref{bracket 1})$}\\
&=(-1)^{|\beta||\gamma|+|\beta||\alpha|+|\gamma||\alpha|}(d(X_\alpha\hk\{\gamma,\beta\})-\L_{X_\alpha}(d(X_\beta\hk\gamma))+\L_{X_\alpha}\L_{X_\beta}\gamma).
\end{align*} Hence, \begin{equation}\label{first Jacobi}(-1)^{|\alpha||\gamma|}\{\alpha,\{\beta,\gamma\}\}=(-1)^{|\beta||\gamma|+|\beta||\alpha|}\left(d(X_\alpha\hk\{\gamma,\beta\})-\L_{X_\alpha}d(X_\beta\hk\gamma)+\L_{X_\alpha}\L_{X_\beta}\gamma\right).\end{equation} Similarly, since $|\{\gamma,\alpha\}|=|\gamma|+|\alpha|-2$, we have that \begin{align*}\{\beta,\{\gamma,\alpha\}\}&=(-1)^{(|\gamma|+|\alpha|)|\beta|+1}\{\{\gamma,\alpha\},\beta\}\\&=(-1)^{|\gamma||\beta|+|\alpha||\beta|+1}(d(X_\beta\hk\{\gamma,\alpha\})-\L_{X_\beta}d(X_\alpha\hk\gamma)+\L_{X_\beta}\L_{X_\alpha}\gamma).&\text{by $(\ref{bracket 1})$}\\
\end{align*}Hence, \begin{equation}\label{second Jacobi}(-1)^{|\beta||\alpha|}\{\beta,\{\gamma,\alpha\}\}=(-1)^{|\gamma||\beta|+1}\left(d(X_\beta\hk\{\gamma,\alpha\})-\L_{X_\beta}d(X_\alpha\hk\gamma)+\L_{X_\beta}\L_{X_\alpha}\gamma\right).\end{equation} Lastly, using Lemma \ref{Poisson is Schouten} and $(\ref{bracket 1})$, we have that  \begin{equation}\label{third Jacobi} (-1)^{|\gamma||\beta|}\{\gamma,\{\alpha,\beta\}\}=(-1)^{|\gamma||\beta|}\left(d([X_\alpha,X_\beta]\hk\gamma)-\L_{[X_\alpha,X_\beta]}\gamma\right).\end{equation} Now we notice that by (\ref{L bracket}) the terms involving $\L_{X_\alpha}\L_{X_\beta}\gamma$ from (\ref{first Jacobi}) , $\L_{X_\beta}\L_{X_\alpha}\gamma$ from (\ref{second Jacobi}) and $\L_{[X_\alpha,X_\beta]}\gamma$ from (\ref{third Jacobi}) add to zero. Hence we now consider the term $(-1)^{|\gamma||\beta|}d([X_\alpha,X_\beta]\hk\gamma)$ from (\ref{third Jacobi}). Using equations (\ref{bracket hook}), (\ref{Lie}), and (\ref{dL}) it follows

\begin{align*}
d(&[X_\alpha,X_\beta]\hk\gamma)=d\left((-1)^{|\alpha|(|\beta|+1)}\L_{X_\alpha}(X_\beta\hk\gamma)-X_\beta\hk\L_{X_\alpha}\gamma\right)\\
&=(-1)^{|\alpha|(|\beta|+1)}d\left(\L_{X_\alpha}(X_\beta\hk\gamma)\right)-d\left(X_\beta\hk d(X_\alpha\hk\gamma)\right)+(-1)^{|\alpha|+1}d\left(X_\beta\hk X_\alpha\hk d\gamma\right)\\
&=(-1)^{|\alpha||\beta|}\L_{X_\alpha}d(X_\beta\hk\gamma)-\L_{X_\beta} d(X_\alpha\hk\gamma)+(-1)^{|\alpha|}d(X_\beta\hk X_\alpha\hk X_\gamma\hk\omega).\end{align*}
Thus \begin{equation}\label{fourth Jacobi}\begin{aligned}
(-1)^{|\gamma||\beta|}d([X_\alpha,X_\beta]\hk\gamma)&=(-1)^{|\alpha||\beta|+|\gamma||\beta|}\L_{X_\alpha}d(X_\beta\hk\gamma)-(-1)^{|\gamma||\beta|}\L_{X_\beta}d(X_\alpha\hk\gamma)\\
&\quad{}+(-1)^{|\gamma||\beta|+|\alpha|}d(X_\beta\hk X_\alpha\hk X_\gamma\hk\omega).
\end{aligned}\end{equation}Thus, upon adding (\ref{first Jacobi}), (\ref{second Jacobi}) and (\ref{fourth Jacobi}) we are left with 

\begin{align*}
(-1)^{|\alpha||\gamma|}&\{\{\alpha,\beta\},\gamma\}+(-1)^{|\beta||\alpha|}\{\{\beta,\gamma\},\alpha\}+(-1)^{|\gamma||\beta|}\{\{\gamma,\alpha\},\beta\}  \\
&=(-1)^{|\beta||\gamma|+|\beta||\alpha|}d(X_\alpha\hk\{\gamma,\beta\})+(-1)^{|\gamma||\beta|+1}d(X_\beta\hk\{\gamma,\alpha\})\\
&\quad{}+(-1)^{|\gamma||\beta|+|\alpha|}(X_\beta\hk X_\alpha\hk X_\gamma\hk\omega)\\
&=(-1)^{|\beta||\gamma|+|\beta||\alpha|+|\beta|}d(X_\alpha\hk X_\beta\hk X_\gamma\hk\omega)+(-1)^{|\gamma||\beta|+1+|\alpha|}d(X_\beta\hk X_\alpha\hk X_\gamma\hk\omega)\\
&\quad{} +(-1)^{|\gamma||\beta|+|\alpha|}d(X_\beta\hk X_\alpha\hk X_\gamma\hk\omega)\\
&=(-1)^{|\beta||\gamma|+|\beta||\alpha|+|\beta|}d(X_\alpha\hk X_\beta\hk X_\gamma\hk\omega).
\end{align*}

\end{proof}

Summing up the results of this section we have confirmed Theorem 4.1 of \cite{dropbox}: 

\begin{proposition}
\label{graded Lie algebra of forms}
With the above grading, $(\widetilde\Omega_{\text{Ham}}(M),\{\cdot,\cdot\})$ is a graded Lie algebra.
\end{proposition}

\begin{proof}
The bracket is well defined on $\widetilde\Omega_{\text{Ham}}(M)$ since if $\gamma$ is closed then $\{\gamma,\alpha\}=(-1)^kX_\alpha\hk d\gamma=0$. Clearly the bracket is bilinear. Proposition \ref{skew graded} shows that the bracket is skew graded and Proposition \ref{Jacobi} shows that it satisfies the Jacobi identity.
\end{proof}

\subsection{Conserved Quantities and their Algebraic Structure}
We now turn our attention towards conserved quantities. In symplectic geometry, a conserved quantity is a $0$-form $\alpha$ that is preserved by the Hamiltonian, i.e. satisfying $\L_{X_H}\alpha=0$. A generalization of this definition to multisymplectic geometry was given in \cite{cq}; however, we add the requirement that a conserved quantity is also Hamiltonian. By adding in this requirement, we can now take the generalized Poisson bracket of two conserved quantities, as in symplectic geometry. 

We work with a fixed multi-Hamiltonian system $(M,\omega,H)$ with $\omega\in\Omega^{n+1}(M)$ and $H\in\Omega^{n-1}_{\text{Ham}}(M)$, and let $X_H$ denote the corresponding Hamiltonian vector field.

\begin{definition}
A Hamiltonian $(n-k)$-form $\alpha$ in $\Omega^{n-k}_{\text{Ham}}(M)$ is called 
\begin{itemize}
\item locally conserved if $\L_{X_H}\alpha$ is closed,
\item globally conserved if $\L_{X_H}\alpha$ is exact, 
\item strictly conserved if $\L_{X_H}\alpha=0$.
\end{itemize}

As in \cite{cq}, we denote the space of locally, globally, and strictly conserved forms by $\mathcal{C}_{\text{loc}}(X_H)$, $\mathcal{C}(X_H)$, and $\mathcal{C}_{\text{str}}(X_H)$ respectively.  We will let $\widetilde{\mathcal{C}}_{\text{loc}}(X_H)$, $\widetilde{\mathcal{C}}(X_H)$ and $\widetilde{\mathcal{C}}_{\text{str}}(X_H)$ denote the conserved quantities modulo closed forms. 
Note that $\mathcal{C}_{\text{str}}(X_H)\subset\mathcal{C}(X_H)\subset\mathcal{C}_{\text{loc}}(X_H)$ and $\widetilde{\mathcal{C}}_{\text{str}}(X_H)\subset\widetilde{\mathcal{C}}(X_H)\subset\widetilde{\mathcal{C}}_{\text{loc}}(X_H)$.
\end{definition}

The next lemma is a generalization of Lemma 1.7 in \cite{cq}.

\begin{lemma}
\label{conserved interior}
Fix a Hamiltonian $(n-k)$-form $\alpha\in\Omega^{n-k}_{\text{Ham}}(M)$. If $\alpha$ is a local conserved quantity then $[X_\alpha,X_H]\hk \ \omega=0$, for some (or equivalently every) Hamiltonian multivector field $X_\alpha$ of $\alpha$. Conversely, if $[X_\alpha,X_H]\hk\omega=0$  then $\alpha$ is locally conserved.
\end{lemma}

\begin{proof}
Let $X_\alpha$ be an arbitrary Hamiltonian multivector field of $\alpha$. We have that 
\begin{align*}
[X_\alpha,X_H]\hk\omega&=-X_H\hk d(X_\alpha\hk\omega)-d(X_H\hk X_\alpha\hk\omega)\\
&\quad{}+X_\alpha\hk(d(X_H\hk\omega))+X_H\hk X_\alpha\hk d\omega &\text{by Prop \ref{interior}}\\
&=- d(X_H\hk X_\alpha\hk\omega)\\
&=-\L_{X_H}(X_\alpha\hk\omega)&\text{by (\ref{Lie})}\\
&=d\L_{X_H}\alpha&\text{by (\ref{dL}).}
\end{align*}
\end{proof}

Recall the following standard result from Hamiltonian mechanics: If $H$ is a Hamiltonian on a symplectic manifold and $f$ and $g$ are two strictly conserved quantities, i.e. $\{f,H\}=0=\{g,H\}$, then $\{f,g\}$ is strictly conserved. This is because $\L_{X_H}\{f,g\}=\{\{f,g\},H\}=0$ by the Jacobi identity.  Moreover, if $f$ and $g$ are local or global conserved quantities (meaning that their bracket with $H$ is constant) then again $\{f,g\}$ is strictly conserved by the Jacobi identity together with the fact that the Poisson bracket with a constant function vanishes. 

The next proposition generalizes these results to multisymplectic geometry. 

\begin{proposition}
\label{Poisson is strictly conserved}
The bracket of two conserved quantities is a strictly conserved quantity. 
\end{proposition}
\begin{proof}
Let $\alpha\in\Omega^{n-k}_{\text{Ham}}(M)$ and $\beta\in\Omega^{n-l}_{\text{Ham}}(M)$ be any two conserved quantities. Let $X_\alpha$ and $X_\beta$ denote arbitrary Hamiltonian multivector fields corresponding to $\alpha$ and $\beta$ respectively. By definition,
\begin{align*}
\L_{X_H}\{\alpha,\beta\}&=(-1)^{|\beta|}\L_{X_H}X_\beta\hk X_\alpha\hk\omega.
\end{align*}By (\ref{bracket hook}) together with Lemma \ref{conserved interior} we see that we can commute the Lie derivative and interior product. Hence, \[\L_{X_H}\{\alpha,\beta\}=(-1)^{|\beta|}X_\alpha\hk X_\beta\hk\L_{X_H}\omega.\] The claim now follows since $\L_{X_H}\omega=0$, by Proposition \ref{preserve}. 
\del{
 We have that, 
 \begin{align*}
\L_{X_H}\{\alpha,\beta\}&=d(X_H\hk(\{\alpha,\beta\}))+ X_H\hk X_{\{\alpha,\beta\}}\hk\omega\\
&=(-1)^{|\beta||\alpha|+1}d(X_H\hk\{\beta,\alpha\})+X_H\hk X_{\{\alpha,\beta\}}\hk\omega\\
&=(-1)^{|\beta||\alpha|+1+|\alpha|}d(X_H\hk X_\alpha\hk X_\beta\hk\omega)+X_H\hk X_{\{\alpha,\beta\}}\hk\omega\\
&=(-1)^{|\beta||\alpha|+1+|\alpha|}d(X_H\hk X_\alpha\hk X_\beta\hk\omega)-\{\{\alpha,\beta\},H\}\\
&=(-1)^{|\beta||\alpha|+1+|\alpha|}d(X_H\hk X_\alpha\hk X_\beta\hk\omega)+\{H,\{\alpha,\beta\}\}
\end{align*}
where in the last line we used the graded skew symmetry of the bracket together with $|H|=2$. 
Now we notice that, by proposition \ref{Poisson is Schouten}, we have $\{\{\alpha,H\},\beta\}=(-1)^{|\beta|}X_\beta\hk[X_\alpha,X_H]\hk\omega$, so that by lemma \ref{conserved interior} \[\{\{\alpha,H\},\beta\}=0=\{\{\beta,H\},\alpha\}.\] Hence, by the graded Jacobi identity (proposition \ref{Jacobi}) together with the fact that $|H|=2$,  we obtain \[\{H,\{\alpha,\beta\}\}=(-1)^{|\alpha||\beta|+|\alpha|}d(X_H\hk X_\alpha\hk X_\beta). \] It now follows that \[\L_{X_H}\{\alpha,\beta\}=0\] as desired. 
}
\end{proof}

As a consequence, we obtain:

\begin{proposition}
The spaces $(\widetilde{\mathcal{C}}_{\text{loc}}(X_H),\{\cdot,\cdot\})$, $(\widetilde{\mathcal{C}}(X_H),\{\cdot,\cdot\})$, and $(\widetilde{\mathcal{C}}_{\text{str}}(X_H),\{\cdot,\cdot\})$ are graded Lie subalgebras of $(\widetilde\Omega_{\text{Ham}}(M),\{\cdot,\cdot\})$.
\end{proposition}
\begin{proof}
Proposition \ref{Poisson is strictly conserved} shows that each of these spaces is preserved by the bracket. The claim now follows from Proposition \ref{graded Lie algebra of forms}.
\end{proof}
We conclude this section by showing that the Hamiltonian forms constitute an $L_\infty$-subalgebra of the Lie $n$-algebra of observables. Moreover, restricting a homotopy moment map to the Lie kernel gives an $L_\infty$-morphism into this $L_\infty$-algebra:

Let $\widehat L_\infty(M,\omega)=(\widehat L,\{l_k\})$ denote the graded vector space $\widehat L_k=\Omega^{n-1-k}_{\text{Ham}}(M)$ for $k=0,\ldots, n-1$, together with the maps $l_k$ from the Lie $n$-algebra of observables. 

\begin{theorem}
\label{L infinity subalgebra}
The space $(\widehat L,\{l_k\})$ is an $L_\infty$-subalgebra of $(L,\{l_k\})$.
\end{theorem}

\begin{proof}
We note that $l_1$ preserves $\widehat L$ since closed forms are Hamiltonian. For $k>1$, since $l_k$ vanishes on elements of positive degree we need only consider \[l_k(\alpha_1,\cdots,\alpha_k)=-(-1)^{\frac{k(k+1)}{2}}X_{\alpha_k}\hk\cdots\hk X_{\alpha_1}\hk\omega,\] where $\alpha_1,\ldots,\alpha_k$ are Hamiltonian $(n-1)$-forms. By Lemma \ref{rogers 3.7} we see that $l_k(\alpha_1,\ldots,\alpha_k)$ is a Hamiltonian $(n+1-k)$-form.
\end{proof}

\begin{proposition}
\label{conserved quantities are L infinity}
The spaces $\mathcal{C}(X_H)\cap \widehat L$ , $\mathcal{C}_{\text{loc}}(X_H)\cap\widehat L$, and $\mathcal{C}_{\text{str}}(X_H)\cap\widehat L$ are $L_\infty$-subalgebras of $\widehat L_\infty(M,\omega)$.
\end{proposition}

\begin{proof}

The proof is analogous to the proof of Proposition 1.15 in \cite{cq}. Since the proof is short, we include it here. From Theorem \ref{L infinity subalgebra} we see that each of the spaces $\mathcal{C}(X_H)\cap \widehat L$, $\mathcal{C}_{\text{loc}}(X_H)\cap\widehat L$, and $\mathcal{C}_{\text{str}}(X_H)\cap\widehat L$ are closed under each $l_k$. It remains to show that for Hamiltonian $(n-1)$-forms $\alpha_1,\cdots,\alpha_k$ which are (locally, globally, strictly) conserved, that $l_k(\alpha_1,\cdots,\alpha_k)$ is (locally, globally, strictly) conserved. Indeed, \[\L_{X_H}l_k(\alpha_1,\cdots,\alpha_k)=\L_{X_H}X_{\alpha_k}\hk\cdots\hk X_{\alpha_1}\hk\omega.\] Using equation (\ref{bracket hook}) together with Lemma \ref{conserved interior} we see that we can commute the Lie derivative and interior product. The claim then follows since $\L_{X_H}\omega=0$.
\end{proof}

\subsection{Continuous Symmetries and their Algebraic Structure}

Fix a multi-Hamiltonian system $(M,\omega,H)$. Our motivation for the definition of a continuous symmetry comes from Hamiltonian mechanics; we directly generalize the definition. As is the case with conserved quantities, we define three types of continuous symmetry. 

\begin{definition}
We say that a Hamiltonian multivector field $X\in\X_{\text{Ham}}(M)$ is 

\begin{itemize}
\item a local continuous symmetry if $\L_XH$ is closed,
\item a global continuous symmetry if $\L_XH$ is exact,
\item a strict continuous symmetry if $\L_XH=0$.
\end{itemize}
Note that a continuous symmetry automatically preserves $\omega$ by Proposition \ref{preserve}. We denote the space of local, global, and strict continuous symmetries by $\mathcal{S}_{\text{loc}}(H)$, $\mathcal{S}(H)$, and $\mathcal{S}_{\text{str}}(H)$ respectively. Moreover, we let $\widetilde{\mathcal{S}}_{\text{loc}}(H)$, $\widetilde{\mathcal{S}}(H)$, and $\widetilde{\mathcal{S}}_{\text{str}}(H)$ denote the quotient by the kernel of $\omega$. 

We will say that a multivector field $X$ is a weak (local, global, strict) continuous symmetry if $\L_X\omega=0$ and $\L_XH$ is closed, exact, or zero respectively. That is, a weak continuous symmetry is not necessarily Hamiltonian.
\end{definition}

\del{

\begin{remark}
Motivation for calling these things symmetries comes from the fact that with this definition of symmetry sets up a correspondence with our notion of conserved quantity. 
\end{remark}

For a Lie group acting on a symplectic Hamiltonian system, we call each infinitesimal generator a symmetry. For example, on the standard phase space $(T^\ast\R^3,\omega)$ with Hamiltonian $H=K-U$, we would call $SO(3)$ a symmetry group since the flow of each infinitesimal generator preserves $H$. In the multisymplectic setup, $H$ is replaced with a higher degree form and the infinitesimal generators are replaced with the wedge products of infinitesimal generators.   

\begin{example}
Consider $\R^3$ with $2$-plectic form $\mu=\text{vol}$. Let $H=E$ be the electromagnetic field. Want charge to be conserved so find that symmetry.  We will see later that under Noether's theorem, the corresponding conserved quantity is charge.
\end{example}

\begin{remark} For the case of the multisymplectic manifold $(\R^n,\text{vol})$,  in sectino ?? we give a geometric interpretation of a symmetry. In particular, we will see that the symmetries of degree $k$ are in one-to-one correspondence to $k$-Lagrangian submanifolds of $\R^{2k}$.
\end{remark}
}

\del{
Proposition \ref{kernel} showed that any two Hamiltonian vector fields corresponding to a Hamiltonian form differ by something in the kernel of $\omega$. Conversely, we have that 

\begin{proposition}
If two symmetries differ by something in the kernel of $\omega$, then any of the conserved quantities they generate differ by a closed form. 
\end{proposition}

\begin{proof}
Suppose that $X$ and $Y$ are (local, global or strict) symmetries with arbitrary corresponding conserved quantities $\alpha$ and $\beta$ respectively.  Suppose that $X-Y$ is in the kernel of $\omega$. Then \[d\alpha=X\hk \omega=Y\hk\omega=d\beta\] and so $\alpha$ and $\beta$ differ by a closed term.
\end{proof}

}

\begin{proposition}
\label{graded Lie algebra of vector fields}
We have $(\X_{\text{Ham}}(M),[\cdot,\cdot])$ is a graded Lie subalgebra of  $(\Gamma(\Lambda^\bullet(TM)),[\cdot,\cdot])$. 
\end{proposition}
\begin{proof}
By equation (\ref{interior equation}) we see that $[X,Y]\hk\omega=(-1)^ld(X\hk Y\hk\omega)$. Hence the space of Hamiltonian multivector fields is closed under the Schouten bracket. 
\end{proof}

\begin{proposition}
\label{symmetry super algebra}
The spaces $\mathcal{S}_{\text{loc}}(H)$, $\mathcal{S}(H)$, and $\mathcal{S}_{\text{str}}(H)$ are graded Lie subalgebras of $(\X_{\text{Ham}}(M),[\cdot,\cdot])$.
\end{proposition}

\begin{proof}
We see that each of $\mathcal{S}_{\text{loc}}(H)$, $\mathcal{S}(H)$, and $\mathcal{S}_{\text{str}}(H)$ are closed under the Schouten bracket directly from equations (\ref{dL}) and (\ref{L bracket}).
\end{proof}
\del{
The next example shows that none of $\mathcal{S}_{\text{loc}}(H), \mathcal{S}(H)$ or $\mathcal{S}_{\text{str}}(H)$ are necessarily closed under the wedge product.
\begin{example}
To Do.
\end{example}
}
The next lemma generalizes Lemma 2.9 (ii) of \cite{cq}.
\begin{lemma}
\label{symmetry interior}
Let $Y\in\Gamma(\Lambda^k(TM))$. If $Y$ is a local continuous symmetry, then $[Y,X_H] \hk\  \omega=0$. Conversely, if $[Y,X_H]\hk\ \omega=0$ and $\L_Y\omega=0$, then $Y$ is a local continuous symmetry.
\end{lemma}

\begin{proof}
We have that 
\begin{align*}[Y,X_H]\hk\omega&=(-1)^{k+1}\L_Y(X_H\hk\omega)-X_H\hk\L_Y\omega&\text{by (\ref{bracket hook})}\\
&=(-1)^{k}\L_YdH&\text{since $\L_Y\omega=0$ and $X_H\hk\omega=-dH$}\\
&=-d\L_YH&\text{by (\ref{dL}).}
\end{align*}
\end{proof}

Recall that for a group $G$ acting on a manifold $M$ we had defined in equation $(\ref{S_k})$ the set $S_k:=\{V_p \ : \ p\in\Rho_{\g,k}\}$. Proposition \ref{infinitesimal generator of Schouten} showed that $S=\oplus S_k$ was a graded Lie algebra.  We now get the following.

\begin{proposition}
\label{symmetries are L infinity}The spaces  $\mathcal{S}_{\text{loc}}(H)\cap S$, $\mathcal{S}(H)\cap S$, and $\mathcal{S}_{\text{str}}(H)\cap S$ are graded Lie subalgebras of $S$. 
\end{proposition}
\begin{proof}
By Proposition \ref{symmetry super algebra} we have that the spaces of symmetries are preserved by the Schouten bracket. The claim now follows by Proposition \ref{infinitesimal generator of Schouten}.
\end{proof}

\del{

\begin{corollary}
If $Y$ is a (local,global, strict) symmetry, then $Y\hk\omega$ is a strict conserved quantity.
\end{corollary}

\begin{proof}e roles of $Y$ and $X_H$ in the above lemma, we have that \[[X_H,Y]\hk\omega=\L_{X_H}(Y\hk\omega).\]
\end{proof}

}

\newpage

 \section{Noether's Theorem in Multisymplectic Geometry}
In this section we show how Noether's theorem extends from symplectic to multisymplectic geometry. To see this generalization explicitly, we first recall how Noether's theorem works in symplectic geometry.

\subsection{Noether's Theorem in Symplectic Geometry}

In this section we briefly recall the notions from symplectic geometry. More information can be found in \cite{me}, for example. Let $(M,\omega,H)$ be a Hamiltonian system. That is $(M,\omega)$ is symplectic and $H$ is in $C^\infty(M)$. Noether's theorem gives a correspondence between symmetries and conserved quantities. If $f\in C^\infty(M)$ is a (local, global) conserved quantity then $X_f$ is a (local, global) continuous symmetry. Conversely, if a vector field $X_f$ is a (local, global) continuous symmetry, then $f$ is a (local, global) conserved quantity. Note that in the symplectic case, local and strict symmetries and conserved quantities are the same thing.

If $X$ is only a weak (local, global) continuous symmetry, then $\L_X\omega=0$ so that by the Cartan formula  around each point there is a neighbourhood $U$ and a function $f\in C^\infty(U)$ such that $X=X_f$ on $U$.  This function $f$ is a (local, global) conserved quantity in the Hamiltonian system $(U,\omega|_U,H|_U)$.

If we only consider the symmetries and conserved quantities coming from a moment map $\mu:\g\to C^\infty(M)$ then, under the assumption of an $H$-preserving group action, each symmetry $\xi$ has corresponding global conserved quantity $\mu(\xi)$ and vice versa. 

The rest of this subsection formalizes this, and the following sections will generalize it to multi-symplectic geometry.

Recall that an equivariant moment map gives a Lie algebra morphism between $(\g,[\cdot,\cdot])$ and $(C^\infty(M),\{\cdot,\cdot\})$.

\begin{proposition}
\label{difference}
Let $\mu:\g\to C^\infty(M)$ be a momentum map. For $\xi,\eta\in\g$ we have that $\mu([\xi,\eta])=\{\mu(\xi),\mu(\eta)\}+\text{constant}$. If the moment map is equivariant then  $\mu([\xi,\eta])=\{\mu(\xi),\mu(\eta)\}$.
\end{proposition}

\begin{proof}
See Theorem 4.2.8 of \cite{Marsden}.
\end{proof}
\del{
\begin{remark}
The constant in the above proposition actually has a specific form. It is given by a specific Lie-algebra cocycle. It turns out that this cocycle has a generalization to multisymplectic geometry, a topic being explored by the author in \cite{future}.
\end{remark}
}
\del{

First note that if two conserved quantities differ by a constant they give the same symmetry. Moreover if two symmetries differ by something in the kernel of $\omega$ then they give the same conserved quantity. These two results are trivial, but we state them as propositions for reference when we generalize to the multisymplectic case.

\begin{proposition}
\label{difference of cq}
If $f,g\in C^\infty(M)$ are conserved quantities and $f=g+c$ then $X_f=X_g$. 
\end{proposition}

\begin{proposition}
\label{difference of cs}
If $X$ and $Y$ are symmetries whose difference is in the kernel of $\omega$, then they generate conserved quantities that differ by a constant.
\end{proposition}
}
As stated above, it is clear that in the symplectic case $\mathcal{C}_{loc}(X_H)=\mathcal{C}_{str}(X_H)$ and $\mathcal{S}_{loc}(H)=\mathcal{S}_{str}(H)$. It is easily verified that the map $\alpha\mapsto X_\alpha$ is a Lie algebra morphism from  $(\mathcal{C}(X_H),\{\cdot,\cdot\})$ to $(\mathcal{S}(H),[\cdot,\cdot])$ and from $(\mathcal{C}_{loc}(X_H),\{\cdot,\cdot\})$ to $(\mathcal{S}_{loc}(H),[\cdot,\cdot])$. However, under the quotients this map turns into a Lie algebra isomorphism. 

\begin{proposition}
\label{graded Lie iso}
The map $\alpha \mapsto X_\alpha$ is a Lie algebra isomorphism from $(\widetilde{C}(X_H),\{\cdot,\cdot\})$ to $(\widetilde{S}(H),[\cdot,\cdot])$ and $(\widetilde{\mathcal{C}}_{loc}(X_H),\{\cdot,\cdot\})$ to $(\widetilde{\mathcal{S}}_{loc}(H),[\cdot,\cdot])$.\end{proposition}

As a consequence of this proposition, we can now see how a momentum map sets up a Lie algebra isomorphism between the symmetries and conserved quantities it generates. Let $C=\{\mu(\xi); \xi\in\g\}$ and $S=\{V_\xi;\xi\in\g\}$. Let $\widetilde C$ be the quotient of $C$ by constant functions. Let $\widetilde S$ denote the quotient of $S$ by the kernel of $\omega$. Since the kernel of $\omega$ is trivial, $S=\widetilde{S}$. Then we get an induced well defined Poisson bracket on $\widetilde C$ and an induced well defined Lie bracket on $\widetilde S$.  We thus get a Lie algebra isomorphism:

\begin{proposition}
\label{la isomorphism}
The map between $(\widetilde C,\{\cdot,\cdot\})$ and $(\widetilde S,[\cdot,\cdot])$ that sends $[V_\xi]$ to $[\mu(\xi)]$ is a Lie-algebra isomorphism.

\end{proposition}

With our newly defined notions of symmetry and conserved quantity on a multisympletic manifold, we now exhibit how these concepts generalize to the setup of multisymplectic geometry. 

\del{

If two conserved quantties differ by a constant, then they will generate the same symmetry. Conversely, if two symmetries differ by something in the kernel of $\omega$ (which is zero in the symplectic case) then they will generate the same conserved quantity.co-moment

A moment map gives us symmetries, each $\xi$ generates a symmetry and by Noether we see that the corresponding conserved quantity is $\mu(\xi)$. We know that the moment map is a Lie algebra homomorphism between. (123)

Consider symmetries of the form. Consider their quotient by constant functions. Consider conserved quantites of form...and their quotient by kernel of $\omega$. Then the Lie bracket/Poisson descends to well defined bracket. The moment map gives us the following Lie algebra isomorphism.  (1234)

}
\subsection{The Correspondence between Mutlisymplectic Conserved Quantities and Continuous Symmetries}

We first examine the correspondence between symmetries and conserved quantities on multi-Hamiltonian systems. We will make repeated use of the following equations. Fix $\alpha\in\Omega^{n-k}_{\mathrm{Ham}}(M)$. By definition we have that \[\{\alpha,H\}=-X_H\hk X_\alpha\hk\omega=X_{H}\hk d\alpha=\L_{X_H}\alpha-d(X_H\hk\alpha).\] But we also know that $\{\alpha,H\}=-\{H,\alpha\}$, since $|H|=2$. Thus, by definition of the Poisson bracket and equation (\ref{Lie}) we have that \[-\{H,\alpha\}=(-1)^kX_\alpha\hk X_H\hk\omega=(-1)^{k+1}X_\alpha\hk dH=\L_{X_\alpha}H-d(X_\alpha\hk H).\]Putting these together we obtain 
\begin{equation}
\label{Noether 1}
\L_{X_\alpha}H=d(X_\alpha\hk H)+\L_{X_H}\alpha-d(X_H\hk\alpha)
\end{equation}
and
\begin{equation}
\label{Noether 2}
\L_{X_H}\alpha=d(X_H\hk\alpha)+\L_{X_\alpha}H-d(X_\alpha\hk H).
\end{equation}
\del{
As in the case of symplectic geometry, a symmetry only gives a conserved quantity locally, and this will be resolved when we consider symmetries and conserved quantities coming from a homotopy moment map.}
\begin{theorem} \label{Noether theorem 1}If $\alpha\in\Omega^{n-k}_{\text{Ham}}(M)$ is a (local, global) conserved quantity then any corresponding Hamiltonian $k$-vector field is a (local, global) continuous symmetry. Conversely, if $A\in\Gamma(\Lambda^k(TM))$ is a (local, global) continuous symmetry, then any corresponding Hamiltonian form is a (local, global) conserved quantity.

\end{theorem}

\begin{proof} 
Consider $\alpha\in\Omega^{n-k}_{\text{Ham}}(M)$. Let $X_\alpha$ be an arbitrary Hamiltonian multivector field.  Then, by equation (\ref{Noether 1}) we have that
\[\L_{X_\alpha}H=d(X_\alpha\hk H)+\L_{X_H}\alpha-d(X_H\hk\alpha).\]

Thus, if $\alpha$ is a (local or global) conserved quantity then $X_\alpha$ is a (local or global) continuous symmetry.

Conversely, suppose that $A$ is a (local or global) continuous symmetry and let $\alpha$ be a corresponding Hamiltonian form. Following the same argument above, we have by equation (\ref{Noether 2})
\[\L_{X_H}\alpha=d(X_H\hk\alpha)+\L_{X_\alpha}H-d(X_\alpha\hk H)\]
\del{
In summary we have the following correspondences: \[\mathcal{S}_{\text{loc}}(H) \longleftrightarrow \mathcal{C}_{\text{loc}}(X_H) \ \ \ \ \ \ \ \ \ \ \mathcal{S}(H) \longleftrightarrow \mathcal{C}(X_H)\]and \[\mathcal{S}_{\text{str}}(H) \longleftrightarrow \mathcal{C}(X_H) \ \ \ \ \ \ \ \ \ \ \mathcal{C}_{\text{str}}(X_H) \longleftrightarrow \mathcal{S}(H).\]
 }
\end{proof}

The correspondence between strictly conserved quantities and strict continuous symmetries is a little bit different. We have that

\begin{corollary}
If $\alpha\in\Omega^{n-k}_{\mathrm{Ham}}(M)$ is a strictly conserved quantity then $X_\alpha$ is a global continuous symmetry. Conversely, if $A$ is a strict continuous symmetry then the corresponding Hamiltonian $(n-k)$-form $\alpha$ is a global conserved quantity. 

\end{corollary}

\begin{proof}
This follows from the proof of the above theorem.
\end{proof}

\begin{remark}
If we were to consider weak continuous symmetries in the above theorem, then by the Poincar\'e lemma, a continuous symmetry would still give a conserved quantity, but only in a neighbourhood around each point of the manifold.
\end{remark}

The following simple example exhibits the correspondence.
\begin{example}
Consider $M=\R^3$ with volume form $\omega=dx\wedge dy\wedge dz$, $H=-xdy$ and $\alpha=zdx$. Then $dH=-dx\wedge dy$ so that $X_H=\frac{\pd}{\pd z}$. Also, $d\alpha=dz\wedge dx$ and so $X_\alpha=\frac{\pd}{\pd y}$. By the Cartan formula, we have that \[\L_{X_\alpha}H=-dx +dx=0,\] which means that $X_\alpha\in\mathcal{S}_{\text{str}}(H)\subset\mathcal{S}(H)$. We also have that \[\L_{X_H}\alpha=d(X_H\hk \alpha)+\{\alpha,H\}=d(X_H\hk\alpha)-d(X_\alpha\hk H)=dx.\] That is, $\alpha\in \mathcal{C}(X_H)$. Thus $\alpha$ is a global conserved quantity and $X_\alpha$ is a global continuous symmetry.
\end{example}

\subsection{Weak Moment Maps as Morphisms}

We work with a fixed multi-Hamiltonian system $(M,\omega,H)$ with acting symmetry group $G$. By definition, a moment map is an $L_\infty$-morphism between the Chevalley-Eilenberg complex and the Lie $n$-algebra of observables. Recall that in Section 4 we had defined the $L_\infty$-algebra $\widehat L(M,\omega)$, where $\widehat L$ consisted entirely of Hamiltonian forms: $\widehat L=\oplus_{k=0}^{n-1}\Omega^{n-1-k}_{\text{Ham}}(M)$.  

\begin{proposition}
A weak homotopy moment map is an $L_\infty$-morphism from $(\Rho_\g,\partial,[\cdot,\cdot])$ to $(\widehat L, \{l_k\})$.
\end{proposition}

\begin{proof}
Equation (\ref{wmm kernel}) shows that a homotopy moment map sends each element of the Lie kernel to a Hamiltonian form. Hence the claim follows from Proposition \ref{Schouten is closed} and Theorem \ref{L infinity subalgebra}.
\end{proof}

Next we study how a weak homotopy moment  map interacts with the generalized Poisson bracket on the space of Hamiltonian forms. In particular, to make a connection with Proposition \ref{difference} from symplectic geometry, we compare the difference of $f_{k+l-1}([p,q])$ and $\{f_k(p),f_l(q)\}$.

Let $G$ be a Lie group acting on a multi-Hamiltonian system $(M,\omega,H)$. Let $(f)$ be a weak homotopy moment  map. By equation (\ref{wmm kernel}) we see that under this restriction the image of the moment map is contained in the $L_\infty$-algebra  $\widehat L(M,\omega)$ of Hamiltonian forms. Moreover, we obtain that every element in the image is a conserved quantity. This was one of the main points of \cite{cq}. Indeed, in \cite{cq} Propositions 2.12 and 2.21 say:

\begin{proposition}
\label{H action}
If the group locally or globally  preserves $H$, then $f_k(p)$ is a local conserved quantity for all $p\in\Rho_{\g,k}$. If the group strictly preserves $H$ then $f_k(p)$ is a globally conserved quantity for all $p\in\Rho_{\g,k}$.
\end{proposition}

Thus, by restricting a homotopy moment  map to the Lie kernel, we see that, under assumptions on the group action, every element is a conserved quantity, analogous to the setup in symplectic geometry.

As a consequence of Theorem \ref{Noether theorem 1} we see that the moment map also gives a family of continuous symmetries.

\begin{proposition}\label{H action 2}
If the group locally or globally  preserves $H$, then $V_p$ is a local continuous symmetry for all $p\in\Rho_{\g,k}$. If the group strictly preserves $H$ then $V_p$ is a global continuous symmetry for all $p\in\Rho_{\g,k}$.
\end{proposition}
\begin{example}\bf{(Motion in a conservative system under translation)}\rm

Recall that in Example \ref{translation} we considered the translation action of $\R^3$ on $(M,\omega,H)$ where $M=T^\ast\R^3=\R^6$, $\omega=\text{vol}$ and \[H=\frac{1}{2}\left((p_1q^2dq^3-p_1q^3dq^2)+(p_2q^1dq^3-p_2q^3dq^1)+(p_3q^1dq^2-p_3q^2dq^1)\right)dp_1dp_2dp_3,\] where $q^1$, $q^2$, $q^3$ are the standard coordinates on $\R^3$ and $q^1$, $q^2$, $q^3$, $p_1$, $p_2$, $p_3$ are the induced coordinates on $T^\ast\R^3$. It is easy to check that each of $\L_{\frac{\pd}{\pd q^i}}H$ are exact for $i=1,2,3$. That is,  the group action globally preserves $H$. Hence, by Proposition \ref{H action} each of the differential forms, computed in Example \ref{translation}, \[f_1(e_1)=\frac{1}{2}(q^2dq^3-q^3dq^2)dp_1dp_2dp_3, \] \[ f_1(e_2)=\frac{1}{2}(q^1dq^3-q^3dq^1)dp_1dp_2dp_3, \] \[f_1(e_3)=\frac{1}{2}(q^1dq^2-q^2dq^1)dp_1dp_2dp_3,\]
\[f_2(e_1\wedge e_2)=q^3dp_1dp_2dp_3, \ \  \ \ \ \ f_2(e_1\wedge e_3)=q^2dp_1dp_2dp_3, \ \ \ \ \ \ f_2(e_2\wedge e_3)=q^1dp_1dp_2dp_3,\]and
\[f_3(e_1\wedge e_2\wedge e_3)=\frac{1}{3}\left(p_1dp_2dp_3+p_2dp_3dp_1+p_3dp_1dp_2\right).\]
are all globally conserved.  Thus, by Example \ref{translation}, the Lie derivative of these differential forms by the geodesic spray are all exact.

Moreover, by Proposition \ref{H action 2}, each of $\frac{\pd}{\pd q^1},\frac{\pd}{\pd q^2},\frac{\pd}{\pd q^3},\frac{\pd}{\pd q^1}\wedge\frac{\pd}{\pd q^2},\frac{\pd}{\pd q^1}\wedge\frac{\pd}{\pd q^3},\frac{\pd}{\pd q^2}\wedge\frac{\pd}{\pd q^3}$ and $\frac{\pd}{\pd q^1}\wedge\frac{\pd}{\pd q^2}\wedge\frac{\pd}{\pd q^3}$ are global continuous symmetries in this multi-Hamiltonian system. 

\end{example}

\begin{proposition}
\label{Schouten under moment}Let $p\in\Rho_{\g,k}$ and $q\in\Rho_{\g,l}$ be arbitrary. Set $\alpha=f_{n-k}(p)$ and $\beta=f_{n-l}(q)$. Then we have that $-\zeta(k)\zeta(l)[V_p,V_q]$ is a Hamiltonian multivector field for $\{\alpha,\beta\}$.
\end{proposition}

\begin{proof}
By definition of the moment map (equation (\ref{wmm kernel})) we have that $X_\alpha-\zeta(k)V_p$ and $X_\beta-\zeta(l)V_q$ are in the kernel of $\omega$. Hence, by Proposition \ref{well defined} we have that \[[X_\alpha-\zeta(k)V_p,X_\beta-\zeta(l)V_q]\hk\omega=0.\]Proposition \ref{well defined} also shows that \[[X_\alpha-\zeta(k)V_p,X_\beta-\zeta(l)V_q]\hk\omega=([X_\alpha,X_\beta]+\zeta(k)\zeta(l)[V_p,V_q])\hk\omega.\]Thus \[[X_\alpha,X_\beta]\hk\omega=-\zeta(k)\zeta(l)[V_p,V_q]\hk\omega.\]

The claim now follows from Proposition \ref{Poisson is Schouten}.
\end{proof}

Our generalization of Proposition \ref{difference} to multisymplectic geometry is:
\begin{proposition}
\label{difference is closed}
For $p\in\Rho_{\g,k}$ and $q\in\Rho_{\g,l}$ we have that \[\{f_k(p),f_l(q)\}-(-1)^{k+l+kl}f_{k+l-1}([p,q])\] is a closed $(n+1-k-l)$-form.
\end{proposition}
\begin{proof}
By definition of a homotopy moment map (equation (\ref{fhmm})) we have that 

\begin{align*}
d(f_{k+l-1}([p,q]))&=-f_{k+l-2}(\partial[p,q])-\zeta(k+l-1)V_{[p,q]}\hk\omega\\
&=-\zeta(k+l-1)V_{[p,q]}\hk\omega&\text{by Proposition \ref{Schouten is closed}}\\
&=\zeta(k+l-1)[V_p,V_q]\hk\omega&\text{by Proposition \ref{infinitesimal generator of Schouten}}\\
&=\zeta(k+l-1)\zeta(k)\zeta(l)[X_{f_k(p)},X_{f_l(p)}]\hk\omega&\text{by Proposition \ref{Schouten under moment}}\\
&=-(-1)^{k+l+kl}X_{\{f_k(p),f_l(q)\}}\hk\omega&\text{by Proposition \ref{Poisson is Schouten} and Remark \ref{ugly signs}}\\
&=(-1)^{k+l+kl}d(\{f_k(p),f_l(q)\})&\text{by definition.}
\end{align*}
\end{proof}
\del{
\begin{remark}
In the symplectic case we have that $\mu([\xi,\eta])-\{\mu(\xi),\mu(\eta)\}$ is closed (i.e. constant) and if $\mu$ is equivariant then this constant is zero. In the multiysymplectic set up, the above proposition shows that $f_{k+l-1}([p,q])-\{f_k(p),f_l(q)\}$ is a closed form. It can be shown that if the moment map $(f)$, restricted to the Lie kernel, is equivariant, then this difference is exact (this is the content of ongoing research in \cite{future}). We thus have obtained the following generalization from symplectic geometry: A homotopy moment map, restricted to the Lie kernel, is equivariant if and only if $\{f_k(p),f_l(q)\}=f_{k+l-1}([p,q])$ in deRham cohomology.
\end{remark}
}
From this proposition we see that a moment map does not necessarily preserve brackets; however, we now show that once we pass to certain cohomology groups then it will. Moreover, the moment map will give an isomorphism of graded Lie algebras, generalizing Proposition \ref{la isomorphism}. Recall that we had defined  $\widetilde\X^k_{\text{Ham}}(M)$ to be the quotient of $\X^k_{\text{Ham}}(M)$ by the kernel of $\omega$ restricted to $\Lambda^k(TM)$. We set $\widetilde\X_{\text{Ham}}(M)=\oplus\widetilde\X^k_{\text{Ham}}(M)$. 

\begin{proposition}
The Schouten bracket on $\X_{\mathrm{Ham}}(M)$ descends to a well defined bracket on $\widetilde \X_{\mathrm{Ham}}(M)$.

\end{proposition}

\begin{proof} This follows directly from Proposition \ref{well defined}.
\end{proof}
Similarily, we let $\widetilde\Omega^{n-k}_{\text{Ham}}(M)$ denote the quotient of $\Omega^{n-k}_{\text{Ham}}(M)$ by the closed forms of degree $n-k$ and set $\widetilde\Omega_{\text{Ham}}(M)=\oplus_{k=1}^n\widetilde\Omega^{n-k}_{\text{Ham}}(M)$. Recall that Proposition \ref{graded Lie algebra of forms} showed that $(\widetilde\Omega_{\text{Ham}}(M),\{\cdot,\cdot\})$ was a well defined graded Lie algebra. 
\begin{theorem}
\label{la iso}
The map $\alpha\mapsto X_\alpha$ is a graded Lie algebra isomorphism from $(\widetilde\Omega_{\mathrm{Ham}}(M),\{\cdot,\cdot\})$ to $(\widetilde\X_{\mathrm{Ham}}(M),[\cdot,\cdot])$.
\end{theorem}

\begin{proof}
The map is well defined since the Hamiltonian multivector field of a closed form is the zero vector field.  The map is clearly surjective. It is injective since if $X_\alpha=X_\beta$ then $d\alpha=d\beta$. Lastly, by Lemma \ref{Poisson is Schouten}, we have that $X_{\{\alpha,\beta\}}=[X_\alpha,X_\beta]$ in the quotient space.
\end{proof}
We have now obtained a generalization of Proposition \ref{graded Lie iso} from symplectic geometry.
\begin{corollary} The map $\alpha\mapsto X_\alpha$ is a graded Lie algebra isomorphism from $(\widetilde{\mathcal{C}}(X_H),\{\cdot,\cdot\})$ to $(\widetilde{\mathcal{S}}(H),[\cdot,\cdot])$ and from $(\widetilde{\mathcal{C}}_{loc}(X_H),\{\cdot,\cdot\})$ to $(\widetilde{\mathcal{S}}_{loc}(H),[\cdot,\cdot])$.

\end{corollary}
\begin{proof}
We know from Proposition \ref{Poisson is strictly conserved} that each of the spaces of conserved quantities are closed under the Poisson bracket. Similarly, by Proposition \ref{symmetry super algebra}, the spaces of continuous symmetries are all closed under the Schouten bracket. The claim now follows from Theorem \ref{la iso}.
\end{proof}

We now give a generalization of Proposition \ref{la isomorphism} to multisymplectic geometry: We let $C_k$ denote the image of the Lie kernel under the moment map. That is, let $C_k=f_{n-k}(\Rho_{\g,n-k})$. Let $\widetilde C_k$ denote the quotient of $C_k$ by closed forms and set $\widetilde C=\oplus_{k=1}^n \widetilde C_k$.   Recall that we had defined $S_k$ to be the set $\{V_p; p\in\Rho_{\g,k}\}$. Let $\widetilde S_k$ denote the quotient of $S_k$ by elements in the Lie kernel and set $\widetilde S=\oplus_{k=1}^n\widetilde S_k$.  Our generalization of Proposition \ref{la isomorphism} is given by the following corollaries.

\begin{corollary}
For an $H$-preserving group action, a momentum map induces an $L_\infty$-algebra morphism from $\widetilde S$ to $\widetilde C \cap \widehat L$ given by $V_p\mapsto f_k(p)$.
\end{corollary}

\begin{proof}
We see by Proposition \ref{difference is closed} that the Poisson bracket preserves $\widetilde C$. The claim follows since by definition a homotopy moment map is an $L_\infty$-morphism.
\end{proof}

\begin{corollary}
For an $H$-preserving group action, an equivariant homotopy moment map induces an isomorphism of graded Lie algebras between $(\widetilde S,[\cdot,\cdot])$ and $(\widetilde C, \{\cdot,\cdot\})$. Explicitly, the map is given by $[V_p]\mapsto [f_k(p)]$. 
\end{corollary}

\begin{proof}
The Lie algebra isomorphism given in Theorem \ref{la iso} is precisely the moment map. Indeed, if $\alpha=f_k(p)$ for $p\in\Rho_{\g,k}$, then $X_{f_k(p)}=V_p$ in $\widetilde\X_{\text{Ham}}(M)$, since both are Hamiltonian vector fields for $\alpha$. Proposition \ref{difference is closed} now shows that the moment map preserves the Lie brackets on these quotient spaces.

\end{proof}

The morphism properties of a weak moment map are also related to its equivariance. Recall that for a symplectic action of a connected Lie group $G$ acting on a symplectic manifold $(M,\omega)$ and a moment map, $f:\g\to C^\infty(M)$, then by Definition \ref{inf equiv moment}, $f$ is equivariant if and only if $f$ is a Lie algebra morphism from $(\g,[\cdot,\cdot])$ to $(C^\infty(M),\{\cdot,\cdot\})$. That is, if and only if \[f([\xi,\eta])=\{f(\xi),f(\eta)\}.\] Taking $d$ of both sides of this equation yields:
\begin{proposition}\label{morph 1}
A moment map $f$ induces a morphism from $\g$ onto the quotient of $C^\infty(M)$ by constant functions. That is, a moment map induces a Lie algebra morphism from $(\g,[\cdot,\cdot])$ to $(C^\infty(M)/\text{closed},\{\cdot,\cdot\})$, regardless of equivariance. Moreover, the moment map $f$ is equivariant if and only if $f$ is a morphism from $(\g,[\cdot,\cdot])$ to $(C^\infty(M),\{\cdot,\cdot\})$.

\end{proposition}

We now restate Proposition \ref{morph 1} in an equivalent way, but which will allow for a direct generalization to multisymplectic geometry: Notice that $\g$ is a $\g$-module under the Lie bracket action and $C^\infty(M)$ is $\g$-module under the action $\xi\cdot\g=L_{V_{\xi}}g$, where $\xi\in\g$ and $g\in C^\infty(M)$. Proposition \ref{morph 1} is equivalent to:

\begin{proposition}\label{Morph 1}
A moment map  $f$ always induces a $\g$-module morphism from $\g$ to $C^\infty(M)/\text{closed}$. Moreover, the moment map $f$ is equivariant if and only if it is a $\g$-module morphism from $\g$ to $C^\infty(M)$.
\end{proposition}

Now let a connected Lie group $G$ act multisymplectically on an $n$-plectic manifold $(M,\omega)$. 

\begin{proposition}
For any $1\leq k\leq n$, we have that $\Rho_{\g,k}$ is a $\g$-module under the action $\xi\cdot p = [p,\xi]$, where $p\in\Rho_{\g,k}$, $\xi\in\g$, and $[\cdot,\cdot]$ is the Schouten bracket. 
\end{proposition}

\begin{proof}
This follows since Lemma \ref{lemma formula} shows that $[p,\xi]$ is in the Lie kernel.
\end{proof}

\begin{proposition}
For any $1\leq k \leq n$, we have that $\Omega^{n-k}_{\mathrm{Ham}}(M)$ is a $\g$-module under the action $\xi\cdot\alpha=\L_{V_\xi}\alpha$, where $\alpha\in\Omega^{n-k}_{\mathrm{Ham}}(M)$ and $\xi\in\g$.
\end{proposition}

\begin{proof}
Suppose that $\alpha\in\Omega^{n-k}_{\mathrm{Ham}}(M)$ is a Hamiltonian $(n-k)$-form. Then $d\alpha=-X_\alpha\hk\omega$ for some $X_\alpha\in\Gamma(\Lambda^k(TM))$. Then, for $\xi\in\g$,

\begin{align*}
d\L_{V_\xi}\alpha&=-\L_{V_\xi}(X_\alpha\hk\omega)\\
&=-\L_{V_\xi}(X_\alpha\hk\omega)+X_\alpha\hk\L_{V_\xi}\omega &\text{since $\L_{V_\xi}\omega=0$}\\
&=[V_\xi,X_\alpha]\hk\omega &\text{by the product rule}
\end{align*}

Hence $\L_{V_\xi}\alpha$ is in $\Omega^{n-k}_{\mathrm{Ham}}(M)$.

\end{proof}

Our generalization of Proposition \ref{Morph 1} to multisymplectic geometry is:

\begin{theorem}\label{morphism 1}
For any $1\leq k \leq n$, the $k$-th component of a moment map $f_k$ is a $\g$-module morphism from $\Rho_{\g,k}$ to $\Omega^{n-k}_{\mathrm{Ham}}(M)/\mathrm{closed}$. Moreover, a weak $k$-moment map $f_k$ is equivariant if and only if it is a $\g$-module morphism from $\Rho_{\g,k}$ to $\Omega^{n-k}_{\mathrm{Ham}}(M)$.
\end{theorem}

\begin{proof}
Suppose that $(f)$ is a weak moment map. Then, by definition
\begin{align*}
df_k([\xi,p])&=-\zeta(k)V_{[\xi,p]}\hk\omega\\
&=-\zeta(k)[V_\xi,V_p]\hk\omega\\
&=-\zeta(k)\L_{V_\xi}(V_p\hk\omega)\\
&=\zeta(k)\zeta(k)d\L_{V_\xi}f_k(p)\\
&=d\L_{V_\xi}f_k(p).
\end{align*}
 This proves the first statement of the theorem. Now suppose $f_k$ is equivariant. It follows that $\Sigma_k=0$. Thus, by Proposition \ref{comp of Sigma} we have $f_k([\xi,p])=\L_{V_\xi}f_k(p)$. Conversely, if $f_k$ is a $\g$-module morphism, that $f_k([\xi,p])=\L_{V_\xi}f_k(p)$ for every $\xi\in\g$ and $p\in\Rho_{\g,k}$. That is, $\Sigma_k=0$. 
\end{proof}

\newpage
\section{Applications}
We first apply the generalized Poisson bracket to extend the theory of the classical momentum and position functions on the phase space of a manifold to the multisymplectic phase space.
\subsection{Classical Multisymplectic Momentum and Position Forms}
Recall the following notions from Hamiltonian mechanics:

Let $N$ be a manifold and $(T^\ast N,\omega=-d\theta)$ the canonical phase space. Given a group action on $N$ we can extend this to a group action on $T^\ast N$ that preserves both tautological forms $\theta$ and $\omega$. It is easy to check that a moment map for this action is given by $f:\g\to C^\infty(N) \ , \ \xi\mapsto V_\xi\hk\theta$. From this moment map we can introduce the classical momentum and position functions, as discussed in Propositions 4.2.12 and 5.4.4 of \cite{Marsden}. Given $X\in\Gamma(TN)$, its classical momentum function is $P(X)\in C^\infty(T^\ast N)$ defined by $P(X)(\alpha_q):= \alpha_q(X_q)$. Corollary 4.2.11 of \cite{Marsden} then shows that \begin{equation}\label{momentum form equal}P(V_\xi)=f(\xi).\end{equation} Next, given $h\in C^\infty(N)$ define $\widetilde h\in C^\infty(T^\ast N)$ by $\widetilde h= h\circ\pi$. The function $\widetilde h$ is referred to as the corresponding position function. The following Poisson bracket relations between the momentum and position functions are then obtained in Proposition 4.2.12 of \cite{Marsden}: 
\begin{equation}\label{classical 1}\{P(X),P(Y)\}=P([X,Y])\end{equation}
\begin{equation}\label{classical 2}\{\widetilde h,\widetilde g\}=0\end{equation}and
\begin{equation}\label{classical 3}\{\widetilde h, P(X)\}=\widetilde{X(h)}.\end{equation} 
These bracket relations are the starting point for obtaining a quantum system from a classical system. 

\begin{remark}In \cite{Marsden} their first bracket relation actually reads $\{P(X),P(Y)\}=-P([X,Y])$. This is because their defining equation for a Hamiltonian vector field is $dh=X_h\hk\omega$, as compared to our $dh=-X_h\hk\omega$. 
\end{remark}

The goal of this subsection is to show how these concepts generalize to the multisymplectic phase space. The results we obtain are all new. As in Example \ref{Multisymplectic Phase Space}, let $N$ be a manifold and $(M,\omega)$ the multisymplectic phase space. That is, $M=\Lambda^k(T^\ast N)$ and $\omega=-d\theta$ is the canonical $(k+1)$-form on $M$. Let $\pi:M\to N$ denote the projection map. In Example \ref{multisymplectic phase space} we showed that \[f_l:\Rho_{\g,l}\to\Omega^{k-l}_{\text{Ham}}(M) \ \ \ \ \ \ \ \ \ \ p\mapsto -\zeta(l+1)V_p\hk\theta\] was a weak homotopy moment map for the action on $M$ induced from the action on $N$.

\begin{definition}
Given decomposable $X=X_1\wedge\cdots\wedge X_l$ in $\Gamma(\Lambda^l(TN))$ we define its momentum form $P(X)\in \Omega^{k-l}(M)$ by \[P(X)(\mu_x)(Z_1,\dots, Z_{k-l}):=-\zeta(l+1) \mu_x(X_1,\cdots,X_l,\pi_\ast Z_1,\cdots, \pi_\ast Z_{k-l}),\]where $\mu_x$ is in $M$, and then extend by linearity to non-decomposables. Moreover, given $\alpha\in\Omega^{k-l}(N)$, we define the corresponding position form to be $\pi^\ast\alpha$, a $(k-l)$-form on $M$. Notice that in symplectic case ($k=1$) these definitions coincide with the classical momentum and position functions.
\end{definition}

\begin{definition}\label{complete lift}
Given a vector field $Y\in\Gamma(TN)$ with flow $\theta_t$, the complete lift of $Y$ is the vector field $Y^\sharp\in\Gamma(TM)$ whose flow is $(\theta_t^\ast)^{-1}$. For a decomposable multivector field $Y=Y_1\wedge\cdots\wedge Y_l\in\Gamma(\Lambda^l(TN))$ we define its complete lift $Y^\sharp$ to be $Y_1^\sharp\wedge\cdots \wedge Y_l^\sharp$ and then extend by linearity.
\end{definition}

For $\xi\in\g$, let $V_\xi$ denote its infinitesimal generator on $N$ and let $V^\sharp_\xi$ denote its infinitesimal generator on $M$. Similarly, for $p=\xi_1\wedge\cdots\wedge \xi_l$ in $\Lambda^l\g$ let $V_p$ denote $V_{\xi_1}\wedge\cdots\wedge V_{\xi_l}$ and $V^\sharp_p$ denote $V^\sharp_{\xi_1}\wedge\cdots\wedge V^\sharp_{\xi_l}$. Notice that by definition, we are not abusing notation by letting $V_p^\sharp$ denote both the complete lift of $V_p$ and the infinitesimal generator of $p$ under the induced action on $M$.

Lastly, note that by the equivariance of $\pi:M\to N$, we have \[\pi_\ast(V^\sharp_p)=V_p\circ\pi.\]

We now examine the bracket relations between our momentum and position forms. We first rewrite the momentum form in a different way:

\begin{proposition}
\label{momentum form}
For $Y\in\Gamma(\Lambda^l(TN))$ we have that $P(Y)=-\zeta(l+1)Y^\sharp\hk\theta$.
\end{proposition}

\begin{proof}
Let $Y=Y_1\wedge\cdots \wedge Y_l$ be an arbitrary decomposable element of $\Gamma(\Lambda^l(TN))$. Let $Z_1,\cdots, Z_{k-l}$ be arbitrary vector fields on $M$. Fix $\mu_x\in M$. Then

\begin{align*}
(Y^\sharp\hk\theta)_{\mu_x}(Z_1,\cdots, Z_{k-l})&=\theta_{\mu_x}(Y_1^\sharp,\cdots, Y_l^\sharp, Z_1,\cdots, Z_{k-l})\\
&=\mu_x(\pi_\ast Y_1^\sharp,\cdots, \pi_\ast Y_l^\sharp, \pi_\ast Z_1,\cdots, \pi_\ast Z_{k-l})\\
&=\mu_x(Y_1,\cdots,Y_l,\pi_\ast Z_1,\cdots, \pi_\ast Z_{k-l})\\
&=-\zeta(l+1)P(Y)_{\mu_x}(Z_1,\cdots, Z_{k-l}).
\end{align*}
\end{proof}
As a corollary to the above proposition, we obtain a generalization of (\ref{momentum form equal}) to multisymplectic geometry:

\begin{corollary}
For $p=\xi_1\wedge\cdots\wedge\xi_l$ in $\Rho_{\g,l}$ we have \[P(V_p)=f_l(p).\]
\end{corollary}

\begin{proof}
\del{
Let $\mu_x\in\Lambda^l(T_x^\ast N)$ be arbitrary.  Then, for arbitrary $Z_1,\cdots, Z_{k-l}\in\Gamma(TM)$, we have 

\begin{align*}
f_l(p)(\mu_x)(Z_1,\dots,Z_{k-l})&=-\zeta(l+1)(V^\sharp_p\hk\theta)_{\mu_x}(Z_1,\dots, Z_{k-l})\\
&=-\zeta(l+1)\mu_x(\pi_\ast(V^\sharp_p),\pi_\ast Z_1,\dots,\pi_\ast Z_{k-l})&\text{by definition of $\theta$}\\
&=-\zeta(l+1)\mu_x(V_{\xi_1},\dots,V_{\xi_l},\pi_\ast Z_1,\dots,\pi_\ast Z_{k-l})\\
&=P(V_p)(\mu_x)(Z_1,\dots,Z_{k-l})
\end{align*}
}This follows immediately from Proposition \ref{momentum form} since $V_p^\sharp$ is the infinitesimal generator of $p$ on $M$.
\end{proof}

\del{
\begin{remark}
If we have a group acting on $N$, then the induced action on $M$ is given precisely by the pull back. Hence, for an element $\xi\in\g$ we have that the infinitesimal generator on $M$ is exactly the complete lift of the infinitesimal generator of $\xi$ on $N$. This justifies the use of the $\sharp$ notation in the above proposition.
\end{remark}
}

In the symplectic case, given $Y\in\Gamma(TN)$ the complete lift $Y^\sharp\in\Gamma(T(T^\ast N))$ preserves the tautological forms $\theta$ and $\omega$. Hence $d(Y^\sharp\hk\theta)=Y^\sharp\hk\omega$, showing that each momentum function is Hamiltonian with Hamiltonian vector field the complete lift of the base vector field. 

In the multisymplectic case, it is no longer true that $\L_{Y^\sharp}\theta=0$ for a multivector field $Y$. Instead, we need to restrict our attention to multivector fields in the Lie kernel, which we defined in Definition \ref{Lie kernel}. We quickly recall this definition and some terminology and notation introduced in \cite{MS}. 

Any degree $l$-multivector field is a sum of multivectors of the form $Y=Y_1\wedge\cdots\wedge Y_l$. We consider the differential graded Lie algebra $(\Gamma(\Lambda^\bullet(TN)),\pd)$ where $\pd_l:\Gamma(\Lambda^l(TN))\to\Gamma(\Lambda^{l-1}(TN))$ is given by \[\pd_l(Y_1\wedge\cdots\wedge Y_l)=\sum_{1\leq i\leq j\leq l}[Y_i,Y_j]\wedge Y_1\wedge\cdots\wedge\widehat Y_i\wedge\cdots\wedge\widehat Y_j\wedge\cdots\wedge Y_l.\]As in \cite{MS}, for a differential form $\tau$, let \[(\hk \ \L)_Y\tau=\sum_{i=1}^lY_1\wedge\cdots\wedge \widehat Y_i\wedge\cdots\wedge Y_l\hk\L_{Y_i}\tau.\]

A more general version of Lemma \ref{extended Cartan} is given by Lemma 3.4 of \cite{MS}:

\begin{lemma}
\label{extended Cartan lemma}
For a differential form $\tau$ and $Y=Y_1\wedge\cdots\wedge Y_l\in\Gamma(\Lambda^l(TN))$ we have that \[Y\hk d\tau-(-1)^ld(Y\hk\tau)=(\hk \ \L)_Y\tau-\pd_l(Y)\hk\tau.\]
\end{lemma}

\begin{definition}
As in Definition \ref{Lie kernel}, we call $\Rho_{l}=\text{ker }\pd_l$ the $l$-th Lie kernel. 
\end{definition}

\begin{proposition}
\label{momentum form 2}
For an $l$-multivector field in the Lie kernel, $Y\in\Rho_{l}$, we have that $P(Y)$ is in $\Omega^{k-l}_{\mathrm{Ham}}(M)$. More precisely, $\zeta(l)Y^\sharp$ is a Hamiltonian multivector field for $P(Y)$. 
\end{proposition}
\begin{proof}
Abusing notation, let $\partial_l$ denote the differential on both $\Gamma(\Lambda^\bullet(TN))$ and $\Gamma(\Lambda^\bullet(TM))$. By definition, we have $\partial_l(Y)=0$. It follows that $\partial_l(Y^\sharp)=0$. Now, since the action on $M$ preserves $\theta$, we have that $(\hk \ \L)_{Y^\sharp}\theta=0$. Thus, by Proposition \ref{momentum form} and Lemma  \ref{extended Cartan lemma}, we have that 
\begin{align*}
d(P(Y))&=-\zeta(l+1)d(Y^\sharp\hk\theta)\\
&=-\zeta(l+1)(-1)^l(Y^\sharp\hk d\theta)\\
&=\zeta(l+1)(-1)^{l}(Y^\sharp\hk\omega)\\
&=-\zeta(l)Y^\sharp\hk\omega,
\end{align*} where in the last equality we used Remark \ref{ugly signs}.
\end{proof}

\begin{remark}
In the setup of classical Hamiltonian mechanics, the phase space of $N$ is just $T^\ast N$, and so $k=l=1$. Since $\Rho_{1}=\Gamma(TN)$ we see that we are obtaining a generalization from Hamiltonian mechanics.
\end{remark}
We now arrive at our generalization of equation (\ref{classical 1}):

\begin{proposition}
For $Y_1\in\Rho_s$ and $Y_2\in\Rho_t$ we have that
\[\{P(Y_1),P(Y_2)\}=-(-1)^{ts+s+t}P([Y_1,Y_2]) -\zeta(s+1)\zeta(t+1)d(Y_1^\sharp\hk Y_2^\sharp\hk\theta).\]

\end{proposition}

\begin{proof}
Using Proposition \ref{momentum form 2}, Remark \ref{ugly signs} and the definition of the bracket, we have
\begin{equation}\label{ugly Poisson}
\begin{aligned}\{P(Y_1),P(Y_2)\}&=(-1)^{t+1}\zeta(t)\zeta(s)Y_2^\sharp\hk Y_1^\sharp\hk\omega\\&=(-1)^{ts+t}\zeta(s+t)Y_2^\sharp\hk Y_1^\sharp\hk\omega.\end{aligned}\end{equation}On the other hand, by Proposition \ref{momentum form} and Remark \ref{ugly signs} we have \begin{align*}P([Y_1,Y_2])&=-\zeta(s+t)[Y_1^\sharp,Y_2^\sharp]\hk\theta.\end{align*} By Proposition \ref{momentum form 2} and Remark \ref{ugly signs} we have that \[d(Y_1^\sharp\hk\theta)=(-1)^{s+1}Y_1^\sharp\hk\omega\] and \[d(Y_2^\sharp\hk\theta)=(-1)^{t+1}Y_2^\sharp\hk\omega.\] By equation (\ref{interior equation}) we have that
\[[Y_1^\sharp,Y_2^\sharp]\hk\theta=-Y_2^\sharp\hk(d(Y_1^\sharp\hk\theta))+(-1)^td(Y_2^\sharp\hk Y_1^\sharp\hk\theta)-(-1)^{st+s}Y_1^\sharp\hk Y_2^\sharp\hk\omega-(-1)^{st+s+t}Y_1^\sharp\hk(d(Y_2^\sharp\hk\theta)),\] and using the two equations above, this is equal to
\[(-1)^s(Y_2^\sharp\hk Y_1^\sharp\hk\omega)+(-1)^td(Y_2^\sharp\hk Y_1^\sharp\hk\theta)-(-1)^{st+s}Y_1^\sharp\hk Y_2^\sharp\hk\omega-(-1)^{st+s+t}(-1)^{t+1}Y_1^\sharp\hk Y_2^\sharp\hk\omega.\]
\del{&=(-1)^s(Y_2^\sharp\hk Y_1^\sharp\hk\omega)+(-1)^td(Y_2^\sharp\hk Y_1^\sharp\hk\theta)-(-1)^{st+s}Y_1^\sharp\hk Y_2^\sharp\hk\omega+(-1)^{st+s}Y_1^\sharp\hk Y_2^\sharp\hk\omega\\}Simplifying this equation we obtain that 
\[[Y_1^\sharp,Y_2^\sharp]\hk\theta=(-1)^s(Y_2^\sharp\hk Y_1^\sharp\hk\omega)+(-1)^td(Y_2^\sharp\hk Y_1^\sharp\hk\theta).\]

Thus, \begin{equation}\label{ugly Poisson 2}P([Y_1^\sharp,Y_2^\sharp])=-\zeta(s+t)(-1)^s(Y_2^\sharp\hk Y_1^\sharp\hk\omega)-\zeta(s+t)(-1)^td(Y_2^\sharp\hk Y_1^\sharp\hk\theta).\end{equation}
Equating equations (\ref{ugly Poisson}) and (\ref{ugly Poisson 2}) and using Remark \ref{ugly signs} gives the result.
\end{proof}

To generalize (\ref{classical 2}) and (\ref{classical 3}) to the multisymplectic phase space, we need the following lemma:
\begin{lemma}\label{push zero}
Let $\alpha$ be an arbitrary $(k-l)$-form on $N$ and let $\pi^\ast\alpha$ be the corresponding classical position form in $\Omega^{k-l}(M)$. Then $\pi^\ast\alpha$ is Hamiltonian and $\pi_\ast(X_{\pi^\ast\alpha})=0$.
\end{lemma}

\begin{proof}
Let $q^1,\cdots,q^n$ denote coordinates on $N$, and let $\{p_{i_1\cdots i_k}; 1\leq i_1<\cdots<i_k\leq n \}$ denote the induced fibre coordinates on $M$. In these coordinates we have that \[\theta=\sum_{1\leq i_1<\cdots<i_k\leq n}p_{i_1\cdots i_k}dq^{i_1}\wedge\cdots\wedge dq^{i_k}\] so that \[\omega=-d\theta=\sum_{1\leq i_1<\cdots<i_k\leq n}-dp_{i_1\cdots i_k}\wedge dq^{i_1}\wedge\cdots\wedge dq^{i_k}.\] An arbitrary $(k-l)$-form $\alpha$ on $N$ is given by \[\alpha=\alpha_{i_1\cdots i_{k-l}}dq^{i_1}\wedge\cdots\wedge dq^{i_{k-l}}.\]Abusing notation, it follows that \[\pi^\ast\alpha=\alpha_{i_1\cdots i_{k-l}}dq^{i_1}\wedge\cdots\wedge dq^{i_{k-l}}.\]Thus,\[d\pi^\ast\alpha=\frac{\pd\alpha_{i_1\cdots i_{k-l}}}{\pd q^j}dq^j\wedge dq^{i_1}\wedge\cdots\wedge dq^{i_{k-l}}.\] An arbitrary $l$-vector field on $M$ is of the form \[X=a^{i_1\cdots i_l}\frac{\pd}{\pd q^{i_1}}\wedge\cdots\wedge\frac{\pd}{\pd q^{i_{l}}}+a^{i_1\cdots i_{l-1}}_{J_1}\frac{\pd}{\pd q^{i_1}}\wedge\cdots\wedge\frac{\pd}{\pd q^{i_{l-1}}}\wedge\frac{\pd}{\pd p^{J_1}}+\cdots+a_{J_1\cdots J_{l}}\frac{\pd}{\pd p^{J_1}}\wedge\cdots\wedge\frac{\pd}{\pd p^{J_l}}.\] Now, the multivector field $X_{\pi^\ast\alpha}$ we are looking for satisfies $X_{\pi^\ast\alpha}\hk\omega= d\pi^\ast\alpha$. An exercise in combinatorics shows that there always exists an $l$-vector field $X$ satisfying $X\hk\omega=d\pi^\ast\alpha$, proving that $\pi^\ast\alpha$ is Hamiltonian. Note we can see directly from the equality $X\hk\omega=d\pi^\ast\alpha$ that necessarily  \[a^{i_1\cdots i_l}=0.\] Thus $\pi_\ast(X_{\pi^\ast\alpha})=0$ as desired.

\del{

  \[a^{i_1\cdots i_{l-2}}_{J_1J_2}= \cdots = a_{J_1J_2\cdots J_l}=0.\] 
By taking the interior product of this vector field with $\omega$, we see that
 We can see that for this to happen, necessarily \[a^{i_1\cdots i_l}=0.\] Moreover, one can show through an exercise in combinatorics that system of equations for modelling the equation of finding $a_{J_1}^{i_1\cdots i_{l-1}}$ such that $X\hk\omega=d\pi^\ast\alpha$ always exists showing that $\pi^\ast\alpha$ is always a Hamiltonian form. Moreover, since $a^{i_1\cdots i_l}=0,$ it follows that $\pi_\ast(X_{\pi^\ast\alpha})=0.$ \del{

 Hencesatisfying the equation  To simplify the notation, let $I_{l-1}$ denote a multi-index set of length $l-1$. It follows the vector field $X_{\pi^\ast\alpha}$ we are looking for is of the from $X_{\pi^\ast\alpha}=a^{I_{l-1}}_{J_k}\frac{\pd}{\pd a^{I_{l-1}}}\wedge\frac{\pd}{\pd p^J}$

 Then $X_{\pi^\ast}\alpha$ has $a^{i_1\cdots i_{l-1}}_{J_1}=\frac{\pd\alpha_{i_1\cdots i_{k-l}}}{\pd q^j}$ and all other coefficients equal to zero. In particular then, since $a^{i_1\cdots i_l}=0$, it follows that $\pi_\ast(X_{\pi^\ast\alpha})=0$.
 }
 
 }
\end{proof}

Our generalization of (\ref{classical 2}) is:
\begin{proposition}
For $\alpha\in\Omega^{k-i}(N)$ and $\beta\in\Omega^{k-j}(N)$ we have that \[\{\pi^\ast\alpha,\pi^\ast\beta\}=0.\]
\end{proposition}

\begin{proof}
Let $Z_1,\cdots, Z_{k+1-i-j}\in\Gamma(TM)$ be arbitrary. Then,

\begin{align*}
\{\pi^\ast\alpha,\pi^\ast\beta\}(Z_1,\cdots, Z_{k+1-i-j})&=(-1)^{j+1}X_{\pi^\ast\beta}\hk X_{\pi^\ast\alpha}\hk\omega(Z_1,\cdots, Z_{k+1-i-j})\\
&=(-1)^jX_{\pi^\ast\beta}\hk(\pi^\ast d\alpha)(Z_1,\cdots, Z_{k+1-i-j})\\
&=(-1)^j\pi^\ast d\alpha(X_{\pi^\ast\beta},Z_1,\cdots, Z_{k+1-i-j})\\
&=(-1)^jd\alpha(0,\pi_\ast Z_1,\cdots, \pi_\ast X_{k+1-i-j})&\text{by Lemma \ref{push zero}}\\
&=0.
\end{align*}

\end{proof}

Our generalization of (\ref{classical 3}) is:
\begin{proposition}
For $\alpha\in\Omega^{k-i}(N)$ and $Y\in\Rho_j$, we have that \[\{\pi^\ast\alpha,P(Y)\}=-\zeta(j)\pi^\ast(Y\hk d\alpha).\]
\end{proposition}

\begin{proof}
Let $Z_1,\cdots, Z_{k+1-i-j}\in\Gamma(TM)$ be arbitrary. Then,
\begin{align*}
\{\pi^\ast\alpha,P(Y)\}(Z_1,\cdots, Z_{k+1-i-j})&=(-1)^{j+1}X_{P(Y)}\hk X_{\pi^\ast\alpha}\hk\omega(Z_1,\cdots, Z_{k+1-i-j})\\
&=(-1)^{j+1}\zeta(j)Y^\sharp\hk X_{\pi^\ast\alpha}\hk\omega(Z_1,\cdots,Z_{k+1-i-j})&\text{by Lemma \ref{momentum form 2}}\\
&=(-1)^j\zeta(j+1)Y^\sharp\hk \pi^\ast d\alpha(Z_1,\cdots,Z_{k+1-i-j})\\
&=-\zeta(j)Y^\sharp\hk \pi^\ast d\alpha(Z_1,\cdots,Z_{k+1-i-j})&\text{by Remark \ref{ugly signs}}\\
&=-\zeta(j)d\alpha(\pi_\ast Y^\sharp,\pi_\ast Z_1,\cdots, \pi_\ast Z_{k+1-i-j})\\
&=-\zeta(j)d\alpha(Y,\pi_\ast Z_1,\cdots,\pi_\ast Z_{k+1-i-j})\\
&=-\zeta(j)\pi^\ast(Y\hk d\alpha)(Z_1,\cdots, Z_{k+1-i-j}).
\end{align*}
\end{proof}
\subsection{Closed $G_2$ Structures}

We first recall the standard $G_2$ structure on $\R^7$.  More details for the material in this section can be found in \cite{j}.  Let $x^1,\cdots,x^7$ denote the standard coordinates on $\R^7$ and consider the $3$-form $\varphi_0$ defined by
\[\varphi_0=dx^{123}+dx^1(dx^{45}-dx^{67})+dx^2(dx^{46}-dx^{75})-dx^3(dx^{47}-dx^{56})\] where we have omitted the wedge product signs. The stabilizer of this $3$-form is given by the Lie group $G_2$. For an arbitrary $7$-manifold we define a $G_2$ structure to be a $3$-form $\varphi$ which has around every point $p\in M$ local coordinates with $\varphi=\varphi_0$, at the point $p$. 

The $3$-form induces a unique metric $g$ and volume form, vol,  determined by the equation \[(X\hk\varphi)\wedge(Y\hk\varphi)\wedge\varphi=-6g(X,Y)\text{vol}.\]From the volume form we get the Hodge star operator and hence a $4$-form $\psi:=\ast\varphi$. We will refer to the data $(M^7,\varphi,\psi,g)$ as a manifold with $G_2$ structure. We remark that the $G_2$ form $\varphi$ is more than just non-degenerate:

\begin{proposition}
The $G_2$ form $\varphi$ is fully nondegenerate. This means that $\varphi(X,Y,\cdot)$ is non-zero whenever $X$ and $Y$ are linearly independent.
\end{proposition}

\begin{proof}
See Theorem 2.2 of \cite{MS}.
\end{proof}

\del{
We will call a manifold with $G_2$ structure torsion-free if both $\varphi$ and $\psi$ are closed. A theorem of Fernandez and Gray shows that this happens precisely when $\varphi$ is parallel with respect to the induced metric $g$. Thus we see that a torsion-free $G_2$ structure is an example of a multisymplectic manifold. 

\begin{remark}
All of the results in this section will only use the fact that $\varphi$ is closed so that, in particular, all of our results holds if the $G_2$ structure is torsion-free.
\end{remark}
}
For the rest of this section we assume that $d\varphi=0$. We now briefly recall some first order differential operators on a $G_2$ manifold, while refering the reader to section 4 of \cite{Spiro 2} for more details. Given $X\in\Gamma(TM)$ we will let $X^\flat$ denote the metric dual one form $X^\flat=X\hk g$. Conversely, given $\alpha\in\Omega^1(M)$, let $\alpha^\sharp$ denote the metric dual vector field. Recall that given $f\in C^\infty(M)$ its gradient is defined by \[\text{grad}(f)=(df)^\sharp.\] 

From the metric and the three form we can define the cross product of two vector fields.  Given $X,Y,Z\in\Gamma(TM)$ the cross product $X\times Y$ is defined by the equation \[\varphi(X,Y,Z)=g(X\times Y,Z).\] Equivalently, the cross product is defined by \[(X\times Y)=(Y\hk X\hk\varphi)^\sharp.\] In coordinates, this says that \begin{equation}\label{cross}(X\times Y)^l=X^iY^i\varphi_{ijk}g^{kl}.\end{equation}The last differential operator we will consider is the curl of a vector field. We first need to recall the following decomposition of two forms on a $G_2$ manifold. 

\begin{proposition}
The space of $2$-forms on a $G_2$ manifold has the $G_2$ irreducible decomposition \[\Omega^2(M)=\Omega^2_7(M)\oplus\Omega^2_{14}(M),\] where \[\Omega^2_7(M)=\{X\hk\varphi ; X\in\Gamma(TM)\}\] and \[\Omega^2_{14}(M)=\{\alpha\in\Omega^2(M); \psi\wedge\alpha=0\}.\] The projection maps: $\pi_7:\Omega^2(M)\to\Omega^2_7(M)$ and $\pi_{14}:\Omega^2(M)\to\Omega^2_{14}(M)$ are given by \begin{equation}\label{pi 7}\pi_7(\alpha)=\frac{\alpha-\ast(\varphi\wedge\alpha)}{3}\end{equation} and \begin{equation}\label{pi 14}\pi_{14}(\alpha)=\frac{2\alpha+\ast(\varphi\wedge\alpha)}{3}.\end{equation}
\end{proposition}

\begin{proof}
See Section 2.2 of \cite{Spiro}.
\end{proof}

We can now define the curl of a vector field. Given $X\in\Gamma(TM)$ its curl is defined by \begin{equation}\label{curl defn}(\text{curl}(X))^\flat=\ast(dX^\flat\wedge\psi).\end{equation}This is equivalent to saying that \begin{equation}\label{curl is 14}\pi_{7}(dX^\flat)=\text{curl}(X)\hk\varphi.\end{equation} In coordinates,

\begin{equation}\label{curl coord}\text{curl}(X)^l=(\nabla_aX_b)g^{ai}g^{bj}\varphi_{ijk}g^{kl},\end{equation} where $\nabla$ is the Levi-Civita connection corresponding to $g$. This is reminiscent of the fact that in $\R^3$ the curl is given by the cross product of $\nabla$ with $X$. Again, we refer the reader to Section 4.1 of \cite{Spiro 2} for more details.
\\

 We now translate our definition of Hamiltonian forms and vector fields into the language of $G_2$ geometry.  By definition, we see that a $1$-form is Hamiltonian if and only if its differential is in $\Omega^2_7(M)$. That is, \[\Omega^1_{\text{Ham}}(M)=\{\alpha\in\Omega^1(M) ; \pi_{14}(d\alpha)=0\}.\] Similarly,

\[\X_{\text{Ham}}^1(M)=\{X\in\Gamma(TM) ; X=\text{curl}(\alpha^\sharp) \text{ and } \pi_{14}(d\alpha)=0 \text{ for some } \alpha\in\Omega^1(M)\}.\]

Note that if $M$ is compact, then it follows from (\ref{pi 7}) and Hodge theory that there are no non-zero Hamiltonian $1$-forms.
\begin{proposition}
\label{Hamiltonian is curl}
If $\alpha$ is a Hamiltonian $1$-form then its corresponding Hamiltonian vector field is $\mathrm{curl}(\alpha^\sharp)$.
\end{proposition}
\begin{proof}
Since a Hamiltonian $1$-form satisfies $\pi_{14}(\alpha)=0$, this follows immediately from equation (\ref{curl is 14}). 
\end{proof}

From Proposition \ref{Hamiltonian is curl} and equation (\ref{cross}) we see that the generalized Poisson bracket is given by the cross product: \begin{equation}\label{bracket is cross} \{\alpha,\beta\}=\text{curl}(\alpha^\sharp)\times\text{curl}(\beta^\sharp),\end{equation}
for $\alpha,\beta\in\Omega^1_{\text{Ham}}(M)$.

Proposition \ref{preserve} showed that a Hamiltonian vector field preserves the $n$-plectic form. In the language of $G_2$ geometry this gives:

\begin{proposition}
Given $\alpha\in\Omega^1(M)$ with $\pi_{14}(d\alpha)=0$, the curl of $\alpha^\sharp$ preserves the $G_2$ structure. That is, \[\L_{\mathrm{curl}(\alpha^\sharp)}\varphi=0.\]
\end{proposition}

\begin{proof}
This follows immediately from Propositions \ref{Hamiltonian is curl} and \ref{preserve}.
\end{proof}

\del{

By proposition ?? of \cite{Spiro} we have that curl(curl($X$))$=.$. Hence we see that

\begin{proposition}
If $\alpha$ is Hamiltonian then so is curl($\alpha^\sharp$). In other words, if $\pi_{14}(d\alpha)=0$ then $\pi_{14}(\mathrm{curl}(\alpha^\sharp))=0$.
\end{proposition}

}
As a consequence of Proposition \ref{Poisson is Schouten}, we get the following:

\begin{proposition}
Let $\alpha$ and $\beta$ be in $\Omega^1(M)$ with $\pi_{14}(d\alpha)=0=\pi_{14}(d\beta)$. Then \[\pi_{14}(\mathrm{curl}(\alpha^\sharp)\times\mathrm{curl}(\beta^\sharp))=0.\] Moreover,
\[\mathrm{curl}(\mathrm{curl}(\alpha^\sharp)\times\mathrm{curl}(\beta^\sharp))=[\mathrm{curl}(\alpha^\sharp),\mathrm{curl}(\beta^\sharp)].\]

\end{proposition}
\begin{proof}
By equation (\ref{bracket is cross}) and Lemma \ref{Poisson is Schouten} we see that $d(\mathrm{curl}(\alpha^\sharp)\times\mathrm{curl}(\beta^\sharp))=[X_\alpha,X_\beta]\hk\omega$. Thus, $\mathrm{curl}(\alpha^\sharp)\times\mathrm{curl}(\beta^\sharp)$ is in $\Omega_7^2(M)$, showing that $\pi_{14}(\mathrm{curl}(\alpha^\sharp)\times\mathrm{curl}(\beta^\sharp))=0.$ Moreover, we have that

\begin{align*}
\text{curl}(\text{curl}(\alpha^\sharp)\times\text{curl}(\beta^\sharp))\hk\varphi&=\text{curl}(\{\alpha,\beta\})\hk\varphi&\text{by (\ref{bracket is cross})}\\
&=d(\{\alpha,\beta\})\\
&=[X_\alpha,X_\beta]\hk\varphi&\text{by Lemma \ref{Poisson is Schouten}}\\
&=[\text{curl}(\alpha^\sharp),\text{curl}(\beta^\sharp)]\hk\varphi &\text{by Proposition \ref{Hamiltonian is curl}.}
\end{align*}

The proposition now follows since $\varphi$ is non-degenerate.
\del{
By proposition \ref{Poisson is Schouten}, $X_{\{\alpha,\beta\}}=[X_\alpha,X_\beta]$ up to an element in the kernel of the $n$-plectic form. On a $G_2$ manifold the kernel of $\varphi$ is trivial. Hence, by definition 
\begin{align*}\{\alpha,\beta\}^\sharp&=X_\alpha\times X_\beta\\
&=\text{curl}(\alpha^\sharp)\times\text{curl}(\beta^\sharp)\\
\end{align*}
Since $X_{\{\alpha,\beta\}}=\text{curl}(\text{curl}(\alpha^\sharp)\times\text{curl}(\beta^\sharp))$ we obtain the result.
}
\end{proof}

We now consider the definition of a homotopy moment map in the setting of a $G_2$ manifold. 
The equations defining the components of a homotopy moment map, i.e. (\ref{wmm kernel}), reduce to finding functions $f_1:\g\to\Omega^1(M)$ and $f_2:\Rho_{\g,2}\to C^\infty(M)$ satisfying 

\begin{equation}\label{G2 moment map 1} \pi_{14}(d(f_1(\xi)))=0 \text{ and curl}((f_1(\xi))^\sharp)=V_\xi. \end{equation}
\begin{equation}\label{G2 moment map 2} V_\xi\times V_\eta= -(d(f_2(\xi\wedge\eta)))^\flat.\end{equation}

We finish this section by computing a homotopy moment map in the following set up, extending Example 6.7 of \cite{ms}. 

Consider $\R^7=\R\oplus\C^3$ with standard $3$-form given by \[\varphi=\frac{1}{2}\left(dz_1\wedge dz_2\wedge dz_3+d\o z_1\wedge d\o z_2\wedge d\o z_3\right)-\frac{i}{2}(dz_1\wedge d\o z_1+dz_2\wedge d\o z_2+dz_2\wedge d\o z_3)\wedge dt.\] In terms of $t,x_1,x_2,x_3,y_1,y_2,y_3$ this is \[\varphi=dx_1dx_2dx_3-dx_1dy_2dy_3-dy_1dx_2dy_3-dy_1dy_2dx_3-dtdx_1dy_1-dtdx_2dy_2-dtdx_3dy_3,\]where we have omitted the wedge signs. Equivalently, \[\varphi=\mathrm{Re}(\Omega_3)-dt\wedge \omega_3\] where $\Omega_3$ is the standard holomorphic volume and $\omega_3$ is the standard Kahler form on $\C^3$.  That is, $\Omega_3=dz_1\wedge dz_2\wedge dz_3$ and $\omega_3=\frac{i}{2}(dz_1\wedge d\o z_1+dz_2\wedge d\o z_2+dz_3\wedge d\o z_3)$.

 As in Examples \ref{complex moment map 1} and \ref{complex moment map 2} we consider the standard action by the diagonal maximal torus $T^2\subset SU(3)$ given by $(e^{i\theta},e^{i\eta})\cdot(t,z_1,z_2,z_3)=(t,e^{i\theta}z_1,e^{i\eta}z_2,e^{-i(\theta+\eta)}z_3)$. We have $\mathfrak{t}^2=\R^2$ and that the infinitesimal generators of $(1,0)$ and $(0,1)$ are \[A=\frac{i}{2}\left(z_1\frac{\pd}{\pd z_1}-z_3\frac{\pd}{\pd z_3}-\o z_1\frac{\pd}{\pd\o z_1}+\o z_3\frac{\pd}{\pd \o z_3}\right)\]and \[B=\frac{i}{2}\left(z_2\frac{\pd}{\pd z_2}-z_3\frac{\pd}{\pd z_3}-\o z_2\frac{\pd}{\pd\o z_2}+\o z_3\frac{\pd}{\pd \o z_3}\right)\] respectively.
 
By Example \ref{complex moment map 1} it follows that 
\begin{align*}
A\hk\varphi&=A\hk(\Omega_3-dt\wedge\omega_3)\\
&=\frac{1}{2}d(\text{Im}z_1z_3dz_2)-\frac{1}{4}dt\wedge d(|z_1|^2-|z_3|^2)\\
&=\frac{1}{2}d\left(\text{Im}(z_1z_3dz_2)-\frac{1}{2}(|z_1|^2-|z_3|^2)dt\right).
\end{align*}
Similarly,
\begin{align*}
B\hk\varphi=\frac{1}{2}d\left(\text{Im}(z_1z_2dz_3)-\frac{1}{2}(|z_1|^2-|z_2|^2)dt\right).
\end{align*}
It follows that \[f_1((1,0))=\frac{1}{2}\text{Im}(z_1z_3dz_2)-\frac{1}{4}(|z_1|^2-|z_3|^2)dt\] and \[f_1((0,1))=\frac{1}{2}\text{Im}(z_1z_2dz_3)-\frac{1}{4}(|z_1|^2-|z_2|^2)dt\] give the first component of a homotopy moment map. Plugging in $f_1((1,0))$ and $f_1((0,1))$ into (\ref{pi 14}) shows that \[\pi_{14}(f_1((1,0)))=0=\pi_{14}(f_1((0,1))).\] Moreover, using (\ref{curl coord}) one can directly verify that \[\text{curl}(f_1((1,0)))^\sharp=A\] and \[\text{curl}(f_1((0,1)))^\sharp=B,\] confirming (\ref{G2 moment map 1}).

Using Example \ref{complex moment map 2} it follows that
\begin{align*}
B\hk A\hk\varphi&=B\hk A\hk(\Omega_3-dt\wedge \omega_3)\\
&=B\hk A\hk\Omega_3\\
&=\frac{1}{4}d(\text{Re}(z_1z_2z_3)).
\end{align*}

Thus the second component of the homotopy moment map is given by  \[f_2(A\wedge B)=\frac{1}{4}\text{Re}(z_1z_2z_3),\]in accordance with Example 6.7 of \cite{ms}.
Written out in the coordinates $t$, $x_1$, $x_2$, $x_3$, $y_1$, $y_2$, $y_3,$ the infinitesimal vector fields coming from the torus action are 
\[A=\frac{1}{2}\left(-y_1\frac{\pd}{\pd x_1}+y_3\frac{\pd}{\pd x_3}+x_1\frac{\pd}{\pd y_1}-x_3\frac{\pd}{\pd y_3}\right),\]
\[B=\frac{1}{2}\left(-y_2\frac{\pd}{\pd x_2}+y_3\frac{\pd}{\pd x_3}+x_2\frac{\pd}{\pd y_2}-x_3\frac{\pd}{\pd y_3}\right).\]

Using the metric to identify $1$-forms and vector fields, equation  (\ref{cross}) gives the cross product of $A$ and $B$ to be 
\begin{align*}4(A\times B )&=(y_2y_3-x_2x_3)\frac{\pd}{\pd x_1}+(y_1y_3-x_1x_3)\frac{\pd}{\pd x_2}+(y_1y_2-x_1x_2)\frac{\pd}{\pd x_3}+\\
&+(x_2y_3+x_3y_2)\frac{\pd}{\pd y_1}+(x_3y_1+x_1y_3)\frac{\pd}{\pd y_2}+(x_1y_2+x_2y_1)\frac{\pd}{\pd y_3}\\
&=d(x_1x_2x_3-x_1y_2y_3-y_1x_2y_3-y_1y_2x_3)\\
&=d(\text{Re}(z_1z_2z_3))
\end{align*}
confirming equation (\ref{G2 moment map 2}).  We thus have extended Example 6.7 of \cite{ms} by obtaining a full homotopy moment map for the diagonal torus action on $\R^7$ with the standard torsion-free $G_2$ structure.

\newpage
\section{Concluding Remarks}

This work poses many natural questions for future research. The following are just a few ideas:

\begin{enumerate}
\del{
\item The existence and uniqueness of homotopy moment maps has been studied in \cite{existence 1} and \cite{existence 2}, for example. However, in this paper we were mostly concerned with moment maps restricted to the Lie kernel, i.e. weak moment maps. It would thus be desirable to have results on the existence and uniqueness of these restricted maps.

It is clear that a moment map restricts to a moment map on the Lie kernel. That is, if a collection of maps satisfies equation (\ref{hcmm}), then it satisfies equation (\ref{wmm kernel}). However, are there examples of weak moment maps which do not come from the restriction of a full moment map?  As mentioned throughout the paper, this question is currently being investigated in \cite{future}.
}
\item In our work, we provided a few examples of multi-Hamiltonian systems. What are some examples of other physical or interesting multi-Hamiltonian systems to which this work could be applied?

\item In Section 6.1 we generalized the classical momentum and position functions on the phase space of a manifold to momentum and position forms on the multisymplectic phase space. Since, as discussed in \cite{Marsden}, the classical momentum and position functions play an important role in connecting classical and quantum mechanics, a natural question is if there is an analogous application of our more general theory to quantum mechanics?

\item In symplectic geometry, Proposition \ref{inf cocycle} shows that a moment map $f:\g\to C^\infty(M)$ induces a Lie algebra morphism from $(\g,[\cdot,\cdot])$ to the quotient space $(C^\infty(M)/\mathrm{closed},\{\cdot,\cdot\})$, and if $f$ is equivariant then it is a Lie algebra morphism from $(\g,[\cdot,\cdot])$ to $(C^\infty(M),\{\cdot,\cdot\})$. Moreover, Proposition 4.10 showed that both $\Omega^\bullet_{\mathrm{Ham}}(M)/\mathrm{closed}$ and $\Omega^\bullet_{\mathrm{Ham}}(M)/\mathrm{exact}$ are graded Lie algebras while Proposition 5.9 showed that a weak homotopy moment map is always a graded Lie algebra morphism from $\Rho_\g$ to $\Omega^\bullet_{\mathrm{Ham}}(M)/\mathrm{closed}$. 

\hspace{0.2cm} Hence, a natural question is:

If $(f)$ is an equivariant weak moment map, does it induce a graded Lie algebra morphism from $(\Rho_\g,[\cdot,\cdot])$ to $(\Omega^\bullet_{\mathrm{Ham}}(M)/\mathrm{exact},\{\cdot,\cdot\})$? Does the converse hold?

\item In  our work, we provided a couple of examples of $n$-plectic group actions to which our theory of the existence and uniqueness of moment maps could be applied. There are many other interesting $n$-plectic geometries; see for example \cite{questions}, \cite{me} and \cite{cq}. What does the work done in this thesis say about the existence and uniqueness of moment maps in these setups?

\item Given a weak moment map $(f)$ with $f_k:\Rho_{\g,k}\to\Omega^{n-k}_{\mathrm{Ham}}(M)$, does there exists a full homotopy moment map $(h)$ whose restriction to the Lie kernel is $(f)$? Moreover, in \cite{existence 2} an interpretation of homotopy moment maps was given in terms of equivariant cohomology. Is there an analogous interpretation for weak moment maps?

\end{enumerate}

\del{where in the last equality we have used the standard metric to identify vector fields and $1$-forms. Thus, by (\ref{G2 moment map 2}), we see that we should define the second component of the homotopy co momentum map to be \[f_2(A\wedge B)=x^1x^2x^3-x^1y^2y^3-y^1x^2y^3-y^1y^2x^3.\] A straightforward computation shows that \[x^1x^2x^3-x^1y^2y^3-y^1x^2y^3-y^1y^2x^3=-\frac{1}{2}\text{Re}(z_1z_2z_3)\] so that an alternative way of writing this component of the moment map is \[f_2(A\wedge B)=-\frac{1}{2}\text{Re}(z_1z_2z_3).\] This agrees with the computation done in example 6.7 of \cite{ms}. 
}
\del
{

We now try to find the first component of the homotopy moment map. 

Taking the interior product of $\varphi$ by $B$ gives

\begin{align*}
\frac{2}{i}(B\hk\varphi)&=z_2dz^3dz^1+\frac{i}{2}z_2d\o z^2\wedge dt-z_3dz^1\wedge dz^2-\frac{i}{2}z_3d\o z^3\wedge dt-\\
&-\o z_2d\o z^3\wedge d\o z^1+\frac{i}{2}\o z_2dz_2\wedge dt+\o z_3d\o z^1\wedge d\o z^1-\frac{i}{2}\o z_3dz_3\wedge dt\\
&=-d\left(\frac{1}{4}(|z_2|^2-|z_3|^2)dt-\text{Im}(z_1z_3dz^2)\right)
\end{align*}

It follows that \[f_1(B)=-\frac{1}{4}(|z_2|^2-|z_3|^2)dt-\text{Im}(z_1z_3dz^2).\] Similarly, \[f_1(A)=-\frac{1}{4}(|z_1|^2-|z_3|^2)dt-\text{Im}(z_2z_3dz^1).\]
}
\del{

$A\hk\varphi$
\[\frac{i}{2}\left[z_2dz_3\wedge dz_1-\frac{i}{2}z_2d\o z^2\wedge dt - z_3dz^1\wedge dz^2-\frac{i}{2}z_3d\o z^3dt-\o z_2d\o z^3\wedge d\o z^1+\frac{i}{2}\o z_2dz_2\wedge dt+\o z_3d\o z^1\wedge d\o z^2-\frac{i}{2}\o z_3dz_3\wedge dt\right]\]

A computation then shows that \[A\hk\varphi=-d\left(\frac{1}{4}\left(|z_2|^2-|z_3|^2\right)dt-\text{Im}(z_1z_3dz^2)\right)\]and \[B\hk\varphi=-\left(\frac{1}{4}\left(|z_1|^2-|z_3|^2\right)dt-\text{Im}(z_2z_3dz^1)\right).\]Moreover,  \[B\hk A\hk\varphi=-\frac{1}{2}d(\text{Re}(z_1z_2z_3)).\]

We see that  \[f_1(A)=-\frac{1}{4}\left(|z_2|^2-|z_3|^2\right)dt-\text{Im}(z_1z_3dz^2),\]

\[f_1(B)=-\frac{1}{4}\left(|z_1|^2-|z_3|^2\right)dt-\text{Im}(z_2z_3dz^1),\] and \[f_2(A\wedge B)=-\frac{1}{2}\text{Re}(z_1z_2z_3).\]

We also can see directly that \[d(f_1(A))=A\hk\varphi \text{ and } d(f_1(B))=B\hk\varphi\] and \[A\hk B\hk\varphi=d(f_2(A\wedge B))\] as desired.

}

\del{

\subsection{Symmetries and Conserved Quantities on the Multisymplectic Tangent Bundle}

Given a multisymplectic manifold $(M,\omega)$ we can lift $\omega$ to a multisymplectic form on $TM$. I believe by lifting, we can take a moment map on $M$ to a moment map on $TM$ and symmetries and conserved quantities can also be lifted.

If the form is the volume form then each of the spaces $\Lambda^k(TM)$ have a canonical $n$-plectic structure, the complete lift, namely the pull back of the form on $\Lambda^k(T^\ast M)$. Hence in this case, all symmetries are Lagrangian submanifolds of $T(\Lambda^k(TM))$. 

From $(M,\omega)$ get $(TM,\omega^c)$...Can you take vertical (complete lift) of arbitrary vector field.. you can take vertical lift of any form. Then vertical lift of symmetry is symmetry.

\subsection{Symmetries and Conserved Quantities From Isotropy Subgroups} 
Fix a multi-Hamiltonian system $(M,\omega,H)$.  Let $G$ act on the manifold and fix $p\in\Rho_{\g,k}$. Note that by proposition \ref{Schouten is closed} the adjoint action on $\Lambda^k\g$ preserves $\Rho_{\g,k}$. Let $G_p$ denote the isotropy subgroup for the adjoint action of $G$ on $\Rho_{\g,k}$ and let $\g_p$ denote its Lie algebra. That is, $\g_p=\{\xi\in\g: [\xi,p]=0\}$. Let $G_p^0$ denote the connected component of the identity in $G_p$.

\begin{proposition}
We have that $\xi\wedge p$ is in $\Rho_{\g,k}$ if and only if $\xi$ is in $\g_p$.
\end{proposition}
\begin{proof}
This follows from proposition \ref{Schouten is closed}.
\end{proof}

We consider the multi-Hamiltonian system $(M,V_p\hk\omega, V_p\hk H)$. We summarize the result of section 3.2 of \cite{cq} in the following proposition.

\begin{proposition}
The form $V_p\hk\omega$ is a closed $(n+1-k)$-form on $M$ that is invariant under the action of $G_p$.  Moreover, $V_p\hk H$ is $G_p$ invariant and it is a a Hamiltonian form in $(M,V_p\hk \omega)$ with Hamiltonian vector field $X_H$.

\end{proposition}

Furthermore, in section 3.2 of \cite{cq}, it is shown how a moment map $(f):\Lambda^\bullet\g\to L_\infty(M,\omega)$ induces a moment map for $(M,V_p\hk\omega,V_p\hk X)$. In proposition 3.8 of \cite{cq} it is proven that 

\begin{proposition}
A moment map for the $G_p^0$ action on $(M,V_p\hk\omega)$ is given by $(f^p):\Lambda^\bullet\g_p\to L_\infty(M,V_p\hk\omega)$ with components \[f_j^p:\Lambda^j\g_p\to \Omega^{n-k-j}(M) \ \ \ \ \ \ \ \ \ \ q\mapsto-(-1)^{k(k+1)}f_{j+k}(q\wedge p)\]
\end{proposition}

We now apply Noether's theorem to this setup.

We see that if $q$ is in $\Rho_{\g_p,l}$ then we get a conserved quantity in $(M,V_p\hk\omega,V_p\hk H)$. This symmetry is going to be an element of $\Gamma(\Lambda^{k+j-1}(TM))$ since $f_j(q)$ is in $\Omega^{n+1-k-j}(M)$. We claim that a symmetry is given by $V_q\wedge V_p$ and we prove this as a corollary to something more general.

\begin{proposition}
Let $Y\in\Gamma(\Lambda^l(TM))$ be a local symmetry in $(M,\omega,H)$. Then $Y\wedge V_p$ is a local symmetry in $(M,V_p\hk\omega, V_p\hk H)$. It is also a local symmetry in $(M,\omega,H)$.
\end{proposition}

\begin{proof}
By we need to show that 

\[[Y\wedge V_p,X_H]\hk(V_p\hk \omega)=0\]

We have that $[Y\wedge V_p, X_H]=Y\wedge[V_p,X_H] + [Y,X_H]\wedge V_p$ Hence

\begin{align*}
[Y\wedge V_p,X_H]\hk(V_p\hk \omega)&=Y\wedge[V_p,X_H]\hk(V_p\hk\omega) + [Y,X_H]\wedge V_p(V_p\hk\omega)\\
&=Y\hk V_p\hk[V_p\hk X_H]\hk\omega +V_p\hk V_p\hk[Y,X_H]\hk\omega\\
&=Y\hk V_p\hk[V_p\hk X_H]\hk\omega\\
&=0
\end{align*}
\end{proof}
This confirms the above proposition. We know that $V_p$ is a symmetry, and the infintesimal generator of $f^p_j(q)=f_{j+k}(q\wedge p)$ is $V_q\wedge V_p$.

\subsection{Multisymplectic Legendre Transform}

In \cite{Marsden} a statement of Noether's is given as Corollary 4.2.14. It states that

We demonstrate how this statement can be extended to multisymplectic geometry. In particular, we relate Hamiltonian system given in example ?? $(T^\ast M,G, H)$ to the induced multi-Hamiltonian system $(TM,\mathbb{FL}^\ast\theta_L)$.

\subsection{Transgression of Continuous Symmetries}

Let $M$ be a manifold and let $\Sigma$ be a compact, oriented manifold without boundary. Let $M^\Sigma=C^\infty(\Sigma,M)$, the space of smooth maps from $\Sigma$ to $M$. Given $V\in\Gamma(TM)$, get $V^l\in\Gamma(TM^\Sigma)$, where $V^l$ is the image under the transgression map. A multisymplectic structure on $M$ gives a multisymplectic structure on $M^\Sigma$, by the transgression map.  The transgression map also takes a moment map on $M$ to one on $M^\Sigma$. It's true that a transgression map sends a symmetry to a symmetry and preserves Noether's theorem.

}


\begin{thebibliography}{99}
\vspace{0.2cm}

\bibitem{Marsden} R. Abraham and J. E. Marsden.
{\em Foundations of Mechanics, 2nd edition, } Reading, Massachusetts: Benjamin/Cummings Publishing Company, 1978.


\bibitem{string} J.C. Baez and C.L. Rogers
{\em Categorified symplectic geometry and the string Lie $2$-algebra, } Homology, Homotopy Appl., Volume 12, Number 1 (2010): 221-236.


\bibitem{questions} M. Callies, Y. Fregier, C. L. Rogers, and M. Zambon.
{\em Homotopy moment  maps, } Advances in Mathematics 303 (2016): 954-1043.

\bibitem{dasilva} A. Cannas da Silva
{\em Lectures on symplectic geometry, } Volume 1764 of Lecture Notes in Mathematics. Springer-Verlag, Berlin, 2001.


\bibitem{dropbox} F. Cantrijn, L. A. Ibort, and M. De Leon.
{\em Hamiltonian structures on multisymplectic manifolds, } Rend. Sem. Mat. Univ. Politec. Torino, Volume 54, Issue 3 (1996): 225-236.

\bibitem{poisson} M. Forger, C. Paufler and H. Romer.
{\em The Poisson bracket for Poisson forms in multisymplectic field theory, } Reviews in Mathematical Physics,  Volume 15, Issue 07 (2003): 705-743.

\bibitem{existence 2}Y. Fr\'egier,  C. Laurent-Gengoux, and M. Zambon. 
{\em A cohomological framework for homotopy moment maps, } Journal of Geometry and Physics 97 (2015): 119-132.


\bibitem{covariant} M. Gotay, J. Isenberg, J. Marsden, and R. Montgomery.
{\em  Momentum maps and classical relativistic fields part I: covariant field theory, } Available as arXiv:physics/9801019.

\bibitem{me} J. Herman.
{\em Noether's Theorem Under the Legendre Transform, } Available as arXiv:1409.5837 (2014).

\bibitem{me2} J. Herman.
{\em Noether's Theorem in Multisymplectic Geometry, } Differential Geometry and its Applications 56 (2018): 260 - 294.

\bibitem{future} J. Herman.
{\em Existence and Uniqueness of Weak Homotopy Moment Maps, } Journal of Geometry and Physics 131 (2018): 52 - 65.
\bibitem{j} D. Joyce.
{\em Compact manifolds with special holonomy, } Oxford University Press on Demand, 2000.

\bibitem{Spiro} S. Karigiannis.
{\em Flows of $G_2$ structures, I., } Quarterly Journal of Mathematics 60 (2009), 487-522.

\bibitem{Spiro 2} S. Karigiannis.
{\em Some Notes on $G_2$ and Spin$(7)$ Geometry, } Recent Advances in Geometric Analysis; Advanced Lectures in Mathematics, Vol. 11; International Press, (2010), 129-146.

\bibitem{ms}T. B. Madsen and A. Swann. 
{\em Multi-moment maps, } Advances in Mathematics, Volume 229, Issue 4 (2012): 2287-2309.

\bibitem{MS} T. B. Madsen and A. Swann.
{\em Closed forms and multi-moment maps, } Geometriae Dedicata (2013): 1-28.

\bibitem{marle} C. M. Marle.
{\em The Schouten-Nijenhuis bracket and interior products, } Journal of Geometry and Physics, Volume 23, Issues 3-4 (1997): 350-359.

\bibitem{rogers} C. L. Rogers.
{\em  L-infinity algebras from multisymplectic geometry, } Lett. Math. Phys., Volume 100, Issue 1 (2012): 29-50.

\bibitem{existence 1}L. Ryvkin, and T. Wurzbacher.
{\em Existence and unicity of co-moments in multisymplectic geometry, } Differential Geometry and its Applications 41 (2015): 1-11.


\bibitem{cq} L. Ryvkin, T. Wurzbacher, and M. Zambon.
{\em Conserved quantities on multisymplectic manifolds, } arXiv preprint arXiv:1610.05592 (2016).

\bibitem{kunneth} Charles A. Weibel.
{\em An introduction to homological algebra,} volume 38 of Cambridge Studies in Advanced Mathematics, 1994.


\end{thebibliography}
\end{document}